\documentclass[12pt, a4paper]{amsart}
\usepackage[utf8]{inputenc}
\usepackage[T1]{fontenc}
\usepackage{amsmath,amssymb,amsthm,dsfont}
\usepackage{hyperref}

\usepackage[english]{babel}
\usepackage[all]{xy}
\usepackage[margin = 1 in]{geometry}

\theoremstyle{plain}
\newtheorem{Th}{Theorem}[section]
\newtheorem{Lem}[Th]{Lemma}

\newtheorem{Prop}[Th]{Proposition}
\newtheorem{Def}[Th]{Definition}

\newtheorem{Thintro}{Theorem}

\newcommand{\NN}{\mathbb{N}}

\newcommand{\KK}{\mathbb{K}}
\renewcommand{\AA}{\mathbb{A}}

\newcommand{\ie}{\textit{i.e.},\;}

\newcommand{\ZZ}{\mathbb{Z}}

\newcommand{\PP}{\mathbb{P}}
\newcommand{\EE}{\mathbb{E}}
\newcommand{\CC}{\mathbb{C}}
\newcommand{\RR}{\mathbb{R}}

\newcommand{\Wedge}{\textstyle\bigwedge}
\newcommand{\eps}{\varepsilon}
\newcommand{\sign}{\mathrm{sign}}
\newcommand{\SL}{\mathrm{SL}}
\newcommand{\GL}{\mathrm{GL}}

\title{Convergence to stable laws for products of random matrices}


\author{Axel P\'eneau}
\address{Université de Tours, Université d’Orléans, CNRS, IDP, UMR 7013, Tours, France}

\email{axel.peneau@univ-tours.fr}
\date{\today}

\begin{document}

	\begin{abstract}
		Under reasonable algebraic assumptions and under an infinite second order moment assumption, we show that the logarithm of the norm (log-norm) of a product of random i.i.d. matrices with entries in $\mathbb{R}$ or in any other local field satisfies a generalized Central Limit Theorem (GCLT) in the sense of Paul Lévi. 
		
		The proof is based on a weak law of large number for the difference $\Delta_n$ between the log-norm of the product of the first $n$ matrices and the sum of their log-norms. 
		This weak law of large numbers morally says that $\Delta_n$ behaves like a sum of i.i.d. random variables that have a finite moment of order $2q$ as long as the log-norm of each matrices has a finite moment of order $q$ for a given $q > 0$.
		
		This gain of moment is the central result of the present paper and is based on the construction of pivotal times. Moreover, these results admit a nice higher rank extension when one looks at the full Cartan projection instead of the log-norm.
	\end{abstract}
	\maketitle
	
	\tableofcontents
	
	\section{Introduction}
	
	\subsection{Notations and motivation}
	Through the present paper, $E$ denotes a Euclidean, Hermitian or ultra-metric space of dimension $d \ge 2$ over a locally compact field denoted by $\KK$.
	Let $\mathrm{GL}(E)$ be the space of linear endomorphism of $E$ and let $\mathrm{SL}(E)$ be the space of endomorphism of $E$ that have determinant $1$.
	
	A semi-group\footnote{This simply means that for all pair $g, g' \in \Gamma$, the composition $g g'$ is also in $\Gamma$.} $\Gamma < \mathrm{GL}(E)$ is said to be \textbf{strongly irreducible} if there is no non-trivial $\Gamma$-invariant subset $\{0\} \subsetneq A \subsetneq E$ that is the union of finitely many subspaces.
	We say that $\Gamma$ is \textbf{proximal} if there exists an element $g \in \Gamma$ such that $\lambda_1(g) > \lambda_2(g)$, where $\lambda_1(g) \ge \lambda_2(g) \ge \dots \ge \lambda_d(g)$ denote the logarithms of the moduli of the eigenvalues of $g$ counted with multiplicity.
	In other words $\Gamma$ is proximal if and only if there exists a rank $1$ projection in the closure of $\KK \Gamma$.
	
	Let $(\gamma_k)_{k \ge 0}$ be a sequence of independent and identically distributed random matrices and write $\nu$ for the law of $\gamma_0$, as a short notation, we simply write $(\gamma_k)_{k \ge 0} \sim \nu^{\otimes\NN}$.  
	We are interested in the random walk $(\overline{\gamma}_n)_{n \ge 0}$, which is the sequence of left-to right partial products of $(\gamma_k)$, \ie $\overline\gamma_n  = \gamma_0 \cdots \gamma_{n-1}$ for all $n \ge 0$ and therefore $\overline\gamma_n \sim \nu^{*n}$.
	More specifically, we are interested in the limit behaviour of the law of the log-norm $\log\|\overline\gamma_n\|$ when the support of $\sum \nu^{*n}$ \textemdash\, denoted by $\Gamma_\nu$ and which is the intersection of all closed set that contain all the $\overline{\gamma}_n$'s with probability $1$ \textemdash\,  is strongly irreducible and proximal.
	For all matrix $g$, we write $\kappa(g) = \log\|g\| = \max_{x \in E \setminus\{0\}}\log(\|gx\|/\|x\|) \in [-\infty, + \infty)$, for the log-norm of $g$.
	
	The motivation for the present paper is the proof of the following result, which extends the General Central Limit Theorem to our model of non-commuting random walk.

	\begin{Thintro}[Convergence to stable law with infinite variance]\label{th:loi-stable}
		Let $\nu$ be a probability measure over $\mathrm{SL}(E)$ and let $(\gamma_n)_{n \ge 0} \sim \nu^{\otimes\NN}$.
		Assume that $\Gamma_\nu$ is proximal, that $\int\kappa^2 d\nu = \EE(\kappa(\gamma_0)^2) = + \infty$ and that $\kappa_* \nu$ is in the domain of attraction of a non-degenerate (\ie not a Dirac) probability measure $\mathcal{L}$, in the sense that there exist two non-random sequences $(a_n)\in (0, + \infty)^\NN$ and $(b_n) \in \RR^\NN$ such that for all bounded and continuous function $f$, we have:
		\begin{equation}\label{dom-at}
			\lim_{n \to \infty} \EE\left(f\left(\frac{\sum_{k = 0}^{n-1} \kappa(\gamma_k) - b_n}{a_n}\right)\right) = \int_{-\infty}^{+\infty} f d\mathcal{L}.
		\end{equation}
		Then for all such sequences $(a_n)$ and $(b_n)$, there exists a constant $b \ge 0$ such that for all bounded and continuous function $f$, we have:
		\begin{equation}\label{weak-limit-alpha}
			\lim_{n \to \infty} \EE\left(f\left(\frac{\kappa(\overline\gamma_n) - b_n + n b}{a_n}\right)\right) = \int_{-\infty}^{+\infty} f d\mathcal{L}
		\end{equation}
	\end{Thintro}
	
	If one instead assumes that $\int\kappa^2 d\nu < + \infty$ and that $\Gamma_\nu$ is strongly irreducible and proximal, Benoist and Quint have shown in \cite[Theorem~1.1]{CLT16} that:
	\begin{equation}
		\exists \lambda > 0, \exists a \ge 0, \forall f, \; \lim_{n \to \infty} \EE\left(f\left(\frac{\kappa(\overline\gamma_n) - n \lambda}{\sqrt{n}}\right)\right) = \int_{-\infty}^{+\infty} \frac{f(ax) e^{-x^2/2}}{\sqrt{2\pi}} dx.
	\end{equation}
	Where $\lambda$ is the almost sure limit of $\kappa(\overline\gamma_n)/n$, which is well defined as soon as $\int \kappa d\nu < + \infty$ by the works Furstenberg and Kesten in \cite{porm}, and proven to be positive, when $\Gamma_\nu$ is strongly irreducible proximal, by the joint work of Guivarc'h, Raugi and Lepage \cite{GR1985}.
	
	We remind that \cite[Theorem~8.1, page~298]{Feller_68} states that the non-degenerate probability measure $\kappa_*\nu$, supported on $\RR_{\ge 0}$, is in the domain of attraction of a non-degenerate probability measure $\mathcal{L}$, in the sense of \eqref{dom-at} if and only if there exists a constant $0 < \alpha \le 2$ such that we have:
	\begin{equation}\label{slow-var-var}
		\forall r > 0,\; \lim_{t \to  +\infty} \frac{\int_{0}^{rt} u^2 d\kappa_*\nu(u)}{\int_{0}^{t} u^2 d\kappa_*\nu(u)} = r^{2 - \alpha}.
	\end{equation}
	In this case, by the works of Lévi and others, the law $\mathcal{L}$ is determined, up to an affine transformation, by $\alpha$-alone. 
	
	The proof of Theorem \ref{th:loi-stable} simply consists in showing that $\frac{\log\|\gamma_0 \cdots \gamma_n\|- \sum_{k = 0}^{n-1} \log\|\gamma_k\| + nb}{a_n}$ converges in probability to $0$.
	For that, we use Theorem \ref{th:Delta-unif}, which is the breakthrough of the present article.
	Theorem \ref{th:loi-stable} is a remarkable applications of the method developed in the present paper. 
	Yet the present paper is not really about the domain of attraction of stable distributions.
	Indeed, we only use the black-box result stated as Lemma \ref{lem:stable-law}.
	
	We denote by $\kappa$ the map $g \mapsto \log\|g\|$ (where $\|g\| = \max_{x \neq 0} \|gx\|/\|x\|$ is the operator norm of $g$) and denote by $N$ the map $g \mapsto \kappa(g) + \kappa(g^{-1})$. 
	Let us remind that that $\kappa$ is sub-additive (in the sense that $\kappa(gh) \le \kappa(g) + \kappa(h)$ for all pair $g,h$) and therefore $N$ also is.
	Moreover $N$ is non-negative and on $\mathrm{SL}(E)$, we have $0 \le \kappa \le N \le \kappa d$.
	
	Given $(\gamma_n)$ i.i.d., if we assume only that $\EE(\kappa(\gamma_0)) < +\infty$, we know by the work of Furstenberg and Kesten \cite{porm} that the sequence $\kappa(\overline\gamma_{n})/n$ converges almost surely to the constant $\lambda_1(\nu) := \lim_n \EE(\kappa(\overline\gamma_n))/n \in [-\infty, +\infty)$, called first Lyapunov exponent of $\nu$.
	In \cite[Theorem 1.1]{CLT16}, Benoist and Quint show that if $\kappa_* \nu$ and $N_* \nu$ both\footnote{On $\mathrm{SL}_d$, we have $0 \le \kappa \le N \le (d-1) \kappa$ so the two moment conditions are equivalent. That being said, in the setting of \cite[Theorem 1.1]{CLT16}, the random walk is on $\mathrm{GL}(E)$.} have a finite moment of order $2$ and that $\nu$ is strongly irreducible and proximal, then $\kappa_*\nu^{*n}$ satisfies a Central Limit Theorem.
	
	Using the tools developed for the proof of Theorem \ref{th:loi-stable} we can (quite unexpectedly) lower the moment assumption on $N$ to a first moment assumption.
	
	\begin{Thintro}[Central limit Theorem with first moment assumption on the inverse]\label{th:clt}
		Let $\nu$ be a strongly irreducible and proximal probability measure over $\mathrm{GL}(E)$ and let $(\gamma_n)_{n \ge 0} \sim \nu^{\otimes\NN}$.
		Assume that $\EE(\kappa(\gamma_0)^2) < +\infty$ and $\EE(N(\gamma_0)) < + \infty$. 
		Then the law of $\frac{\kappa(\overline{\gamma}_n) - n \lambda_1(\nu)}{\sqrt{n}}$ converges to a centred Gaussian law.
		It means that there exists a constant $a \ge 0$, such that for all bounded continuous function $f$, we have:
		\begin{equation}\label{weak-clt-1}
			\EE\left(f\left(\frac{\kappa(\overline{\gamma}_n) - n \lambda_1(\nu)}{\sqrt{n}}\right)\right) \to \int_{-\infty}^{+ \infty}  \frac{f(ax) e^{-x^2/2}}{\sqrt{2\pi}} dx.
		\end{equation}
		Moreover, we have $\lim_n\mathrm{Var}(\kappa(\overline{\gamma}_n)) / n = a^2$.
	\end{Thintro}
	
	The fact that $\lim_n\mathrm{Var}(\kappa(\overline{\gamma}_n)) / n = a^2$ in Theorem \ref{th:clt} is not a mere consequence of \eqref{weak-clt-1}, the combination of the two is equivalent to saying that \eqref{weak-clt-1} hold for all function $f$ that is sub-quadratic and continuous, in the sense that $(f(x)/(x^2+1))_{x \in \RR}$ is bounded and continuous.
	The continuity condition on $f$ can also be relaxed to a piecewise continuity condition using the fact that the Gaussian distribution is absolutely continuous (with respect to the Lebesgue measure) but the is not really the kind of problems we are interested in in the present article.
	Another equivalent reformulation of the conclusion of Theorem \ref{th:clt} is that the quadratic Wasserstein distance between the law of $\frac{\kappa(\overline{\gamma}_n) - n \lambda_1(\nu)}{\sqrt{n}}$ and the centred Gaussian distribution of variance $a^2$ has limit $0$ \ie there exists a coupling of $(\gamma_n)_n \sim \nu^{\otimes\NN}$ with a sequence $(y_n)_n$ of centred Gaussian random variables of variance $a^2$ (not independent) and such that:
	\begin{equation}\label{wasserstein-clt}
		\EE\left(\left|\frac{\kappa(\overline{\gamma}_n) - n \lambda_1(\nu)}{\sqrt{n}} - y_n\right|^2\right) \to 0.
	\end{equation}
	
	To put emphasis on the non obviousness of the fact that $\lim_n\mathrm{Var}(\kappa(\overline{\gamma}_n)) / n = a^2$, let us look at the central limit Theorem for the coefficients.
	In \cite{moi}, we have seen that for all $v \in E\setminus\{0\}$ and all $f \in E^* \setminus\{0\}$, the distribution of $\frac{\kappa(\overline{\gamma}_n) - \log|f\overline\gamma_n v|}{\sqrt{n}}$ converges to the Dirac at $0$.
	So in the setting of Theorem \ref{th:clt}, the distribution of $\frac{\log|f\overline\gamma_n v| - n \lambda_1(\nu)}{\sqrt{n}}$ converges in distribution to the centred Gaussian distribution of variance $a$.
	It means that for all bounded and continuous function $f$, we have:
	\begin{equation}\label{weak-clt-coef}
		\EE\left(f\left(\frac{\log |f\overline\gamma_n v| - n \lambda_1(\nu)}{\sqrt{n}}\right)\right) \to \int_{-\infty}^{+ \infty}  \frac{f(ax) e^{-x^2/2}}{\sqrt{2\pi}} dx.
	\end{equation}
	Since $\PP(|f\overline\gamma_n v| = 0)$ is not necessarily zero, we use the convention $f(-\infty) = 0$, for the left member of $\eqref{weak-clt-coef}$ to be well defined. 
	Even when $\PP(|f\overline\gamma_n v| = 0)$ for all $n$, the sequence $\mathrm{Var}(\log|f\overline\gamma_n v|) / n$ may be stationary to $+ \infty$.
	
	To illustrate this, let $(x^k_n) \in \{0,1\}^{\NN^2}$ be independent and uniformly distributed, let $E = \RR^2$ and for all $n \ge 0$, let $\gamma_n = R_{\theta_n} \mathrm{diag}(2,1)$, where $\theta_n = \sum_{k =0}^{\infty} 2^{-k!} x^k_n$ and $R_\theta$ denotes the rotation of angle $\theta$. 
	If we take $f = \begin{pmatrix}0& 1\end{pmatrix}$ and $v = \binom{1}{0}$ then for all $n, k$, we have $\PP(|f \overline{\gamma}_n v| \le n2^{-k!}) \ge \PP(\forall i < n, j < k, x^j_i = 0) = 2^{- n k}$ so if we sum over all values of $k$, we get $\EE(\log_-|f \overline{\gamma}_n v|) \ge \sum_{k = 0}^{+\infty} 2^{-n k}k!\log(2) - \log(n) = + \infty$, even though $|f \overline{\gamma}_n v| \neq 0$.
	
	In Theorem \ref{th:loi-stable} we will moreover show that \eqref{weak-limit-alpha} holds for all function $f$ such that $f(x)/(|x|^q + 1)$ is bounded and continuous for all $q < \alpha$, where $\alpha$ is the stability parameter of $\mathcal{L}$, \ie the only constant such that \eqref{slow-var-var} is satisfied for $\mu = \kappa_*\nu$. 
	Again, this is equivalent to saying that there exists a sequence of random variables $(y_n)$ that all have law $\mathcal{L}$ and such that for all $0 < q < \alpha$, we have:
	\begin{equation}\label{wasserstein-loi-stable}
		\EE\left(\left|\frac{\kappa(\overline{\gamma}_n) - b_n + n b}{a_n} - y_n\right|^q\right) \to 0.
	\end{equation}

	\subsection{Probabilistic estimate for the norm-cancellation function}

	Let $(\gamma_n)_{n \ge 0} \sim \nu^{\otimes\NN}$ be a random sequence\footnote{The notation $(\gamma_n)_{n \ge 0} \sim \nu^{\otimes\NN}$, means that $(\gamma_n)_{n \ge 0}$ is a measurable map from a standard probability space $(\Omega, \PP)$ to $\mathrm{GL}(E)^\NN$ and for all family of bounded measurable functions $(\psi_i)_{i \in I}$ with $I \subset \NN$ finite, we have $\EE(\prod_{i \in I}\psi_i(\gamma_i)) = \prod_{i \in I}\int\psi_i d\nu$, where $\EE$ denotes the expectation form associated to $\PP$. In this case that the random variables $(\gamma_n)_{n \ge 0}$ are independent and identically distributed (abbreviated to i.i.d.) of law $\nu$.} and write $\overline{\gamma}_n = \gamma_0 \cdots \gamma_{n-1}$ for all $n \ge 0$.
	 
	We say that the sequence $(\overline\gamma_n)_{n \ge 0}$ is a right\footnote{Some authors prefer to study the left random walk $(\gamma_{n-1}\cdots\gamma_0)_{n \ge 0}$ as it is compatible with the left group action $\mathrm{GL}(E)\curvearrowright E$. Since the methods developed in the present work are not based on the study of this group action, we make the more practical choice of considering the right random walk.} random walk with step of law $\nu$, that way, we have $\overline{\gamma}_n \sim \nu{*n}$ for all $n \ge 0$ (with the convention $\overline\gamma_0 = \mathrm{Id}(E)$). 
	We say that a random walk is strongly irreducible and proximal if its step law is.

	The proof of Theorems \ref{th:loi-stable} and \ref{th:clt} do not rely on an Ergodic theoretic version of Feller's results but on probabilistic estimates for the quantity:
	\begin{equation*}
		\Delta\kappa(\gamma_0, \dots, \gamma_{n-1}) := \kappa(\gamma_0 \cdots \gamma_{n-1}) - \kappa(\gamma_0) - \cdots - \kappa(\gamma_{n-1}).
	\end{equation*}
	By sub-additivity of $\kappa$, we know that $\Delta\kappa(\widetilde\gamma_{0,n})$ is a non-positive random variable for all $n$ so it has a well defined expectation in $[0, + \infty]$.
	We say that $\Delta\kappa(\gamma_0, \cdots, \gamma_{n-1})$ is the total norm-cancellation of the word $(\gamma_0, \cdots, \gamma_{n-1})$.
	We write $\delta(\nu) = \lim_n \frac{\EE(\Delta_n)}{n} \in [0, + \infty]$, the limit exists by sub-additivity of $\kappa$ (indeed, $\Delta_{m+n} \le \Delta_{m} + \Delta_{n}$ for all $m,n$) we say that $\delta(\nu)$ is the average norm-cancellation of $\nu$, it may be $-\infty$.
	Let us note that when $\EE(\kappa(\gamma_0)) < + \infty$, we have $\lambda_1(\nu) = \EE(\kappa(\gamma_0)) + \delta(\nu)$.
	Let us introduce a more general notation.
	
	\begin{Def}
		Let $\Gamma$ be a semi-group. 
		We write $\widetilde\Gamma = \bigsqcup_{n \ge 0} \Gamma^n$ for the semi-group of words on the alphabet $\Gamma$, endowed with the concatenation product $\odot$.
		Let $(V, +)$ be an Abelian group and let $\kappa : \Gamma \to V$, we write $\Delta\kappa : \widetilde\Gamma \to V : (g_0, \dots, g_{l-1}) \mapsto \kappa(g_0\cdots g_{l-1}) - \sum_{k = 0}^{l-1} \kappa(g_k)$.
	\end{Def}
	
	Given a sequence $(\gamma_n)_{n \ge 0}$ in a semi-group, and given two integers $0 \le m \le n$, we write $\widetilde\gamma_{m,n}$ as short for the word $(\gamma_m, \cdots, \gamma_{n-1})$ of length $n-m$ and write $\gamma_{m,n}$ as short for the product $\gamma_m \cdots \gamma_{n-1}$. 
	That way $\gamma_n = \gamma_{n, n+1}$ for all $n$.
	
	\begin{Prop}
		Let $\nu$ be a probability distribution over a semi-group $\Gamma$, let $(\gamma_n) \sim \nu^{\otimes\NN}$ and let $\kappa : \Gamma \to \RR$ be sub-additive.
		There exist a constant $\delta(\nu) \in [- \infty, 0]$, such that we have almost surely:
		\begin{equation}\label{delta-nu}
			\lim_{n \to  + \infty} \frac{\Delta{\kappa}(\widetilde{\gamma}_{0,n})}{n} = \lim_{n \to  + \infty} \frac{\EE(\Delta{\kappa}(\widetilde{\gamma}_{0,n}))}{n} = \delta(\nu).
		\end{equation}
	\end{Prop}
	
	\begin{proof}
		For all $0 \le n \le m$, we have:
		\begin{equation}\notag
			\Delta{\kappa}(\widetilde{\gamma}_{0, m}) = \Delta{\kappa}(\widetilde{\gamma}_{0,n}) + \Delta{\kappa}(\widetilde{\gamma}_{n, m}) + \Delta\kappa(\gamma_{0,n}, \gamma_{n, m}).
		\end{equation}
		Moreover, $\Delta\kappa \le 0$ by sub-additivity of $\kappa$ so $\Delta{\kappa}(\widetilde{\gamma}_{0, m}) \le \Delta{\kappa}(\widetilde{\gamma}_{0,n}) + \Delta{\kappa}(\widetilde{\gamma}_{n, m})$ for all $n \le m$.
		Therefore, we may apply Kingman's ergodic Theorem \cite{K68} to the sequence of non-positive functions $f_n : \mathrm{GL}(E)^\NN \to \RR_{\le 0}; (\gamma_n)_{n \ge 0} \mapsto \Delta{\kappa}(\widetilde{\gamma}_{0,n})$ and to the shift-ergodic measure $\nu^{\otimes\NN}$.
	\end{proof}
	
	Note that in fact, we may define a quantity $\delta(\mu) \in [- \infty, 0]$ for all shift-ergodic measure $\mu$ on $\mathrm{GL}(E)^\NN$, \ie such that given $(\gamma_n)_{n \ge 0}\sim \mu$, we have $(\gamma_{n + 1})_{n \ge 0} \sim \mu$ and all asymptotic events have probability $0$ or $1$.
	In the present paper, we are interested in controlling the distribution of $\Delta{\kappa}(\widetilde{\gamma}_{0,n})$ for all $n$.
	For that we use the following key result.
	
	\begin{Th}[Probabilistic bound on the norm-cancellation]\label{th:Delta-unif}
		Let $\nu$ be a strongly irreducible and proximal probability measure over $\mathrm{GL}(E)$ and let $(\gamma_n) \sim \nu^{\otimes\NN}$.
		There exist constants $C, \beta > 0$ such that for all $0 \le i < j < k$ and for all $t > 0$, we have:
		\begin{equation}\label{Delta-unif}
			\PP(|\Delta\kappa(\gamma_{i,j}, \gamma_{j,k})| > t) \le \sum_{k = 0}^{\infty} C e^{-\beta k} \PP(N(\gamma_0) > t/k)^2. 
		\end{equation}
	\end{Th}
	
	The great improvement from the results we have seen in \cite{moi} is the fact that the probability, on the right hand member of \eqref{Delta-unif}, is squared. 
	A direct consequence of \eqref{Delta-unif} is that for all $q > 0$, we have:
	\begin{equation*}
		\EE(|\Delta\kappa(\gamma_{i,j}, \gamma_{j,k})|^{2q}) \le 2\EE(N(\gamma_0)^q)^2\sum_{k = 0}^{\infty} C e^{-\beta k} k^{2q}. 
	\end{equation*}
	This follows from an integration by parts detailed in the proof of Lemma \ref{lem:ellq-almost}.
	Moreover, by construction, the constant $C$ and $\beta$ are not sensitive to the tail of $\nu$ in the sense that a fitting pair $C, \beta$ is given by a function of $\nu$ that is continuous for the weak-$*$ topology on the open set of strongly irreducible and proximal measures.
	More details on this observation are given in Section \ref{section:pivot}
	
	In the case of $N(\gamma_0)$ being bounded, \eqref{Delta-unif} is equivalent to saying that there exist constants $C, \beta > 0$ such that $\PP(|\Delta\kappa(\gamma_{i,j}, \gamma_{j,k})| > t) \le Ce^{-\beta t}$.
	This is a well known result proven by Guivarc'h and Lepage under the assumption that $N(\gamma_0)$ has a finite exponential moment so we shall assume that $N(\gamma_0)$ in unbounded in the proof to avoid having to treat the bounded case separately. 
	Note that if we only assume that $N(\gamma_0)$ has a finite exponential moment, then \eqref{Delta-unif} does not imply that $|\Delta\kappa(\gamma_{i,j}, \gamma_{j,k})|$ has a uniformly bounded exponential moment so formula \eqref{Delta-unif} shall only be used when the tail of $N_* \nu$ decays slower than any exponential.
	
	Theorem \ref{th:Delta-unif} is derived from the following extension of \cite[Theorem~1.6]{moi}.
	The only difference being that point \eqref{pivot:indep} is given as a qualitative global independence result, while point (7) in \cite[Theorem~1.6]{moi} is a quantitative result on the conditional distribution of the $\gamma_k$'s.
	Indeed, given two matrices $g$ and $h$, and $\eps > 0$, the notation $g \AA^\eps h$ in \cite{moi} means that $|\Delta\kappa(g,h)| \le |\log(\eps)|$ with the present notations.
	
	Given a sequence  of integers $(p_n)_{n \ge 0}$, we write $\overline{p}_n = \sum_{k < n} p_n$ for all $n$ and given a sequence $(\gamma_n)$ in a semi-group, we write $(\gamma^p_n)_{n \ge 0} \in \Gamma^\NN$ for the sequence defined by $\gamma^p_{n} = \gamma_{\overline{p}_n, \overline{p}_{n+1}}$ and $(\widetilde\gamma^p_n)_{n \ge 0} \in\widetilde\Gamma^\NN$ for the sequence characterized by the fact that $\widetilde\gamma^p_{n} = \widetilde\gamma_{\overline{p}_{n}, \overline{p}_{n+1}}$ for all $n$.
	That way, we have $\widetilde\gamma^p_{i,j} = \widetilde\gamma_{\overline{p}_{i}, \overline{p}_{j}}$ and $\gamma^p_{i,j} = \gamma_{\overline{p}_{i}, \overline{p}_{j}}$ for all $i <j$.

	\begin{Th}[Pivoting technique]\label{th:pivot}
		Let $\nu$ be a strongly irreducible and proximal probability measure over $\mathrm{GL}(E)$.
		There exist a compact $K \subset \mathrm{GL}(E)$, a constant $C \ge 0$, an integer $m > 0$ and, defined on the same probability space, a sequence $(\gamma_n) \sim \nu^{\otimes\NN}$, an i.i.d. sequence $(g_n)_{n \ge 0} \in \mathrm{GL}(E)^{\NN}$ and a sequence of integers $(p_n)_{n \ge 0}$ such that the following points hold:
		\begin{enumerate}
			\item For all $n$, we have $\gamma_n = g_n$ or $\gamma_n \in K$.\label{pivot:gammainK}
			\item For all $k$, we have $\widetilde\gamma^p_{2k + 1} \in K^m$ almost surely and the conditional distribution of $\left(\widetilde\gamma^p_{n + 2k + 2}\right)_{n \ge 0}$ with respect to $\left(\widetilde\gamma^p_{n}\right)_{n \le 2k}$ is almost surely constant and does not depend on $k$ either.\label{pivot:markov}
			\item all the $p_k$'s have a finite exponential moment\label{pivot:exp}
			\item The data of $(p_n)_{n \ge 0}$ is independent of $(g_n)_{n \ge 0}$.\label{pivot:indep}
			\item For all $0 \le i < j < k$, we have almost surely:\label{pivot:almost-ad}
			\begin{equation*}
				|\Delta\kappa(\gamma^p_{i,j}, \gamma^p_{j,k})| \le C.
			\end{equation*}
			\item For all integer $k$ and for all $h \in \mathrm{GL}(E)$ that is given by a function of $\left(\widetilde\gamma^p_{n}\right)_{n \neq 2k + 1}$ and $(g_n)_n$, we have:\label{pivot:schottky}
			\begin{gather*}
				\PP\left(\forall j \le 2k + 1, |\Delta\kappa(\gamma^p_{2k + 1}, h)| \le C \right) \ge 3/4. \\
				\PP\left(\forall j \ge 2k + 1, |\Delta\kappa(h, \gamma^p_{2k + 1})| \le C \right) \ge 3/4.
			\end{gather*}
			\item For all $0 < i < j < k$ such that $i$ and $k-1$ are odd and for all $g,h \in \mathrm{GL}(E)$ such that $|\Delta\kappa(g,\gamma^p_{i})| \le C$ and $|\Delta\kappa(\gamma^p_{k-1}, h)| \le C$, we have:\label{pivot:align}
			\begin{equation*}
				|\Delta\kappa(g\gamma^p_{i,j}, \gamma^p_{j,k}h)| \le C.
			\end{equation*}
		\end{enumerate}
	\end{Th}
	
	The proof of Theorem \ref{th:pivot} is essentially the same as the proof of \cite[Theorem~1.6]{moi} since both are direct consequences of \cite[Theorem~4.7]{moi} applied to $(\mathrm{GL}(E), \nu)$, using the formalism of \cite[§2\&3]{moi}. 
	A detailed proof is given in Section \ref{section:pivot}.
	
	In Theorem \ref{th:pivot}, we say that there exists two independent sequences $(g_n)$ and $(p_n)$ such that for all $n$, we have $g_n = \gamma_n$ when $\gamma_n \notin K$. 
	Since the value of $g_n$ does not appear anywhere else in the conclusions of the Theorem, we could have simply said that relatively to the data of the set $I = \{n \ge 0\,|\, \gamma_n \notin K\}$ the data of $(\gamma_n)_{n \in I}$ is independent of the data of $(p_n)_{n \ge 0}$, which means that for all event $A$ that depends on $I$ and $(\gamma_n)_{n \in I}$ and all even $B$ that depends on $I$ and $(p_n)_{n \ge 0}$, we have $\PP(A \cap B\,|\, I) = \PP(A\,|\, I) \PP(B\,|\, I)$ almost surely in $I$.

	\subsection{The Generalized Central Limit Theorem (GCLT)}
	
	Let us detail the framework behind the proof of Theorem \ref{th:loi-stable} and state related results.
	Let us first remind the statement of the classical GCLT for non-negative random variables.	
	
	We say that a sequence of $\sigma$-finite measures $(\mu_n)_n$ on a topological space $X$ converges in the weak-$*$ topology to a $\sigma$-finite measure $\mathcal{L}$ when for all continuous and compactly supported function $f : X \to \RR$ (\ie there exists a compact $K \subset X$ such that $f(x) = 0$ for all $x \in X \setminus K$), we have:
	\begin{equation}\label{cv-int}
		\lim_n \int_{X} fd\mu_n(x) = \int_X f d\mathcal{L}(x),
	\end{equation}
	in this case, we write $\mu_n \rightharpoonup \mathcal{L}$. 
	If moreover $\mu_n$ is a probability measure for all $n$ and $\mathcal{L}$ also is, then \eqref{cv-int} holds for all bounded and continuous function.
	On $X  = \RR$, given a sequence of probability measures $(\mu_n)_n$  and a probability measure $\mathcal{L}$, Lévi's Theorem tells us that $\mu_n \rightharpoonup \mathcal{L}$ if and only if we have $\lim_n \int_{-\infty}^{+\infty} e^{i\theta x} d\mu_n(x) = \lim_n \int_{-\infty}^{+\infty} e^{i\theta x} d\mathcal{L}(x)$.
	For all measure $\mu$ on $\RR$, we write $\widehat{\mu} : \theta \mapsto \int_{-\infty}^{+\infty} e^{i\theta x} d\mu(x)$ for the Fourier transform of $\mu$. This is a continuous maps from $\RR$ to the unit disc in $\CC$.
	
	We say that a probability measure $\mu$ on $\RR$ is in the domain of attraction of a non-degenerate probability measure $\mathcal{L}$ when there exist sequences $(a_n) \in (0, + \infty)^\NN$ and $(b_n) \in \RR^\NN$ such that $(x \mapsto (x-b_n)/a_n)_* \mu \rightharpoonup \mathcal{L}$.
	By Lévi's Theorem, this is equivalent to saying that for all $\theta \in \RR$, we have $\widehat{\mu}(\theta / a_n)^ne^{-i\theta b_n / a_n} \to \widehat{\mathcal{L}}(\theta)$.
	
	In this case, it has been shown that $\mathcal{L}$ is of type $\mathcal{L}^{\alpha, \beta}_{a,b}$, for some constants $0 < \alpha \le 2$, $-1 \le \beta \le 1$, $a > 0$ and $b \in \RR$, where $\mathcal{L}^{\alpha, \beta}_{a,b}$ is characterized by its Fourier transform $\widehat{\mathcal{L}}^\alpha_{a,b}$, defined as:
	\begin{equation}\label{Fourier}
		\widehat{\mathcal{L}}^{\alpha, \beta}_{a,b}: \theta \mapsto \begin{cases}
			\exp\left(i\theta b - C_\alpha |a\theta|^\alpha (1- i \beta\sign(\theta) \tan(\pi\alpha / 2)) \right) & \text{for } \alpha \neq 1, \\
			\exp\left(i\theta b - C_1 |a\theta| - 2 i \beta a \theta \log|a\theta| / \pi) \right) & \text{for } \alpha = 1.
		\end{cases}
	\end{equation}
	Where $\sign$ is the sign function that takes value $0$ at $0$, $+1$ on $\RR_{> 0}$ and $-1$ on $\RR_{< 0}$ and $C_\alpha = \int_{0}^{+ \infty} x^\alpha\sin(x)dx$ for $x < 2$, that way $\lim_{t \to + \infty} t^\alpha\mathcal{L}^\alpha_{a,b}(t, + \infty) = a^{\alpha}$ and $C_2 = 1$, that way $\mathcal{L}^2_{a,b}$ is the Gaussian distribution of variance $a^2$ and mean $b$.
	
	Note that the set $\{\widehat{\mathcal{L}}^{\alpha, \beta}_{a,b}; a> 0, b\in \RR\}$ is closed under multiplication, which means that the space $\{\mathcal{L}^{\alpha, \beta}_{a,b}; a> 0, b\in \RR\}$ is closed under convolution.
	The terminology $\alpha$-stable comes from the fact that for all $a_0, a_1, b_0, b1$, we have $\mathcal{L}^{\alpha, \beta}_{a_0,b_0} * \mathcal{L}^{\alpha, \beta}_{a_1,b_1} = \mathcal{L}^{\alpha, \beta}_{a_2,b2}$ whit $a_2^\alpha = a_0^\alpha + a_2^\alpha$ and $b_2 = b_0 + b_1 + b'$ with $b' = 0$ for $\alpha \neq 1$ and $b' = \frac{2\beta}{\pi} (a_0\log(a_0) + a_1 \log(a_1) - a_2\log(a_2))$.
	
	For $\alpha < 2$, saying that $\widehat{\mu}(\theta / a_n)^ne^{-i\theta b_n / a_n} \to \widehat{\mathcal{L}}(\theta)$ for all $\theta$ is equivalent to saying that for all $t > 0$, we have:
	\begin{equation}
		\lim_{n \to + \infty} t^\alpha\mu(t a_n, + \infty) = \frac{1+\beta}{2}a^\alpha\quad \text{and}\quad \lim_{n \to + \infty} t^\alpha\mu(-\infty, -t a_n) = \frac{1-\beta}{2} a^\alpha 
	\end{equation}
	and that $b_n / a_n \to b$ for $\alpha < 1$, that $(b_n - n\int_{-\infty}^{+ \infty} x d\mu(x))/a_n \to b$ when $\alpha > 1$ and that $(b_n - n\int_{-a_n}^{+a_n} x d\mu(x))/a_n \to b$ when $\alpha =1$.
	For $\alpha = 2$, $\mathcal{L}^{\alpha, \beta}_{a,b}$ is simply the normal distribution of variance $a^2$ and mean $b$ (it does not depend on $\beta$ because $\tan(\pi\alpha / 2) = 0$) and saying that $\widehat{\mu}(\theta / a_n)^ne^{-i\theta b_n / a_n} \to \widehat{\mathcal{L}}(\theta)$ for all $\theta \in \RR$ is equivalent to saying that for all $t > 0$, we have:
	\begin{equation}
		\frac{n}{a_n^2}\left(\int_{-a_n t}^{+ a_nt} x^2 d\mu(x) - \left(\int_{-a_n t}^{+a_nt} x d\mu(x)\right)^2\right) \to a^2
	\end{equation}
	and $(b_n - n\int_{-\infty}^{+ \infty} x d\mu(x))/a_n \to b$.
	
	We prove Theorem \ref{th:loi-stable} as a corollary of the following result.
	
	\begin{Th}[Generalized Central-Limit-Theorem for the norm]\label{th:GCLT-1}
		Let $\nu$ be a strongly irreducible and proximal probability distribution over $\mathrm{GL}(E)$. 
		Let $(\gamma_n)_{n \ge 0} \sim \nu^{\otimes\NN}$ and assume that $\EE(\kappa(\gamma_0)^2) = + \infty$.
		Assume that there exist sequences $(a_n) \in \RR_{> 0}^\NN$ and $(b_n) \in \RR^\NN$ such that the law of $ (\sum_{i = 0}^{n - 1}\kappa(\gamma_i) - b_n)/{a_n}$ converges to a non-degenerate probability distribution $\mathcal{L}$ with stability parameter $0 < \alpha \le 2$.
		Assume that there exists $q > \alpha$ such that $\EE(N(\gamma_0)^{q / 2}) < + \infty$.
		Then, there exist a sequence $(\gamma_n)_{n \ge 0} \sim \nu^{\otimes\NN}$ and a sequence $(y_n)$, defined on the same probability space, such that $y_n \sim \mathcal{L}$ for all $n$ and:
		\begin{equation}
			\lim_{n \to \infty} \EE\left(\left|\frac{\kappa(\overline\gamma_n) - b_n - nb}{a_n} - y_n\right|^q\right) = 0.
		\end{equation}
	\end{Th}

	\begin{Th}[Weak law of large numbers with doubled exponent]\label{th:wlln}
		Let $\nu$ be a strongly irreducible and proximal probability distribution over $\mathrm{GL}(E)$ and let $(\gamma_n)_{n \ge 0} \sim \nu^{\otimes\NN}$.
		Let $0 < q < 2$ be a constant such that $\EE(N(\gamma_0)^{q / 2}) < + \infty$.
		When $q \ge 1$, $\delta(\nu)$ is finite and we have:
		\begin{equation*}
			\EE\left(\left|\frac{\Delta\kappa(\widetilde\gamma_{0,n}) - n \delta(\nu)}{n^{1/q}}\right|^q\right) \underset{n \to + \infty}{\longrightarrow} 0.
		\end{equation*}
		When $q < 1$, we have:
		\begin{equation*}
			\EE\left(\left|\frac{\Delta\kappa(\widetilde\gamma_{0,n})}{n^{1/q}}\right|^q\right) \underset{n \to + \infty}{\longrightarrow} 0.
		\end{equation*}
	\end{Th}
	
	\begin{Th}\label{th:clt-Delta}
		Let $\nu$ be a strongly irreducible and proximal probability distribution over $\mathrm{GL}(E)$ and let $(\gamma_n)_{n \ge 0} \sim \nu^{\otimes\NN}$.
		Assume that $\EE(N(\gamma_0)) < + \infty$.
		Then $\delta(\nu)$ is finite and there exists a coupling of $(\gamma_n)$ with a sequence of identically distributed centred Gaussian random variables $(y_n)$, such that:
		\begin{equation}\label{clt-Delta}
			\EE\left(\left|\frac{\Delta\kappa(\widetilde\gamma_{0,n}) - n \delta(\nu)}{\sqrt{n}} - y_n\right|^2\right) \underset{n \to + \infty}{\longrightarrow} 0.
		\end{equation}
	\end{Th}
	
	In Theorem \ref{th:clt-Delta}, we do not claim that the distribution of the $y_n$'s is non degenerate.
	In fact, when $\KK$ is ultra-metric, it is easy to construct an example of strongly irreducible and proximal random walk for which $\Delta\kappa(\widetilde\gamma_{0,n})$ is equal to $0$ almost surely and for all $n$.
	Simply take $\nu$ to be the distribution of a matrix $(\mathds{1}_{(i,j) = 1,1} + M_{i,j})_{1 \le i,j \le d}$, where $(M_{i,j})_{1 \le i, j \le d}$ is taken i.i.d. following the Haar measure on the ball $\{x \in \KK, |x| < 1\}$. 
	Then for all $\gamma$ in the support of $\nu$, we have $\|\gamma\| = |e_1^* \gamma e_1|$ and by induction, we show that, for all $\gamma, \gamma' \in \Gamma_\nu$, we have $\|\gamma\| = |e_1^* \gamma e_1|$ and $\|\gamma\| = |e_1^* \gamma e_1|$ so $\|\gamma \gamma'\| \ge |e_1^* \gamma \gamma' e_1| = \|\gamma\|\|\gamma'\|$ and therefore $\Delta\kappa(\gamma, \gamma') = 0$.
	Therefore, $\Delta\kappa_*\nu^{\otimes n} = \delta_0$ for all $n$.
	
	In \cite{CLT16}, Benoist and quint give a formula to compute the limit of $\mathrm{Var}(\kappa_*\nu^{*n})/n$ in terms of the $\nu$ stationary measures $\xi$ on $\mathrm{P}(E)$ and $\xi'$ on $\mathrm{P}(E^*)$ when $\EE(N(\gamma_0)^2) < + \infty$. 
	Namely, when $\EE(N(\gamma_0)^2) < + \infty$ following \cite{CLT16}, we have:
	\begin{equation}\label{cocycle}
		\lim_n \mathrm{Var}(\Delta\kappa_*\nu^{\otimes n} )/n = \int_{\mathrm{P}(E)}\int_{\mathrm{GL}(E)} \left(\Delta\kappa(\gamma, x) - \psi(\gamma x) + \psi(x) - \delta(\nu)\right)^2 d\nu(\gamma) d\xi[x].
	\end{equation}
	For $\psi(x) = \int_{\mathrm{P}(E^*)} \Delta\kappa(f, x)d\xi^*[f]$.
	When $\EE(N(\gamma_0)) < + \infty$, \eqref{cocycle} still holds but we do not prove that in the present article.
	
	In an upcoming article, we will see that once we assume $\nu$ to be strongly irreducible and proximal, \eqref{cocycle} holds without any moment assumptions, meaning that when the formula on the right of \eqref{cocycle} is well defined and finite, it is the limit of $\mathrm{Var}(\Delta\kappa_*\nu^{\otimes n} )/n$ and when it is not, $\mathrm{Var}(\Delta\kappa_*\nu^{\otimes n} )/n \to + \infty$.
	
	That is of no use in the present article since we only use the fact that \eqref{clt-Delta} implies that:
	\begin{equation}\label{weak-2}
		\EE\left(\left|\frac{\Delta\kappa(\widetilde\gamma_{0,n}) - n \delta(\nu)}{a_n}\right|^2\right) \underset{n \to + \infty}{\longrightarrow} 0
	\end{equation}
	for all sequence $(a_n)$ such that $\sqrt{n}/a_n \to 0$. 
	In particular, \eqref{weak-2} holds for $(a_n)$ that satisfies \eqref{dom-at} in Theorem \ref{th:loi-stable}.

	\subsection{GCLT in higher rank}
	
	Given a vector space $E$ and an integer $k$, we write $\bigwedge^k E$ for the space, endowed with a alternate $k$-linear map $\bigwedge^k: E^k \to \bigwedge^k E$ that has the universal property of factorizing all alternate $k$-linear maps, equivalently, it is the dual of the space of alternate $k$-linear forms. 
	Note that $\bigwedge^k E$ is a $\KK$-vector space of dimension $\binom{d}{k}$ for all $0 \le k \le d$, is $\{0\}$ for $k > d$ and is undefined for $k < 0$.
	
	We remind that the Euclidean (resp. Hermitian, resp. ultra-metric) norm on $E$ defines a unique Euclidean (resp. Hermitian, resp. ultra-metric) norm on $\bigwedge^k E$ for all $1 \le k \le d$. This norm is characterized by the fact that for all family $(x_i)_{0 \le i < k} \in E^k$, we have $\|\bigwedge_{i = 0}^{k-1} x_i\| = \min_{\bigwedge y_i = \bigwedge x_i}\prod_{i = 0}^{k-1} \|y_i\|$ and therefore $\|\bigwedge_{i = 0}^{k-1} x_i\| \le \prod_{i = 0}^{k-1} \|x_i\|$.
	
	The natural norm on $\bigwedge^k E$ can be constructed for example as a scalar multiple of the norm induced\footnote{Given a linear and surjective map, $\pi : E \to F$, with $E$ a normed vector space, the semi-norm induced by $\pi$ on $F$ is simply the distance (closest points distance or Hausdorff distance, they match in this case) between cosets. When $E$ is finite dimensional, it is easy to check that this semi-norm is in fact a norm.} from the natural projection $E^{\otimes k} = \bigotimes^k E \twoheadrightarrow \bigwedge^k E$, that identifies two $k$-vectors whenever they have the same image by all alternate $k$-linear map.
	The scaling constant is $1$ when $\KK$ is ultra-metric and $\sqrt{k!}$ when $\KK$ is Archimedean, \ie $\KK = \RR$ or $\KK = \CC$.
	For that reason, the choice of the Euclidean or Hermitian norm is not always consistent through the literature. 
	That being said, the choice of scalar multiple of norms plays no importance in the present work since the scaling constant cancels out in the operator norm on $\mathrm{End}(\bigwedge^k E)$.
	
	Given $g \in \mathrm{GL}(E)$, and $1 \le k \le d$, we write $\bigwedge^k g\in \mathrm{GL}(\bigwedge^k E)$ for the map characterized by $\bigwedge^k g \bigwedge_{i = 1}^{k} x_i = \bigwedge_{i = 1}^{k} g x_i$.
	We say the a measure $\nu$ on $\mathrm{GL}(E)$ is totally irreducible if the measure $\bigwedge^k_*\nu$ is strongly irreducible for all $1 \le i \le d$.
	For all $1 \le i \le d$, we define:
	\begin{equation}\label{kappa-i-def}
		\kappa_i(g) = \log\left(\frac{\|\Wedge^i g\|}{\|\Wedge^{i-1} g\|}\right) = \max_{\substack{F \subset E \\ \dim(F) \ge i}} \min_{x \in F \setminus\{0\}} \log\left(\frac{\|gx\|}{\|x\|}\right).
	\end{equation}
	The right hand side makes it clear that the map $k \mapsto\kappa_k(g)$ is non decreasing.
	We write $\dot\kappa : \mathrm{GL}(E) \to \RR^d$ for the map $g \mapsto (\kappa_i(g))_{1 \le i \le d}$, the map $\dot\kappa$ is often referred to as the Cartan projection.
	Note that the right member of \eqref{kappa-i-def} makes sense for all linear map $g \in \mathrm{Hom}(E_0, E_1)$, with the convention $\kappa_i(g) = - \infty$ when $i > \dim(E_0)$ for $E_0$ and $E_1$ two Euclidean, Hermitian or ultra-metric spaces.
	
	For all $0 \le j \le d$, and for all $g \in \mathrm{GL}(E)$, we write $\overline{\kappa}_j(g) = \log\|\bigwedge^k g\| = \sum_{i = 1}^j \kappa_i(g)$ and we write $\overline\kappa(g) = (\overline\kappa_j(g))_{0 \le j \le d} \in \RR^d$. 
	Note that $\overline\kappa_j$ is a sub-additive map for all $0 \le j \le d$ so $\overline\kappa$ and $\Delta\overline\kappa$ are both sub-additive for the coordinate by coordinate comparison order.

	\begin{Def}
		Let $\nu$ be a probability measure on $\mathrm{GL}(E)$ and let $(\gamma_n) \sim \nu^{\otimes\NN}$. 
		For all $0 \le j \le d$, we write:
		\begin{equation}
			\overline\delta(\nu) := \lim_{n \to + \infty}\frac{\EE(\Delta\overline\kappa(\widetilde\gamma_n))}{n}.
		\end{equation}
		When, $\overline\delta(\nu) \in (-\infty, 0]^d$, we write $\dot\delta(\nu) : = (\overline\delta_i(\nu) - \overline\delta_{i-1}(\nu))_{1 \le i \le d}$.
	\end{Def}

	Contrary to $\kappa$, $\dot\kappa$ is not sub-additive, however, the map $\overline{\kappa} = (\kappa_1 + \cdots +\kappa_k)_{0 \le k \le d}$ is sub-additive coordinate by coordinate because we have $\overline{\kappa}_k(g) = \log\|\bigwedge^k g\|$ for all $g \in \mathrm{GL}(E)$ and $0 \le k \le d$.
	Therefore, by Kingman-s sub-additive ergodic Theorem, there is a vector $\overline{\delta}(\nu) \in [-\infty, 0]^d$ such that $\nu^{\otimes \NN}$-almost surely, $\lim_{n}\frac{\Delta\overline\kappa(\widetilde\gamma_n)}{n} = \overline{\delta}(\nu)$.
	When $\overline{\delta}(\nu) \in (-\infty, 0]^d$, we simply set $\dot\delta(\nu)_k = \overline\delta(\nu)_k - \overline\delta(\nu)_{k-1}$ for all $1 \le k \le d$.

	\begin{Th}[Weak law of large numbers in higher rank]\label{th:wlln-high}
		Let $\nu$ be a totally irreducible probability measure on $\mathrm{GL}(E)$ and let $(\gamma_n)_{n \ge 0} \sim \nu^{\otimes\NN}$. 
		Let $0 < q < 2$ such that $\EE(N(\gamma_0)^{q/2}) < + \infty$. 
		When $q \ge 1$, $\dot\delta(\nu)$ is well defined and finite and we set $b = -\dot\delta(\nu)$ otherwise, we set $b = 0$.
		Then we have:
		\begin{equation}
			\lim_{n \to + \infty} \EE\left(\left\|\frac{\Delta\dot\kappa(\widetilde\gamma_n) + nb}{n^{1/q}}\right\|^q\right) = 0.
		\end{equation}
	\end{Th}

	\begin{Th}[Central limit Theorem in higher rank]\label{th:clt-high}
		Let $\nu$ be a totally irreducible probability measure on $\mathrm{GL}(E)$ and let $(\gamma_n)_{n \ge 0} \sim \nu^{\otimes\NN}$. 
		Assume that $N_* \nu$ is integrable. 
		Then there exist constants $b \in \RR^d$ and $a \in\RR^d \otimes\RR^d$ and a coupling of $(\gamma_n) \sim \nu^{\otimes\NN}$ with random variables $y_n \sim \mathcal{N}_a$ for all $n \in \NN$ such that:
		\begin{equation}\label{clt-high}
			\lim_{n \to \infty} \EE\left(\left\|\frac{\Delta\dot\kappa(\widetilde\gamma_n) - n\dot\delta(\nu)}{\sqrt{n}} - y_n\right\|^2\right) = 0.
		\end{equation}
	\end{Th}

	Using Theorems \ref{th:wlln-high} and \ref{th:clt-high}, we prove an analogous of Theorem \ref{th:loi-stable} for totally irreducible distributions.
	
	An interesting fact is that in this case, we do not have an analogous of Theorem \ref{th:GCLT-1} with moment assumptions on $N$.
	That is because $N = \kappa_1 - \kappa_d$.
	So when $\dot\kappa$ is in the domain of attraction of a probability measure of stability parameter $\alpha$, $N$ has a finite moment of order $q$ for all $q < \alpha$.
	We remind that by a remark of Guivarc'h and Raugi in \cite{GR89} under the strong irreducibility assumption, saying that a semi-group is not proximal is equivalent to saying that the quantity $\kappa_1 -\kappa_2$ is bounded on $\Gamma$.

	\begin{Th}[Generalized Central Limit Theorem in higher rank]\label{th:GCLT-Higher}
		Let $\nu$ be a totally irreducible probability measure on $\mathrm{GL}(E)$. 
		Let $(\gamma_n)_{n \ge 0} \sim \nu^{\otimes\NN}$.
		Assume that $\dot\kappa_* \nu$ has no moment of order $2$.
		Assume that there exist sequences $(a_n) \in \RR_{> 0}^\NN$ and $(b_n) \in (\RR^d)^\NN$ such that the law of $(\sum_{i = 0}^{n - 1}\dot\kappa(\gamma_i) - b_n)/{a_n}$ converges in the weak-$*$ topology to a non-degenerate probability distribution $\mathcal{L}$ on $\RR^d$.
		Then there exist a constant $b \in \RR^d$ such that the law of $(\dot\kappa(\overline\gamma_n) - b_n - nb)/{a_n}$ converges to $\mathcal{L}$ in the weak-$*$ topology.
	\end{Th}
	
	For the present article we do not need a fine understanding of the domain of attraction of stable distributions.
	All we need to know is that when the hypotheses of Theorem \ref{th:GCLT-Higher} hold, there exists $\alpha \in (0, 2]$ such that $\dot\kappa_*\nu$ both has a finite moment of order $q$ for all $q \in (0, \alpha)$ and $a_n / n^{1/q} \to 0$ for all $q \in \{2\} \cup (\alpha, +\infty)$.
	The fact that $\sqrt{n}/a_n \to 0$ when $\alpha = 2$ comes from the fact that $\dot\kappa_* \nu$ has no moment of order $2$.
	When $\dot\kappa_* \nu$ has a finite second order moment the technique developed in \cite{CLT16} tell us that the law of $(\dot\kappa(\overline\gamma_n) - n \lambda_1)/\sqrt{n}$ converges to a Gaussian distribution whose covariance matrix is given by a cocycle formula, analogous to \eqref{cocycle}, that we will not detail in the present article.
	
	That being said, in order to gain a basic understanding of the implications of Theorem \ref{th:GCLT-Higher}, let us remind that a law $\mu$ on $\RR^d$ is in the domain of attraction of a stable law, of stability parameter $\alpha$, in the sense that there exist sequences $(a_n) \in \RR_{>0}$ and $(b_n)_{n > 0}$ such that $(x\mapsto (x-b_n)/a_n)_*\nu^{*n} \rightharpoonup \mathcal{L}$ if and only if:
	\begin{itemize}
		\item When $\alpha = 2$, for all $t > 0$, the covariance matrix of the restriction
		\footnote{Since $\mu\{x,\|x\|< a_n\}$ has limit $1$, it does not matter if we look at the restriction or normalized restriction of $\mu$. 
		We remind that the coefficient of index $(i,j)$ of the covariance matrix of the restriction of $\mu$ to $\mathcal{B}(ta_n)$ is $\int_{\mathcal{B}(ta_n)} x_i x_j d\mu(x) - \int_{\mathcal{B}(ta_n)} x_i d\mu(x) \int_{\mathcal{B}(ta_n)} x_j d\mu(x)$. Integrating over all $\RR^d$ on the second term does not change the limit.} 
		of $\mu$ to the ball of radius $t a_n$ multiplied by $n/(a_n)^2$ converges to the covariance matrix of $\mathcal{L}$ and $(b_n - n \int x d\mu(x))/ a_n$ converges to $\int x d\mathcal{L}(x)$. 
		In this case $\mathcal{L}$ is Gaussian.
		
		\item When $\alpha < 2$ the measure\footnote{The measure $n(x \mapsto  x/a_n)_* \mu$ is $n$ times the law of $x / a_n$ for $x \sim \mu$, it has total mass equal to $n$ on $\RR^d$ and to $n (1- \mu\{0\})$ on $\RR^d \setminus\{0\}$, so its limit is zero or has infinite total mass on $\RR^d\setminus\{0\}$.} $n(x \mapsto x/a_n)_* \mu$ converges to a non-zero and locally finite measure, for the weak-$*$ topology on $\RR^d \setminus \{0\}$, and $b_n - n \int_{\mathcal{B}(a_n)})/a_n$ converges in $V$. 
		
		In this case, the limit of $n(x \mapsto x/a_n)_* \mu$ is of type $C\mathcal{P}_\alpha *\beta$, where $\beta$ is a probability measure on the unit sphere, $\mathcal{P}_\alpha$ is the law of density $\alpha t^{-\alpha - 1} dt$ on $(0, + \infty)$ and $C \in (0, + \infty)$ is a constant.
		The law $\mathcal{L}$ is determined by the stability parameter $\alpha$, the measure $\beta$ on the unit sphere (called harmonic measure of $\mathcal{L}$), the constant $C$ and the limit of $b_n - n \int_{\mathcal{B}(a_n)})/a_n$.
		
		More precisely, $\mathcal{L}$ is the law of:
		\begin{equation}
			b + a \sum_{k = 1}^{\infty} \left(\overline{\tau}_k^{-1/\alpha} y_k - \int_{k}^{k+1}t^{-1/\alpha} dt \EE(y_k)\right).
		\end{equation}
		where $a = C^{1/\alpha}$, $b = \lim_n b_n - n \int_{\mathcal{B}(a_n)})/a_n$ and $(\tau_k)_{k \ge 0}, (y_k)_{k \ge 1} \sim \mathcal{E}^{\otimes\NN} \otimes \beta^{\otimes\NN}$ for $\mathcal{E}$ the exponential distribution\footnote{This way, the set $\{\overline{\tau}_k, k \ge 1\}$ is a Poisson point process of intensity Lebesgue on $(0, +\infty)$ and the set $\{a\overline{\tau}_k^{-1/\alpha}y_k, k \ge 1\}$ is a Poisson point process of intensity $C\mathcal{P}_\alpha *\beta$ on $\RR^d \setminus\{0\}$.
		The distribution of the Poisson point process of intensity $\lim_n n(x \mapsto x/a_n)_* \mu$ is the limit of the distribution of the set $\{x_k /a_n\,|\, 0\le k < n\}$ for $(x_k)_{k \ge 0} \sim \nu^\otimes\NN$.} on $(0, + \infty)$ and $\overline\tau_k = \sum_{j =0}^{k-1} \tau_j$.
		In the case of Theorem \ref{th:GCLT-Higher}, the harmonic measure $\beta$ associated to $\mathcal{S}$ has to be supported on the Weil chamber $\{x\in\RR^d\,|\, x_1 \ge x_2 \ge \dots \ge x_d\}$ because $\kappa_* \nu$ is.
	\end{itemize} 
	
	Following the historic framework developed by Lévi in the one-dimensional case, the domain of attraction of multi-dimensional stable distributions can be understood via the Fourier transform.
	However, the harmonic measure $\beta$ does not come out as explicitly as it does via the study of order statistics, developed in \cite{series} and explained in the introduction of \cite{Samorodnitsky1995StableNR}.
	This point of view allows for a more intuitive and dimension-independent understanding of the domain of attraction of stable distributions.


	\section{Pivoting technique and gain of moment}\label{section:pivot}
	
	This section is dedicated to the proof of Theorem \ref{th:Delta-unif}.
	We first prove Theorem \ref{th:pivot} as a corollary of \cite[Theorem~4.7]{moi}.
	Then we prove Theorem \ref{th:Delta-unif} as a corollary of both Lemma \ref{lem:Delta} and Theorem \ref{th:Delta-unif}.
	Lemma \ref{lem:Delta} combined with \cite[Theorem~4.7]{moi} allow to give results that are similar to Theorem \ref{th:Delta-unif} for random walks in spaces that satisfy good contraction properties like trees or relatively hyperbolic groups.
	The reader who is not familiar with \cite[Section~4]{moi} may skip the proof of Theorem \ref{th:pivot} in a first read and proceed to Section \ref{section:tail} of the present article for the proof of Theorem \ref{th:Delta-unif}.

	\subsection{Pivoting technique}
	
	The aim of the present section is to prove Theorem \ref{th:pivot}, using the tools developed in \cite{moi}. 
	This is the occasion to give more details on the theoretical framework behind Theorem \ref{th:pivot} and put emphasis on the constructive nature of the proof.
	
	Let us quickly redefine the notations used in \cite[Section~4]{moi} and in the present paragraph.
	Given a semi-group $\Gamma$, a sequence $(g_n)_{n \ge 0} \in \Gamma^\NN$, and a sequence of integers $(p_n)_{n \ge 0} \in\NN^\NN$, for all $n \ge 0$, we write $\overline{p}_n = p_0 + \cdots + p_{n-1}$ and $\widetilde{g}^p_n = (g_{\overline{p}_n}, \dots, g_{\overline{p}_{n+1}-1})$ and $g^p_n = g_{\overline{p}_n} \cdots g_{\overline{p}_{n+1}-1}$.
	Given a word $\tilde{g} = (g_0, \dots, g_{l-1})$, we write $\Pi(\tilde{g}) = g_0 \cdots g_{l-1} = g_{0,l}$.
	Given two indices $m < n$ and a sequence $(g_n)$ in a semi-group, we write $g_{m,n}$ as a compact notation for $g_m \cdots g_{n-1}$. 
	That way, for all $l < m < n$, we have $g_{l,n} = g_{l,m}  g_{m,n}$. 
	
	We call metric semi-group a semi-group endowed with a second countable and metrizable topology, in practice, we think of $\Gamma$ as $\mathrm{GL}(E)$.
	Given a measurable binary relation $\AA$ on $\Gamma$ and $0 < \rho < 1$, we say that a probability measure $\nu_s$ is $\rho$-Schottky if for all constant $g \in \Gamma$, and for $\gamma \sim \nu_s$, we have:
	\begin{equation}\label{eq:def-schottky}
		\PP(g \AA \gamma) \ge 1 - \rho \quad \text{and} \quad \PP(\gamma \AA g) \ge 1 - \rho.
	\end{equation}
	Given two matrices $g$ and $h$ and $0 < \eps \le 1$, we write $g \AA^\eps h$ when $\|gh\| \ge \eps \|g\| \|h\|$, or equivalently, when $\Delta\kappa(g,h) \le |\log(\eps)|$.
	For all matrix $g \in \mathrm{GL}(E)$, we write $\sigma(g) = \exp(\kappa_2(g)-\kappa_1(g)) \in [0,1]$.
	
	In \cite{moi}, we have seen the following result:
	
	\begin{Th}[Pivotal times~{\cite[Theorem~1.6]{moi}}]\label{th:pivot-prequel}
		Let $\nu$ be a strongly irreducible and proximal probability distribution over $\mathrm{GL}(E)$.
		There exist constants $0 < \varepsilon \le 1/2$ and $m \in\mathbb{N}$, a compact $K \subset \mathrm{GL}(E)$ and a probability distribution $\mu$ on $\mathrm{GL}(E)^\mathbb{N} \times (\mathbb{N}\setminus\{0\})^\mathbb{N}$ such that given $\left(\left(\gamma_n\right)_{n \ge 0},\,\left(p_n\right)_{n\ge 0}\right) \sim \mu$, the following assertions hold.
		\begin{enumerate}
			\item We have $\left(\gamma_n\right)_{n \ge 0} \sim \nu^{\otimes\mathbb{N}}$. \label{pivot:loi}
			\item The sequence $\left(\widetilde{\gamma}^p_{2k}\right)_{k \ge 1}$ is i.i.d. and independent of $\widetilde{\gamma}^p_0$. \label{indep}
			\item For all $k \ge 0$, we have $\widetilde{\gamma}^p_{2k+1} \in K^m$ and $\sigma(\gamma^p_{2k+1}) \le \frac{\varepsilon^6}{48}$ almost surely.\label{pivot:compact-sqz}
			\item For all $k \ge 0$, the conditional distribution of $\widetilde{\gamma}^p_{2k+1}$ with respect to the data of $\left(\widetilde{\gamma}^p_{k'}\right)_{k' \neq 2k + 1}$ is given by a function of $(\widetilde{\gamma}^p_{2k},\widetilde{\gamma}^p_{2k+2})$ that does not depend on $k$.\label{pivot:condition}
			\item Almost surely, and for all $k \ge 0$, we have $\mathbb{P}\left( g \mathbb{A}^\varepsilon \gamma^p_{2k+1} \, \middle| \, (\widetilde{\gamma}^p_{k'})_{k' \neq 2k+1} \right) \ge 3/4$ for all $g$ and $\mathbb{P}\left(\gamma^p_{2k+1} \mathbb{A}^\varepsilon h \, \middle| \, (\widetilde{\gamma}^p_{k'})_{k' \neq 2k+1} \right) \ge 3/4$ for all $h$.\label{schottky}
			\item There exists constants $C, \beta > 0$ such that $\mathbb{E}(e^{\beta p_n}) \le C$ for all $n \ge 0$.\label{pivot:exp-moment}
			\item For all $A \subset \mathrm{GL}(E) \setminus K$, we have $\mathbb{P}\left(\gamma_n \in A \, \middle| \, \left(p_k\right)_{k \in\mathbb{N}}\right) \le 2\nu(A)$.\label{pivot:tail}
			\item For all $i \le j \le k$ such that $i < k$, for all $f \in E^*\cup\mathrm{End}(E)$ such that $f \mathbb{A}^\varepsilon \gamma^p_{i}$ and for all $h \in E \cup \mathrm{End}(E)$ such that $\gamma^p_{k - 1} \mathbb{A}^\varepsilon h$, we have $f \gamma^p_{i} \cdots \gamma^p_{j-1} \mathbb{A}^\frac{\varepsilon}{2} \gamma^p_j \cdots \gamma^p_{k} h$.\label{pivot:herali}
		\end{enumerate}
	\end{Th}
	
	This result is weaker than Theorem \ref{th:pivot}, which moreover states that the data of $(p_n)_{n \ge 0}$ and of $(\gamma_n)_{n \ge 0, \gamma_n \notin K}$ are independent relative to the data of $\{n \ge 0\,|\, \gamma_n \notin K\}$.
	In Theorem \ref{th:pivot-prequel}, point \eqref{pivot:tail} only says that the relative distribution of each $\gamma_n$ with respect to the data of $(p_k)_{k \ge 0}$ is absolutely continuous with respect to $\nu$ outside of $K$ and does not say anything about the coupling between the $(\gamma_n)_n$'s. 
	To get the squared probability in Theorem \ref{th:Delta-unif}, we really need the independence of the $g_n$'s.
	
	We remind that the sequence $(p_n)_n$ of Theorem \ref{th:pivot-prequel} was taken to be the sequence $(\check{p}_n)_n$, constructed in \cite[Theorem~4.7]{moi}.
	The aim of the present section is to show that this sequence $(\check{p}_n)_n$ does satisfy the conclusions of Theorem \ref{th:pivot}.
	
	Before giving the statement of \cite[Theorem~4.7]{moi}, let us remind that the Schottky measure we want to consider is given by the following Lemma.
	This result is a probabilistic formulation of a well known geometric result, credited to Abels, Margulis and Soifert \cite{AMS-esmigroup}.
	Here we state it as it is stated and proven in \cite[Corollary~3.17]{moi}, with $\delta(\eps) = \eps^6 / 48$ and $\rho = 1 / 6$.

	\begin{Lem}\label{lem:schottky}
		Let $\nu$ be a strongly irreducible and proximal probability measure on $\Gamma$. 
		There exists an integer $m$, two constants $\alpha, \eps \in (0,1)$ and a probability measure $\tilde\nu_s$ on $\Gamma^m$ such that:
		\begin{enumerate}
			\item The measure $\nu_s = \Pi_*\tilde\nu_s$ is $1/6$-Schottky for $\AA^\eps$ in the sense of \eqref{eq:def-schottky}.
			\item The measure $\tilde\nu_s$ is absolutely continuous with respect to $\nu^{\otimes m}$ in the sense that $\alpha \tilde\nu_s \le \nu^{\otimes m}$.
			\item For all $(s_0, \dots, s_{m-1})$ in the support of $\tilde\nu_s$, we have $\sigma({s}_{0,m}) \le \eps^6/48$.
			\item The support of $\tilde\nu_s$ is compact in $\Gamma^m$.
		\end{enumerate}
	\end{Lem}
	
	Now let us give the full statement of \cite[Theorem~4.7]{moi} and give it an alias for the present article.
	Given $\alpha \in (0,1)$, we write $\mathcal{G}_\alpha := (1-\alpha)\sum_{k = 0}^\infty{\alpha^k}\delta_k$ for the geometric distribution of parameter $\alpha$, we remind that $\mathcal{G}_\alpha$ has mean $\alpha/(1-\alpha)$.
	Given two sequences of integers $v = (v_n)$ and $w = (w_n)$, we write $w^v = (\sum_{k = \overline{v}_n}^{\overline{v}_{n+1} - 1}w_{k})_{n \ge 0}$, that way, for all sequence $\gamma = (\gamma_n)_n$ in a semi-group, we have $\gamma^{(w^v)} = (\gamma^w)^v$.
	
	\begin{Th}[Pivot extraction~{\cite[Theorem~4.7]{moi}}]\label{th:ex-piv}
		Let $\Gamma$ be a metric semi-group endowed with a measurable binary relation $\AA$.
		Let $\alpha\in (0,1)$ and let $m \ge 1$.  
		Let $\tilde{\nu}_s$ be a probability distribution on $\Gamma^m$ such that $\alpha \tilde\nu_s \le \nu^{\otimes m}$ and let $\tilde\kappa:= \frac{1}{1-\alpha}(\nu^{\otimes m} - \alpha \tilde\nu_s)$.
		Let $\nu_s = \Pi_*\tilde{\nu}_s$ and assume that $\nu_s$ is $1/6$-Schottky for $\AA$.
		Then there exist random sequences $({u}_k)_k, ({s}_k)_k, (w_{2k})_k \sim \tilde\kappa^{\odot\NN}\otimes \tilde\nu_s^{\odot\NN} \otimes \mathcal{G}_{1-\alpha}^{\otimes\NN}$ and two random sequences $(v_k)_k$ and $(\check{p}_k)_k$ all defined on the same probability space such that if we write $w_{2k+1} := 1$ for all $k$, $\check{w}:= mw$, $(\gamma_n)_{n\in\NN}:= \bigodot_{k=0}^\infty \widetilde{u}^{\check w}_{2k} \widetilde{s}^{\check w}_{2k+1}$, $\check{v}:= \check{w}^v$, $\hat{v}_{2k}:= \check{v}_{4k} + \check{v}_{4k + 1} + \check{v}_{4k + 2}$ and $\hat{v}_{2k+1} = \check{v}_{4k + 3}$ for all $k$ (or in compact notations $\hat{v} = \check{v}^{(3,1)^{\odot\NN}}$) and $p := \hat{v}^p$, then the following assertions hold:
		\begin{enumerate}
			\item \label{item:v-indep}
			The data of $(v_k)_k$ is independent of the joint data of $(\tilde{u}_k)_k$ and $(w_{k})_k$. 
			\item \label{item:law-of-v}
			For all $k \in\NN$, we have $v_{4k+1} = v_{4k+2} = v_{4k+3} = 1$ and $\left(
			\frac{v_{4k}-1}{2}\right)_{k\in\NN} \sim \mathcal{G}_{1/3}^{\otimes\NN}$. 
			\item \label{item:zouli-alignemon}
			For all $k \in\NN$, we have $\gamma^{\check{v}}_{4k}\AA \gamma^{\check{v}}_{4k+1} \AA \gamma^{\check{v}}_{4k+2}$. 
			\item \label{item:hat-v-is-stopping-time}
			We have $(\widetilde{\gamma}^{\check{v}}_{4k+3})_{k\in\NN} \sim \tilde{\nu}_s^{\otimes \NN}$ and the sequence $\left((\widetilde{\gamma}^{\check{v}}_{4k},\widetilde{\gamma}^{\check{v}}_{4k+1}, \widetilde{\gamma}^{\check{v}}_{4k+2})\right)_{k\in\NN}$ is i.i.d. and independent of $(\widetilde{\gamma}^{\check{v}}_{4k+3})_{k\in\NN}$.
			\item \label{item:tail}
			The data of $(p_k)_k$ is independent of the joint data of $(\tilde{u}_k)_k$ and $(w_{k})_k$ and $(v_k)_k$ and $(\widetilde{\gamma}^{\hat{v}}_{k})_{k\in \{0,1,2\}+4\NN}$.
			\item \label{item:indep-p}
			Each $p_k$ is a positive odd integer, $p_{2k+1} = 1$ for all $k$ and $(p_{2k+2})_k$ is i.i.d. and independent of $p_0$.
			\item \label{item:exp-p}
			The distribution laws of $p_0$ and $p_2$ have a finite exponential moment.
			\item \label{item:indep}
			The sequence $\left(\widetilde{\gamma}^{\check{p}}_{2k+2}\right)_{k \ge 0}$ is i.i.d. and independent of $\widetilde{\gamma}^{\check{p}}_{0}$.
			\item \label{item:ali}
			For all $k\in\NN$, we have $\gamma^{\check{p}}_{2k} \AA \gamma^{\check{p}}_{2k+1}$ almost surely.
			\item \label{item:ali-rec}
			Almost surely and for all $k \in\NN$, we have $\gamma^{\check{p}}_{2k+1} \AA \gamma^{\hat{v}}_{\overline{p}_{2k+2}}$ and there is a family of odd integers $1 = c_1^k < c_2^k < \dots < c_{j_k}^k = p_{2k+2}$ such that for all $1\le i <j_k$, we have:
			\begin{equation}\label{eq:ali-rec}
				\gamma^{\hat{v}}_{\overline{p}_{2k+2}}\cdots \gamma^{\hat{v}}_{\overline{p}_{2k+2}+c_i^k-1} \AA \gamma^{\hat{v}}_{\overline{p}_{2k+2}+c_i^k} \AA \gamma^{\hat{v}}_{\overline{p}_{2k+2}+c_i^k + 1} \cdots \gamma^{\hat{v}}_{\overline{p}_{2k+2}+c_{i+1}^k - 1}.
			\end{equation}
			\item For all $k \in\NN$, the conditional distribution of $\widetilde\gamma^{\check{p}}_{2k+1}$ with respect to the joint data of $\left(\widetilde\gamma^{\check{p}}_{k'}\right)_{k'\neq 2k+1}$ and $(\tilde{u}_k)_k$ and $(w_k)_k$ and $(v_k)_k$ is the normalized restriction of $\tilde{\nu}_s$ to the measurable set: \label{item:tjr-piv}
			\begin{equation}
				C_k:= \Pi^{-1}\left\{\gamma \in\Gamma\,\middle|\, \gamma^{\check{p}}_{2k} \AA \gamma^{\check{p}}_{2k+1} \AA \gamma^{\hat{v}}_{\overline{p}_{2k+2}} \right\}.		
			\end{equation}
		\end{enumerate}
	\end{Th}

	Let us now admit \ref{th:ex-piv} as a black box result and break down how Theorem \ref{th:pivot} is a direct consequence of Theorem \ref{th:ex-piv}. 
	Namely, we show that for $\nu_s$ as constructed in Lemma \ref{lem:schottky}, the sequence $(\check{p}_k)$ as constructed in Theorem \ref{th:ex-piv} satisfies the conclusions on Theorem \ref{th:pivot}.
	To construct the sequence $(g_n)$, we only need to look at points \eqref{item:v-indep} \eqref{item:indep} of Theorem \ref{th:ex-piv}. 
	The other points of Theorem \ref{th:ex-piv} are only needed to prove points \eqref{pivot:markov} to \eqref{pivot:herali} in Theorem \ref{th:pivot}.
	These points are reformulations of the conclusions of Theorem \ref{th:pivot-prequel} in terms of $\Delta\kappa$ instead of $\AA^\eps$, we will quickly remind how they articulate with Theorem \ref{th:ex-piv} without going too deep into the proof.

	\begin{Lem}\label{lem:descri-bloc-pivot}
		Let $\nu$ be a strongly irreducible and proximal probability distribution over $\mathrm{GL}(E)$.
		Let $\alpha, \eps, m$ and $\tilde{\nu}_s$ be as in Lemma \ref{lem:schottky} and $(\check{p}_k)$ be as in Theorem \ref{th:ex-piv} for $\AA = \AA^\eps$.
		Let $K \subset \Gamma$ be such that $\tilde\nu_s(K^m) = 1$. 
		Let $\nu_{K^c}$ be the normalized restriction of $\nu$ to $K^c = \Gamma \setminus K$ when $\nu(K) < 1$ and $\nu_{K^c} = \delta_{\mathrm{Id}_E}$ otherwise.
		Let $(i_k = \mathds{1}_K(\gamma_k)$ and let $g_k = \gamma_k$ for all $k$ such that $i_k = 0$. 
		Assume that the conditional distribution of $(g_k)_{i_k = 1}$ with respect to all the random variables defined in Theorem \ref{th:ex-piv} is $\nu_{K^c}^{\otimes \{k \,|\, i_k  = 0\}}$.
		Then for all $k$, the joint data of $(\check{p}_k)_{k \ge 0}$ and $(i_k)_{k \ge 0}$ is independent of $(g_k)_{k \ge 0}$.
	\end{Lem}
	
	\begin{proof}
		First, note that $(g_k)_{k \ge 0} \sim \nu_{K^c}^{\otimes \NN}$ and $(g_k)_{k \ge 0}$ is independent of $(i_k)_{k \ge 0}$ so all we have to show is that the data of $(g_k)_{k \ge 0}$ is independent of $(\check{p}_k)_{k \ge 0}$ relative to $(i_k)_{k \ge 0}$.
		
		First, we show that the data of $(g_k)_{k \ge 0}$ is independent of $(w_k)_{k \ge 0}$ relative to $(i_k)_{k \ge 0}$.
		By definition of $(g_k)_{i_k = 1}$, all we have to show is that the conditional distribution of $(\gamma_k)_{i_k = 0}$ with respect to the joint data of $(i_k)$ and $(w_k)$ is almost surely equal to $\nu_{K^c}^{\otimes \{k \,|\, i_k  = 0\}}$.
		First note that the normalized restriction of $\tilde\kappa$ to $\Gamma^m \setminus K^m$ is equal to the normalized restriction of $\nu^{\otimes m}$ to the same set. 
		Moreover, the data of $(i_k)_{k \ge 0}$ is determined by the joint data of $(u_k)_{k \ge 0}$ and $(w_k)_{k \ge 0}$ so the conditional distribution of $(u_k)_{i_k = 0}$ with respect to the joint data of $(i_k)$ and $(w_k)$ is $\nu_{K^c}^{\otimes \{k \,|\, i_k  = 0\}}$.
		
		To conclude, we use points \eqref{item:v-indep} and \eqref{item:tail}, which combined tell us that the joint data of $(u_k)_k$ and $(w_k)_k$ (and therefore of $(i_k)$) is independent of the joint data of $(p_k)_k$ and $(v_k)_k$ so the conditional distribution of $(u_k)_{i_k = 0}$ with respect to the joint data of $(i_k)_k$, $(w_k)_k$, $(v_k)_k$ and $(p_k)_{k \ge 0}$ is $\nu_{K^c}^{\otimes \{k \,|\, i_k  = 0\}}$.
		Moreover, the data of $\check{p}_k$ is determined by the joint data of $(w_k)_k$, $(v_k)_k$ and $(p_k)_{k \ge 0}$, which concludes the proof.
	\end{proof}
	
	\begin{Lem}\label{almost-add-pivot}
		Let $\nu$ be a strongly irreducible and proximal probability distribution over $\mathrm{GL}(E)$.
		Let $\alpha, \eps, m$ and $\tilde\nu_s$ be as in Lemma \ref{lem:schottky}.
		Let $C = |\log(\eps)| + \log(2)$.
		Let $\check{p}$ be as in Theorem \ref{th:ex-piv}, then for all $0 \le i < j < k$, we have almost surely:
		\begin{equation}\label{ijk}
			\left|\Delta\kappa(\gamma^{\check{p}}_{i,j}, \gamma^{\check{p}}_{j,k})\right| \le C.
		\end{equation}
		Moreover, for all $1 \le i < j < k$ and for all $g,h$ such that $\left|\Delta\kappa(g,\gamma^{\check{p}}_{i})\right| \le C$ and $\left|\Delta\kappa(\gamma^{\check{p}}_{k-1}, h)\right| \le C$, we have:
		\begin{equation}\label{ghijk}
			\left|\Delta\kappa(g\gamma^{\check{p}}_{i,j}, \gamma^{\check{p}}_{j,k}h)\right| \le C.
		\end{equation}
	\end{Lem}
	
	\begin{proof}
		In \cite{moi}, we have shown that for $\alpha, \eps, m$ and $\tilde\nu_s$ be as in Lemma \ref{lem:schottky} and $(\check{p}_n)_n$ as in Theorem \ref{th:ex-piv}, the joint law $\mu$ of $(\gamma_n)_{n \ge 0}$ and $(p_n)_{n \ge 0}$ satisfies the conclusion of Theorem \ref{th:pivot-prequel}.
		The formula \eqref{ijk} follows from \eqref{pivot:herali} in Theorem \ref{th:pivot-prequel}.
	\end{proof} 
	
	Let us now give the details of the proof of Theorem \ref{th:pivot}.
	
	\begin{proof}[Proof of Theorem \ref{th:pivot}]
		Let $\alpha, \eps > 0$,  $m \ge 1$ and $K \subset \mathrm{GL}(E)$ and $\tilde{\nu}_s \in \mathrm{Prob}(K^m)$ be as in Lemma \ref{lem:schottky}.
		Let $(u_n), (s_n), (w_n)$, $(\gamma_n)$ and $(\check{p}_n)$ be as in Theorem \ref{th:pivot} and let $(g_n)$ be as in Lemma \ref{lem:descri-bloc-pivot}.
		
		Let us construct the sequence $(g_n)$.
		Let $\nu_{K^c}$ be the normalized restriction of $\nu$ to $K^c = \GL(E) \setminus K$  $\nu_{K^c} = \nu$ otherwise.
		Consider a random sequence $(g'_n)$ such that $(u_n), (s_n), (w_n), (g'_n) \sim \tilde\kappa^{\odot\NN}\otimes \tilde\nu_s^{\odot\NN} \otimes \mathcal{G}_{\alpha}^{\otimes\NN} \otimes \nu_{K^c}^{\otimes\NN}$ and let $g_n = \gamma_n$ when $\gamma_n \notin K$ \ie $g_n = u_n$ when $u_n \notin K$ and $\max\{k\,|\,m\overline{w}_k \le n\}$ is even and $g_n = g'_n$ otherwise.		  
 		Let us show that the conclusions of Theorem \ref{th:pivot} are satisfied for $p = \check{p}$
		
		Point \eqref{pivot:gammainK} in Theorem \ref{th:pivot} is a direct consequence of the construction of $g$.
		Let us check that $(g_n) \sim \mu^{\otimes\NN}$ and that $(g_n)$ is independent of the joint data of $(s_n)$ and $(w_n)$.
		Because of the product structure, $(g_n)$ is trivially independent of $(s_n)$ so all we need to check is that for all $A \subset \mathrm{GL}(E)$ such that $A \notin K$, we have $\PP(u_n \in A \,|\, u_n \notin K) = \nu_{K^c}(A)$.
		For that, write $i_n$ for the event ($\max\{k\,|\,m\overline{w}_k \le n\}$ is even), $(i_n)$ is independent of $(u_n)$ so $\PP(u_n \in A \,|\, u_n \notin K) = \PP(u_n \in A, i_n \,|\, N(u_n) > B, i_n) = \PP(\gamma_n \in A \,|\, \gamma_n \notin K) = \nu_{K^c}(A)$.
		
		Point \eqref{pivot:markov} in Theorem \ref{th:pivot} is a consequence of points \eqref{item:indep} and \eqref{item:tjr-piv} in Theorem \ref{th:ex-piv}. Indeed, by \eqref{item:indep} in Theorem \ref{th:ex-piv}, for all $k \ge 0$, the sequence $(\widetilde\gamma^{\check{p}}_{2 j + 2k + 2})_{j \ge 0}$ is independent of $(\widetilde\gamma^p_{2j})_{0 \le j \le k}$. 
		Moreover, by \eqref{item:tjr-piv} in Theorem \ref{th:ex-piv}, the conditional distribution of $(\widetilde\gamma^p_{2j + 1})_{0 \le j < k}$ with respect to the joint data of $(\widetilde\gamma^p_{2j})_{0 \le j}$ only depend on $(\widetilde\gamma^p_{2j})_{0 \le j \le k}$. 
		Hence, the joint data of $(\widetilde\gamma^p_{2j + 1})_{0 \le j < k}$ and $(\widetilde\gamma^p_{2j})_{0 \le j \le k}$ is independent of $(\widetilde\gamma^{\check{p}}_{2 j + 2k + 2})_{j \ge 0}$.
		By \eqref{item:indep} in Theorem \ref{th:ex-piv} again, the distribution of $(\widetilde\gamma^{\check{p}}_{2 j + 2k + 2})_{j \ge 0}$ does not depend on $k$.
		By point \eqref{item:tjr-piv} in Theorem \ref{th:ex-piv} again, the conditional distribution of $(\widetilde\gamma^p_{2j + 2k + 3})_{0 \le j}$ with respect to $(\widetilde\gamma^p_{n})_{n \in \NN \setminus (2 \NN + k + 3)}$ is given by a function of $(\widetilde\gamma^{\check{p}}_{2 j + 2k + 2})_{j \ge 0}$.
		Hence, the joint data of $(\widetilde\gamma^{\check{p}}_{2 j + 2k + 2})_{j \ge 0}$ and $(\widetilde\gamma^p_{2j + 2k + 3})_{0 \le j}$ is independent of $(\widetilde\gamma^p_{n})_{0 \le n \le 2k}$ and its distribution does not depend on $k$.
	\end{proof}
	
	To conclude the present reminder section, let us also remind the reader that we know how to compute the law of the sequence $(\check{p}_k)$.
	Let $(t_n) \sim \eta^{\otimes\NN}$ with $\eta = \frac{2}{3} \delta_1 + \sum_{k = 1}^{+\infty} {4^{-k}}\delta_{-k}$, let $X_0 = 0$ and for all $n \ge 0$, let $X_{n + 1} = (X_{n} + t_n)^+$.
	by construction of the sequence $(p_k)$ in the proof of \cite[Theorem~4.7]{moi} and by \cite[Lemma~4.12]{moi}, the law of the non-decreasing sequence $(\frac{\overline{p}_{2k + 1}- 1}{2})_{k \ge 0}$ is the same as the law of the non-decreasing sequence $(\max\{j \in \NN\,|\, X_j = k\})_{k \ge 0}$.
	Point \eqref{item:exp-p} in Theorem \ref{th:ex-piv} follows from that fact.
	We also know from point \ref{item:law-of-v} that $(\frac{v_{4k}-1}{2})_{k\in\NN} \sim \mathcal{G}_{1/3}^{\otimes\NN}$ and from point \eqref{item:indep} in Theorem \ref{th:ex-piv} that the data of $(p_k)_k$ is independent of the data of $(v_k)_k$.
	This fully determines the law of $(\check{p})_k$ though we do not have a nice analytical formula for it.
	From the law of large numbers and the fact that $\eta$ has mean $2/9$ and from the fact that $(p_{2k + 2})$ is i.i.d, we deduce that $\EE(p_2) = 8$ and therefore we can compute:
	\begin{equation*}
		\EE(\check{p}_2) = \frac{\EE(p_2)+ 1}{2}\EE(\hat{v}_0) + \frac{\EE(p_2) - 1}{2}\EE(\hat{v}_1) = \frac{9m}{2}\left(\frac52 \cdot\frac{1 - \alpha}{\alpha} + \frac{3}{2}\right) + \frac{7}{2}m.
	\end{equation*}

	\subsection{Study of the tail of the blocks} \label{section:tail}

	This section is dedicated to the proof of Theorem \ref{th:Delta-unif}. 
	We state the intermediate Lemma in a general setting as we will need it in future works to show that the results of the present paper hold for Gromov's relatively hyperbolic groups as well as other groups.
	
	We call metrizable semi-group a second countable Hausdorff-separated topological space endowed with a continuous and associative composition map.
	We call metrizable group a metrizable semi-group with a unit that admits a continuous inverse map.

	\begin{Lem}[Gain of moment]\label{lem:Delta}
		Let $\Gamma$ be a metrizable semi-group. 
		Let $N : \Gamma \to \RR_{\ge 0}$ be a continuous and sub-additive map.
		Let $V$ be a Banach space and let $\kappa : \Gamma \to V$ be a continuous and $N$-almost additive map \ie such that for all $g,h \in \Gamma$, we have:
		\begin{equation}\label{kappa-N}
			\|\Delta\kappa(g,h)\| \le \min\{N(g), N(h)\}.
		\end{equation}
		Let $B$ be a real constant.
		Let $0 \le a \le b$ be random integers, and let $p := b - a$.
		Let $\nu$ be a probability measure on $\Gamma$ and let $(g_k)_{k \ge 0} \sim \nu^{\otimes\NN}$ be a random i.i.d. sequence.
		Assume that the data of $(a, b)$ is independent of the data of $(g_k)_{k \ge 0}$.
		Let $(\gamma_k)_{0 \le k}$ be a random sequence such that $\gamma_k = g_k$ or $N(\gamma_k) \le B$ for all $k$.
		Let $R := \sum_{k = a}^{b - 1} N(\gamma_k) - \max_{a \le k < b} N(\gamma_k)$.
		Then we have:
		\begin{equation}\label{eq:Delta-R}
			\|\Delta\kappa(\widetilde\gamma_{a, b})\| \le R.
		\end{equation}
		Moreover, for all $t \ge 0$, we have:
		\begin{equation}\label{eq:sum-square}
			\PP(R > t) \le \PP(B(p - 1) > t) + \sum_{k = 2}^{\lfloor t / B\rfloor} \binom{k}{2} \PP(p = k)  \PP\left(N(g_0) > \frac{t}{k-1}\right)^2.
		\end{equation}
	\end{Lem}
	
	Before giving the proof of Lemma \ref{lem:Delta}, let us remind the reader that for $\Gamma = \mathrm{GL}(E)$, the maps $\kappa: g \mapsto \log\|g\|$ and $N = g \mapsto \log\|g\| + \log\|g^{-1}\|$, are measurable and satisfy \eqref{kappa-N}.

	\begin{proof}
		Assume $a$ and $b$ to be fixed and let $p = b - a$.
		Note that we have:
		\begin{equation}\label{alt-R}
			R = \min_{a \le i < b} \sum_{\substack{a \le j < b \\ j \neq i}} N(\gamma_j),
		\end{equation}
		The minimum being reached for all $k$ such that $N(\gamma_k) = \max_{0 \le i < p} N(\gamma_i)$.
		Let $a \le i < b$.
		By \eqref{kappa-N}, we have for all $i < j < b$
		\begin{equation*}
			\|\Delta\kappa(\gamma_{i,j}, \gamma_j)\| \le N(\gamma_j).
		\end{equation*}
		By the same argument, we have for all $a \le j < i$:
		\begin{equation*}
			\|\Delta\kappa(\gamma_j,\gamma_{j+1, b})\| \le N(\gamma_j).
		\end{equation*}
		Formally, by a telescopic sums argument, we have:
		\begin{equation*}
			\Delta\kappa(\widetilde\gamma_{a,b}) = \sum_{j = a}^{i - 1} \Delta\kappa(\gamma_j,\gamma_{j+1, b}) + \sum_{j = i + 1}^{b - 1} \Delta\kappa(\gamma_{i,j}, \gamma_j),
		\end{equation*}
		the first sum being empty when $a = i$ and the second sum being empty when $i = b-1$.
		Hence, by triangular inequality, we have:
		\begin{equation*}
			\|\Delta\kappa(\widetilde\gamma_{a,b})\| \le \sum_{a \le j < b, j \neq i}N(\gamma_j)
		\end{equation*}
		If we take $I$ to minimize \eqref{alt-R}, we have:
		\begin{equation*}
			\|\Delta\kappa(\widetilde\gamma_{a,b})\| \le R.
		\end{equation*}
		This proves \eqref{eq:Delta-R}.
		
		Now we prove \eqref{eq:sum-square}.
		Let $B = \max N(K)$. 
		Then for all $i \notin I$, we have $N(\gamma_i) \le B$.
		Let $t \ge (w - 1) B$ and assume that $R \ge t$.
		Then there exist at least two indices $0 \le i < j < w$ such that $N(\gamma_i) > \frac{t}{w-1}$ and $N(\gamma_j) > \frac{t}{p-1}$.
		Otherwise, there would exist $i < p$ such that for all $j \neq i$, we have $\kappa_1(\gamma_j) \le \frac{t}{p-1}$ and, we would have $R \le t$.
		If $R \ge t \ge (p-1) B$, then we necessarily have $a \le i < j < b$ such that $N(\gamma_i) > \frac{t}{p-1}$ and $N(\gamma_i) > \frac{t}{p-1}$. Moreover, we have:
		\begin{equation*}
			\forall i < j \in I,\; \PP\left(N(\gamma_i) > \frac{t}{p-1} \cap N(\gamma_j) > \frac{t}{p-1} \,\middle|\, w\right) = N_*\nu(t / (p-1), + \infty)^2
		\end{equation*}
		Therefore, for all $t \ge (p-1)B$, we have:
		\begin{equation*}
			\PP\left(R > t \,|\, p\right) \le \binom{p}{2}\nu(t / (p-1), + \infty)^2.
		\end{equation*}
		We conclude by taking the sum over all possible values of $p$.
	\end{proof}

	Now we prove Theorem \ref{th:Delta-unif} using Theorem \ref{th:pivot} and Lemma \ref{lem:Delta}.
	Let us first introduce some useful tools for the proof.
	
	\begin{Def}[Subordinated word]
		Let $\Gamma$ be a monoid \ie a semi-group with a unit element.
		We say that a word $\tilde{\gamma} = (\gamma_k)_{0 \le k < L(\tilde\gamma)} \in \widetilde\Gamma$ is subordinated to a word $\tilde{g} = (g_k)_{0 \le k < L(\tilde{g})} \in \widetilde\Gamma$ and write $\tilde\gamma \preceq \tilde{g}$ if there exist a family $0 \le i_0 \le \dots \le i_{L(\tilde\gamma)} \le L(\tilde{g})$ such that for all $0 \le k < L(\tilde\gamma)$, we have $\gamma_k = g_{i_k, i_{k+1}}$.
	\end{Def}
	
	\begin{Prop}\label{prop:subord}
		Let $\Gamma$ be a monoid, and let $\kappa = \Gamma \to \RR$ be sub-additive. 
		Then $\Delta\kappa$ is a non-increasing map for the subordination relation on $\tilde\Gamma$.
	\end{Prop}
	
	\begin{proof}
		We rely on the fact that $\Delta\kappa \le 0$, which is a direct consequence of the sub-additivity of $\kappa$.
		Let $\tilde\gamma \preceq \tilde{g} \in \widetilde\Gamma$.
		Let $l \ge 0$ and let $0 \le i_0 \le \dots \le i_{L(\tilde\gamma)} \le L(\tilde{g})$ be such that for all $0 \le k < L(\tilde\gamma)$, we have $\gamma_k = g_{i_k, i_{k+1}}$.
		We write $i_{-1} = 0$ and $i_{L(\tilde\gamma) + 1}= L(\tilde g)$.
		Formally, we have:
		\begin{equation}
			\Delta\kappa(\tilde{g}) = \Delta\kappa(g_{0, i_0}, g_{i_0, i_l}, g_{i_L(\tilde\gamma), L(\tilde g)}) + \Delta\kappa(\tilde\gamma) + \sum_{k = -1}^{L(\tilde{\gamma})}\Delta\kappa(\widetilde{g}_{i_k, i_k + 1}).
		\end{equation}
		So $\Delta\kappa(\tilde{g})$ is a sum of non positive terms, one of them being $\Delta\kappa(\tilde\gamma)$ and therefore $\Delta\kappa(\tilde{g}) \le \Delta\kappa(\tilde\gamma)$.
	\end{proof}
	
	\begin{Lem}\label{lem:step}
		Let $(p_n)_{n \ge 0}\in\NN_{\ge 1}^\NN$ be independent random variables.
		For all $n \in \NN$, we write $\lfloor n\rfloor_p = \max\{\overline{p}_k \,|\, k \in \NN, \overline{p}_k  \le n\}$ and  $\lceil n \rceil_p = \min\{\overline{p}_k\,|\,k \in \NN, \overline{p}_k  \ge n\}$. 
		Then for all $n, t \in \NN$, we have:
		\begin{equation}
			\PP\left(\lceil n \rceil_p - \lfloor n\rfloor_p = t\right) \le t \max_{0 \le k} \PP(p_k = t).
		\end{equation}
	\end{Lem}
	
	\begin{proof}
		Let $n$ be fixed and write $l_n := \max\{k \in \NN \,|\,\overline{p}_k  \le n\}$.
		Note that $l_n \le n$ because we assumed the measures $\eta_i$ to be supported on $\NN_{\ge 1}$.
		Note also that for all $n$, we have $\lceil n \rceil_p - \lfloor n\rfloor_p = 0$ or $\lceil n \rceil_p - \lfloor n\rfloor_p = p_{l_n}$.
		We have:
		\begin{equation*}
			(p_{l_n} = t) = \bigcup_{k \le n} (p_k \ge t) \cap (k = l_n)
		\end{equation*}
		Note also that for all $k,n$ , we have $k  = l_n$ if and only if $\overline{p}_k \le n$ and $p_k > n - \overline{p}_n$ so we have:
		\begin{align*}
			(p_{l_n} = t) & \subset \bigcup_{k = 0}^n (p_k \ge t) \cap (\overline{p}_k \le n) \cap (p_k > n - \overline{p}_n) \\
			& \subset \bigcup_{k \le j \le n} (p_k = t) \cap (\overline{p}_k = j) \cap (p_k > n - j).
		\end{align*}
		Now for all $k \le j \le n$, the events $(\overline{p}_k = j)$ and $(p_k = t)$ are independent by assumption.
		Therefore, we have:
		\begin{align*}
			\PP(p_{l_n} = t) & = \sum_{k \le j \le n} \PP(\overline{p}_k = j) \PP(p_k = t) \mathds{1}_{n-j < t} \\
			& = \sum_{j = n - t + 1}^{n} \sum_{k = 0}^j \PP(\overline{p}_k = j) \PP(p_k = t)
		\end{align*}
		For all $j$, by positivity of $p$, there is at most one integer $k$ such that $p_k = j$, so we have $\sum_{k = 0}^j \PP(\overline{p}_k = j) = \PP(\exists k, \overline{p}_k = j) \le 1$.
		Therefore:
		\begin{equation}
			\PP(p_{l_n} = t) \le \sum_{j = n - t + 1}^{n} \max_{0 \le k \le j} \PP(p_k = t) \le t \max_{0 \le k} \PP(p_k = t). \qedhere
		\end{equation}
	\end{proof}

	Now we prove the following result which says that Points \eqref{pivot:gammainK}, \eqref{pivot:exp}, \eqref{pivot:indep} and \eqref{pivot:almost-ad} in Theorem \ref{th:pivot} imply Theorem \ref{th:Delta-unif}.
	One may note that the result still holds when only assuming $\kappa$ to be $N$-almost additive.
	The sub-additivity assumption is however very convenient for the proof and true in most practical applications.
	
	\begin{Lem}\label{lem:delta-pivot}
		Let $\Gamma$ be a metrizable monoid. Let $N : \Gamma \to \RR_{\ge 0}$ be a continuous almost additive map and let $\kappa : \Gamma \to \RR$ be a continuous, $N$-almost additive and sub-additive map.
		Let $K \subset \Gamma$ be compact.
		Let $(p_n)_{n \ge 0} \in \NN_{\ge 1}^\NN$ be a sequence of independent random variables and let $(g_n) \in \Gamma^{\NN}$ be a random i.i.d. sequence, that is independent of $(p_n)$.
		Let $(\gamma_n) \in\Gamma^\NN$ be a random sequence such that for all $n$, we have $\gamma_n = g_n$ or $\gamma_n \in K$.
		Assume that there exist non-random constants $C, \beta > 0$ such that for all $n, t \ge 0$, we have $\PP(p_n = t) \le Ce^{-\beta t}$ and such that for all $0 \le i  < j < k$, we have:
		\begin{equation}
			|\Delta\kappa(\gamma^p_{i,j}, \gamma^p_{j,k})| \le C.
		\end{equation}
		Assume also that the law of $\PP(N(g_0) > 0) > 0$.
		Then there exist constants $C'', \beta'' > 0$ such that for all $t \ge 0$, and for all $0 \le i < j < k$, we have:
		\begin{equation}\label{bound-delta'}
			\PP\left(\Delta\kappa(\gamma_{i,j}, \gamma_{j,k}) > t \right) \le \sum_{k = 1}^{+\infty} C''e^{-\beta'' k} \PP(N(g_0) > t/k)^2.
		\end{equation}
	\end{Lem}
	
	\begin{proof}
		For all $n \in \NN$, we write $\lfloor n\rfloor_p = \max\{k\,|\, \overline{p}_k  \le n\}$ and  $\lceil n \rceil_p = \min\{k\,|\, \overline{p}_k  \ge n\}$. 
		Let $0 \le i < j < k$.
		Formally, when $\lceil i \rceil_p \le \lfloor j \rfloor_p$ and $\lceil j \rceil_p\le \lfloor k \rfloor_p$ we have:
		\begin{multline}\label{decomp-ijk}
			\Delta\kappa\left(
			\widetilde\gamma_{\lfloor i \rfloor_p, \lceil i \rceil_p} \odot
			(\gamma_{\lceil i \rceil_p, \lfloor j \rfloor_p}) \odot
			\widetilde\gamma_{\lfloor j \rfloor_p, \lceil j \rceil_p} \odot
			(\gamma_{\lceil j \rceil_p, \lfloor k \rfloor_p}) \odot 
			\widetilde\gamma_{\lfloor k \rfloor_p, \lceil k \rceil_p}\right) \\ 
			= \Delta\kappa(
			\gamma_{\lfloor i \rfloor_p, \lceil i \rceil_p},
			\gamma_{\lceil i \rceil_p, \lfloor j \rfloor_p}, 
			\gamma_{\lfloor j \rfloor_p, \lceil j \rceil_p},
			\gamma_{\lceil j \rceil_p, \lfloor k \rfloor_p}, 
			\gamma_{\lfloor k \rfloor_p, \lceil k \rceil_p}) \\ 
			+ \Delta\kappa\left(
			\widetilde\gamma_{\lfloor i \rfloor_p, \lceil i \rceil_p} \right)
			+ \Delta\kappa\left(
			\widetilde\gamma_{\lfloor j \rfloor_p, \lceil j \rceil_p} \right)
			+ \Delta\kappa\left( 
			\widetilde\gamma_{\lfloor k \rfloor_p, \lceil k \rceil_p} \right).
		\end{multline} 
		By assumption, we have:
		\begin{equation}
			\left|\Delta\kappa(
			\gamma_{\lfloor i \rfloor_p, \lceil i \rceil_p},
			\gamma_{\lceil i \rceil_p, \lfloor j \rfloor_p}, 
			\gamma_{\lfloor j \rfloor_p, \lceil j \rceil_p},
			\gamma_{\lceil j \rceil_p, \lfloor k \rfloor_p}, 
			\gamma_{\lfloor k \rfloor_p, \lceil k \rceil_p})\right| \le 4 C
		\end{equation}
		Moreover, by Proposition \ref{prop:subord}, we have:
		\begin{equation*}
			\Delta\kappa\left(
			\widetilde\gamma_{\lfloor i \rfloor_p, \lceil i \rceil_p} \odot
			(\gamma_{\lceil i \rceil_p, \lfloor j \rfloor_p}) \odot
			\widetilde\gamma_{\lfloor j \rfloor_p, \lceil j \rceil_p} \odot
			(\gamma_{\lceil j \rceil_p, \lfloor k \rfloor_p}) \odot 
			\widetilde\gamma_{\lfloor k \rfloor_p, \lceil k \rceil_p}\right) \le \Delta\kappa(\gamma_{i,j}, \gamma_{j,k}).
		\end{equation*}
		In conclusion, we have:
		\begin{equation}\label{upper-bound-delta}
			\Delta\kappa(\gamma_{i,j}, \gamma_{j,k}) \le 4 C 
			+ |\Delta\kappa\left( \widetilde\gamma_{\lfloor i \rfloor_p, \lceil i \rceil_p} \right)|
			+ |\Delta\kappa\left( \widetilde\gamma_{\lfloor j \rfloor_p, \lceil j \rceil_p} \right)|
			+ |\Delta\kappa\left( \widetilde\gamma_{\lfloor k \rfloor_p, \lceil k \rceil_p} \right)|.
		\end{equation}
		In case $\lceil i \rceil_p > \lfloor j \rfloor_p$, we have $\lceil i \rceil_p > j$ and therefore $\lfloor i \rfloor_p = \lfloor j \rfloor_p$ and $\lceil i \rceil_p = \lceil j \rceil_p$ and if we assume that $\lceil j \rceil_p \le \lfloor k \rfloor_p$ then formally, we have:
		\begin{multline}\label{decomp-ik}
			\Delta\kappa\left(
			\widetilde\gamma_{\lfloor i \rfloor_p, \lceil i \rceil_p} \odot
			(\gamma_{\lceil i \rceil_p, \lfloor k \rfloor_p}) \odot 
			\widetilde\gamma_{\lfloor k \rfloor_p, \lceil k \rceil_p}\right)
			= \Delta\kappa(
			\gamma_{\lfloor i \rfloor_p, \lceil i \rceil_p},
			\gamma_{\lceil i \rceil_p, \lfloor k\rfloor_p},  
			\gamma_{\lfloor k \rfloor_p, \lceil k \rceil_p}) \\ 
			+ \Delta\kappa\left( \widetilde\gamma_{\lfloor i \rfloor_p, \lceil i \rceil_p} \right)
			+ \Delta\kappa\left( \widetilde\gamma_{\lfloor k \rfloor_p, \lceil k \rceil_p} \right).
		\end{multline} 
		The same formula holds when $\lceil j \rceil_p > \lfloor k \rfloor_p$, and $\lceil i \rceil_p \le \lfloor j \rfloor_p$.
		Moreover, in this case, the word $(\gamma_{i,j}, \gamma_{j,k})$ is subordinated to the word $\widetilde\gamma_{\lfloor i \rfloor_p, \lceil i \rceil_p} \odot
		(\gamma_{\lceil i \rceil_p, \lfloor k \rfloor_p}) \odot 
		\widetilde\gamma_{\lfloor k \rfloor_p, \lceil k \rceil_p}$ so we have:
		\begin{equation*}
			\Delta\kappa(\gamma_{i,j}, \gamma_{j,k}) \le 2 C 
			+ |\Delta\kappa\left( \widetilde\gamma_{\lfloor i \rfloor_p, \lceil i \rceil_p} \right)|
			+ |\Delta\kappa\left( \widetilde\gamma_{\lfloor k \rfloor_p, \lceil k \rceil_p} \right)|.
		\end{equation*} 
		and therefore \eqref{upper-bound-delta} still holds.
		When we assume that $\lceil j \rceil_p > \lfloor k \rfloor_p$, and $\lceil i \rceil_p > \lfloor j \rfloor_p$, we have $ = \lfloor i \rfloor_p = \lfloor j \rfloor_p = \lfloor k \rfloor_p$ and $\lceil i \rceil_p = \lceil j \rceil_p = \lceil k \rceil_p$ so the word $(\gamma_{i,j}, \gamma_{j,k})$ is subordinated to the word $\widetilde\gamma_{\lfloor i \rfloor_p, \lceil i \rceil_p}$ and \eqref{upper-bound-delta} still holds.
		
		Let us now explicit a uniform probabilistic bound on $|\Delta\kappa\left( \widetilde\gamma_{\lfloor n \rfloor_p, \lceil n \rceil_p} \right)|$ for all $n \ge 0$.
		The data of $\lfloor n \rfloor_p, \lceil n \rceil_p$ is determined by $(p_k)_{k \ge 0}$ so it is independent of the data of $(g_n)_{n \ge 0}$.
		By Lemma \ref{lem:Delta}, for $B = \max N(K)$ and for all $t \ge 0$, we have:
		\begin{multline*}
			\PP(|\Delta\kappa\left( \widetilde\gamma_{\lfloor n \rfloor_p, \lceil n \rceil_p} \right)| > t) \le \PP(B (\lceil n \rceil_p - \lfloor n \rfloor_p - 1) > t) \\
			+ \sum_{k = 2}^{\lfloor t / B \rfloor} \binom{k}{2}\PP(\lceil n \rceil_p - \lfloor n \rfloor_p = k) \PP\left(N(g_0) > \frac{t}{k - 1}\right)
		\end{multline*}
		By Lemma \ref{lem:step}, we have $\PP(\lceil n \rceil_p - \lfloor n \rfloor_p = k) \le k Ce^{-\beta k}$ for all $k$ and all $n$ and therefore:
		\begin{equation*}
			\PP(\lceil n \rceil_p - \lfloor n \rfloor_p - 1 > t / B) \le \lceil 1 + t / B \rceil \frac{C}{e^{\beta} - 1} e^{-\beta \lceil t / B \rceil} + \frac{C}{(e^{\beta} - 1)^2} e^{-\beta \lceil t / B \rceil}.
		\end{equation*} 
		Let $M \ge 1$ be such that $\PP(N(g_0) > B / M) \ge 1 / M$, such a $M$ exist because $\lim_{M \to + \infty} \PP(N(g_0) > B / M) - 1 / M = \PP(N(g_0) > 0)$, which is positive by assumption. 
		Then, for all $k \ge M t / B + 1$, we have $\PP\left(N(g_0) > \frac{t}{k - 1}\right) \ge \PP(N(g_0) > B / M) \ge 1 / M$.
		Now let $C', \beta' > 0$ be such that for all $k \ge 2$, we have $C e^{-\beta k} k^2(k-1)/2 \le C'e^{-\beta k'}$ and for all $t \ge 0$, we have:
		\begin{equation*}
			\frac{C'}{e^{\beta'} - 1} e^{-\beta \lceil Mt / B\rceil} \frac{1}{M} \ge \lceil 1 + t / B \rceil \frac{C}{e^{\beta} - 1} e^{-\beta \lceil t / B \rceil} + \frac{C}{(e^{\beta} - 1)^2} e^{-\beta \lceil t / B \rceil}.
		\end{equation*}
		By exponential comparison Theorem, we know that there exist such a $C'$ for all $\beta' < \beta$.
		Then we have:
		\begin{equation*}
			\PP(|\Delta\kappa\left( \widetilde\gamma_{\lfloor n \rfloor_p, \lceil n \rceil_p} \right)| > t) \le \sum_{k \ge 1} C'e^{-\beta'k} \PP(N(g_0) > t / k)
		\end{equation*}
		By \eqref{upper-bound-delta}, we have for all $t$:
		\begin{align}
			\PP\left(|\Delta\kappa(\gamma_{i,j}, \gamma_{j,k})| > t\right) & \le \sum_{n \in \{i,j,k\}} \PP(|\Delta\kappa\left( \widetilde\gamma_{\lfloor n \rfloor_p, \lceil n \rceil_p} \right)| > (t - 4 C) / 3) \\
			& \le \sum_{k = 1}^{+\infty} 3C'e^{-\beta'k} \PP\left(N(g_0) > \frac{t - 4 C}{3 k}\right)
		\end{align}
		For $t > 8 C$, we have $\frac{t - 4 C}{3 k} \ge \frac{t}{6 k}$ and:
		\begin{equation*}
			 \forall t > 8 C, \;\sum_{k = 1}^{+ \infty} 3C'e^{-\beta' k} \PP\left(N(g_0) > \frac{t}{6 k}\right) = \sum_{k \in 6\NN_{\ge 1}} 3C'e^{-\beta' k / 6} \PP\left(N(g_0) > \frac{t}{k}\right).
		\end{equation*}
		Hence, we have:
		\begin{equation}
			\PP\left(|\Delta\kappa(\gamma_{i,j}, \gamma_{j,k})| > t\right) \le \sum_{k = 1}^{+ \infty} 3C'e^{-\beta' k / 6} \PP\left(N(g_0) > \frac{t}{k}\right)
		\end{equation}
		For $t \le 8C$, and for all $k \ge 8 C M / B$, we have $\PP(N(g_0) > t/k) \ge 1/M$ and therefore:
		\begin{equation*}
			\sum_{k = \lceil 8 C M / B\rceil}^{+ \infty} 3C'e^{-\beta' k / 6} \PP\left(N(g_0) > \frac{t}{k}\right) \ge  \frac{3C'}{1 - e^{-\beta' / 6}} \frac{e^{-\beta \lceil 8 C M / B\rceil / 6}}{M}
		\end{equation*}
		We can always assume that $C' \ge M(1 - e^{-\beta' / 6})e^{\beta \lceil 8 C M / B\rceil / 6} / 3$, in that case, we have:
		\begin{equation*}
			\forall t \le 8C, \; \sum_{k = 1}^{+ \infty} 3C'e^{-\beta' k / 6} \PP\left(N(g_0) > \frac{t}{k}\right) \ge 1 \ge \PP\left(|\Delta\kappa(\gamma_{i,j}, \gamma_{j,k})| > t\right)
		\end{equation*}
		In conclusion, we have \eqref{bound-delta'} for $C'' = 3 C'$ and $\beta'' = \beta' / 6$.
	\end{proof}

	\begin{proof}[Proof of Theorem \ref{th:Delta-unif}]
		Let $K \subset\Gamma$, $m \in \NN_{\ge 1}$ and $C$ be as in Theorem \ref{th:pivot}.
		Let $(\gamma_n) \sim \nu^{\otimes\NN}$, $(g_n)$ and $(p_n)$ be random sequences defined on the same probability space as in Theorem \ref{th:pivot}.
		To apply Lemma \ref{lem:delta-pivot}, we need to show first that $\PP(N(g_0) > 0) > 0$ and that there are constant $C, \beta > 0$ such that for all $n$, we have $\PP(p_n = t) \le C e^{-\beta t}$ and that the $p_n$'s are independent. 
		The later is a direct consequence of points \eqref{pivot:markov} and \eqref{pivot:exp} in theorem \ref{th:pivot}.
		
		By \eqref{pivot:markov} (looking only at the sequence of lengths), the sequence $(p_{2k})_{k \ge 0}$ is independent (indeed, the data of $(p_{2 n})_{n > k}$ is independent of $(p_{2 n})_{n \le k}$ for all $k$) and since $p_{2k + 1} = m$ for all $k$, it is non-random and therefore independent of everything (even itself) so the sequence $(p_{k})_{k \ge 0}$ is independent.
		By \eqref{pivot:exp} all the $p_k$-s have a finite exponential moment so there exist constants $C_k, \beta_k > 0$ such that $\PP(p_k = t) \le C_k e^{-\beta_k C}$.
		Moreover, by \eqref{pivot:markov} again, the sequence $(p_{2k})_{k \ge 1}$ is identically distributed so all $p_k$ shares its law with $p_0$, $p_1$ or $p_2$ so we have for $C = \max\{C_0, C_1, C_2\}$ and $\beta_k = \min\{\beta_0, \beta_1, \beta_2\}$, we have $\PP(p_n = t) \le C e^{-\beta t}$ for all $n$.
		
		Now let us prove that \eqref{bound-delta'} implies \eqref{Delta-unif}.
		If we assume $N(\gamma_0)$ to be unbounded, then for all $t \ge B = \max N(K)$, we have $\PP(N(\gamma_0) > t) \ge \PP(N(g_0) > t) \PP(N(\gamma_0) \notin K)$ and for all $t < B$, we have $\PP(N(\gamma_0) > t) \ge \PP(N(\gamma_0) > B)$.
		Therefore, we have $\PP(N(\gamma_0) > t) \ge \PP(N(\gamma_0) > B) \PP(N(g_0) > t)$ so \eqref{bound-delta'} implies \eqref{Delta-unif} with $C = C'' \PP(N(\gamma_0) > B)^{-2}$ a,d $\beta = \beta''$.
		Moreover, we have $\PP(N(g_0) > B) > 0$ so we may apply Lemma \ref{lem:delta-pivot}.
		
		If we instead assume that $\PP(N(g_0) > B) = 0$ and that $N_*\nu$ is non-degenerate, then we use \eqref{decomp-ijk} or \eqref{decomp-ik} which tells us that:
		\begin{equation*}
			\Delta\kappa(\gamma_{i,j}, \gamma_{j,k}) \le 4C + B\left(\lceil i \rceil_p - \lfloor i \rfloor_p  \lceil i \rceil_p + \lceil j \rceil_p - \lfloor j \rfloor_p  \lceil i \rceil_p +\lceil k \rceil_p - \lfloor k \rfloor_p  \lceil i \rceil_p\right)
		\end{equation*}
		By Lemma \ref{lem:step}, this has a bounded exponential moment. 
		We conclude using that fact that $N_*\nu$ is non degenerate, \ie there exists $\alpha > 0$ such that $\PP(N(\gamma_0) > \alpha) \ge \alpha$ so on the right hand side of \eqref{Delta-unif}, we have $\sum Ce^{-\beta k} \PP(N(\gamma_0) > t/k)^2 \ge C \alpha^2 e^{-\beta t /\alpha}$.
		If $N_*\nu$ is degenerate then $\Delta\kappa(\gamma_{i,j}, \gamma_{j,k}) = 0$ almost surely for all $i < j<k$ and \eqref{Delta-unif} is trivial.
		Note that this last case does not occur when $\nu$ is proximal.
	\end{proof}
	
	\section{Study of almost additive processes}\label{section:almost-additive}
	
	A consequence of Theorem \ref{th:Delta-unif} is that the random process $(\kappa(\gamma_{m,n}))_{0 \le m \le n}$ is almost additive in a way that we define in the first paragraph of the present section.
	To be able to use some of the result in further works, we study almost additive processes in an axiomatic way.
	For the moment, it is not clear whether or not the formulation in terms of almost additive processes yields result that could not already be proven using classical ergodic theoretic tools.
	We state results in full generality to avoid having to give different statements for $\kappa$ and $\dot \kappa$ and to be able to apply them in a broader context.
	
	\subsection{Definition and motivation}
	
	In the present section, we call non-additive process in an Abelian group $V$ a family $(S_{m,n})_{0 \le m < n}$ of elements of $V$.
	We assume $V$ to be endowed with a second countable $\sigma$-algebra.
	In the present section, we study non-additive processes in a second countable Hilbert space $V$, over $\RR$ endowed with a scalar product denoted by $\langle\cdot,\cdot\rangle$ and the associated norm $\|\cdot\|$.

	\begin{Def}[Mixing process]\label{def:mix}
		Let $S = (S_{m,n})_{0 \le m < n}$ be a random process.
		We say that $S$ is time-invariant if it has the same law as the time-shift of $S$, defined as $TS := (S_{m+1, n+1})_{0 \le m < n}$.
		We say that $S$ has no memory if for all $0 \le m$, the process $T^k S := (S_{m+j,m+k})_{0 \le j < k}$ is independent of the joint data of $(S_{j,k})_{0 \le j < k \le m}$.
		We say that $S$ is mixing if it has no memory and is invariant.
	\end{Def}
	
	For now, we have not used the group structure on $V$ to define the probabilistic notions of having no memory and being time invariant. 
	To be able to say anything meaningful on mixing processes, we assume the following moment condition.
	
	\begin{Def}[Almost-additivity]
		Let $S = (S_{m,n})_{0 \le m < n}$ be a random process in a Hilbert space $V$.
		Given $q > 0$ and $C \ge 0$, we say that $S$ is $C$-almost additive in $\mathrm{L}^q$ if for all $l < m < n$, we have:
		\begin{equation}\label{ellq-almost-add}
			\EE\left(\left\| S_{l,n} - S_{l,m} - S_{m,n} \right\|^q\right) \le C.
		\end{equation}
	\end{Def}
	
	From now, let us denote by $\Delta S(l,m,n)$ the quantity $S_{l,n} - S_{l,m} - S_{m,n}$.
	We remind that given a (random or not) sequence $(x_n)_{n \ge 0}\in V^\NN$, and $0 \le m \le n$, we write $x_{m,n} = \sum_{k = m}^{n-1} x_k$ for the partial sum of $x$, that way $(x_{m,n})$ is an additive process \ie $0$-almost additive.
	To avoid confusion without overloading the notations, we use upper case letters for general processes and lower case letters for those that come from sequences of random variables.
	Contrary to additive processes, the distribution law of an almost additive process $(S_{m,n})_{0 \le m \le n}$ can not be deduced from the distribution law of $(S_n)_{n \ge 0}$ or $(\overline{S}_n)_{n \ge 0}$.
	
	Instead of making the strong lack of memory assumption of Definition \ref{def:mix} we could have made the weaker assumption that the law of $S$ is ergodic for $T$ (meaning that all $T$-invariant events have probability $0$ or $1$).
	In \cite{K68}, Kingman proves a law of large numbers for ergodic processes on $\RR$ that are sub-additive \ie if $\EE(S_{0,1})< + \infty)$ and if $\Delta S(l,m,n) \le 0$ almost surely and for all $l < m < n$, then $S_{0,n} / n$ converges almost surely to a non-random limit in $[-\infty, + \infty)$.
	
	In the present work we make the mixing assumption, that is stronger than a simple ergodicity assumption but allows us to give more natural links between the probabilistic behaviour of $S_{0,n}$ and the constants $C$ and $q$ of \eqref{ellq-almost-add}.

	\subsection{Almost additivity of the Cartan projection}
	
	Let us now explain how Theorem \ref{th:Delta-unif} implies that the random family $(\Delta\kappa(\widetilde{\gamma}_{m,n}))_{0 \le m \le n}$, when $\nu$ is strongly irreducible and proximal or $(\Delta\dot\kappa(\widetilde{\gamma}_{m,n}))_{0 \le m \le n}$, when $\nu$ is totally irreducible is almost additive in $\mathrm{L}^q$ as soon as soon as $N(\gamma_0)^{q/2}$ is integrable.
	
	This follows from a simple argument of integration by parts.
	
	\begin{Lem}\label{lem:intergal-square}
		Let $x$ be a real non-negative random variable.
		For all $q > 0$, we have:
		\begin{equation}
			\int_{0}^{+\infty} q t^{q-1} \PP(x > t)^2 dt \le 2 \EE(x^{q/2})^2.
		\end{equation}
	\end{Lem}
	
	\begin{proof}
		Note that for all $t$, we have $q t^{q-1} \PP(x > t)^2 = 2 (t^{q / 2} \PP(x > t)) (\frac{q}{2} t^{q / 2 - 1} \PP(x > t))$.
		Therefore, by Hölder's inequality, we have:
		\begin{equation*}
			\int_{0}^{+\infty} q t^{q-1} \PP(x > t)^2 dt \le 2 \sup_{t \ge 0} t^{q / 2} \PP(x > t) \int_{0}^{\infty} \frac{q}{2} t^{q / 2 - 1} \PP(x > t) dt .
		\end{equation*}
		By Markov's inequality applied to $x^{q/2}$, we have $\PP(x > t) \le \EE(x^{q/2}) / t^{q/2}$ for all $t$, therefore, $\sup_{t \ge 0} t^{q / 2} \PP(x > t) \le \EE(x^{q/2})$. 
		Moreover, by integration by parts, we have $\int_{0}^{\infty} \frac{q}{2} t^{q / 2 - 1} \PP(x > t) dt = \EE(x^{q/2})$.
	\end{proof}

	\begin{Lem}\label{lem:ellq-almost}
		Let $\nu$ be a strongly irreducible and proximal probability distribution over $\mathrm{GL}(E)$ and let $(\gamma_n) \sim \nu^{\otimes \NN}$.
		Let $q > 0$ and assume that $\EE(N(\gamma_0)^{q / 2}) < + \infty$. 
		Let $C, \beta > 0$ be as in Theorem \ref{th:Delta-unif} and let $C' :=  2\EE(N(\gamma_0)^{q/2})^2 \sum_{k = 1}^{+ \infty} Ce^{-\beta k}k^{q}$.
		Then the processes $(\kappa(\widetilde{\gamma}_{m,n}))_{0 \le m \le n}$ and $(\Delta\kappa(\widetilde{\gamma}_{m,n}))_{0 \le m \le n}$ are $C'$-almost additive in $\mathrm{L}^q$.
	\end{Lem}
	
	\begin{proof}
		We want to show that for all $0 \le l < m < n$, we have $\EE(|\Delta\kappa(\gamma_{l,m}, \gamma_{m,n})|^q)\le C'$.
		Let $0 \le l < m < n$ be fixed.
		By Theorem \ref{th:Delta-unif}, for all $t \ge 0$, we have 
		\begin{equation*}
			\PP(|\Delta\kappa(\gamma_{l,m}, \gamma_{m,n})| > t) \le \sum_{k = 1}^{+ \infty} Ce^{-\beta k} \PP(N(\gamma_0) > t / k)^2. 
		\end{equation*}
		By integration by parts and permutation of the sum and integral:
		\begin{align*}
			\EE(|\Delta\kappa(\gamma_{i,j}, \gamma_{j,k})|^q) 
			& = \int_{0}^{+\infty} qt^{q-1}\PP(|\Delta\kappa(\gamma_{i,j}, \gamma_{j,k})| > t) dt \\
			& \le \sum_{k = 1}^{+ \infty} Ce^{-\beta k} \int_{0}^{+\infty} qt^{q-1} \PP(N(\gamma_0) > t / k)^2dt \\
			& \le \sum_{k = 1}^{+ \infty} Ce^{-\beta k}k^{q} \int_{0}^{+\infty} qt^{q-1} \PP(N(\gamma_0) > t)^2dt.
		\end{align*}
		Hence, by Lemma \ref{lem:intergal-square}, we have:
		\begin{equation*}
			\EE(|\Delta\kappa(\gamma_{l,m}, \gamma_{m,n})|^q) \le \sum_{k = 1}^{+ \infty} Ce^{-\beta k}k^{q} 2 \EE(N(\gamma_0)^{q/2})^2. \qedhere
		\end{equation*}
	\end{proof}
	
	To study the higher rank case we simply look at the exterior products representations.
	
	\begin{Lem}\label{lem:ellq-higher}
		Let $\nu$ be a totally irreducible probability distribution over $\mathrm{GL}(E)$ and let $(\gamma_n)\sim \nu^{\otimes\NN}$.
		Let $q > 0$ and assume that $\EE(N(\gamma_0)^{q/ 2}) < + \infty$.
		Then the processes $(\dot\kappa(\widetilde{\gamma}_{m,n}))_{0 \le m \le n}$ and $(\Delta\dot\kappa(\widetilde{\gamma}_{m,n}))_{0 \le m \le n}$ are $\mathrm{L}^q$ almost additive.
	\end{Lem}
	
	\begin{proof}
		Let $\Gamma_\nu$ be the smallest closed semi-group of full $\nu$-measure.
		For all $1 \le k \le d$ and for all $g \in \mathrm{GL}(E)$, we write $\overline\kappa_k(g) = \kappa(\bigwedge^k g)$.
		Then for all $k$, we have formally $\overline\kappa_k = \kappa_1 + \cdots + \kappa_k$.
		Let $\Theta(\nu)$ be the set of indices $1 \le k < d$ such that $\bigwedge^k \nu$ is proximal.
		Using the fact that $N \circ \bigwedge^k \le k N$ and by Lemma \ref{lem:ellq-almost} applied to $\bigwedge^{k}\nu$, for all $k \in \Theta(\nu)$, there exists a constant $C_k$ such that for all $0 \le l \le m \le n$, we have:
		\begin{equation}
			\EE\left(|\Delta\overline{\kappa}_k(\gamma_{l,m}, \gamma_{m,n})|^q\right) \le C_k. 
		\end{equation}
		Moreover, by sub-additivity of the determinant, we have $\Delta\overline{\kappa}_d = 0$ on $\mathrm{GL}(E)$.
		
		By \cite[Lemma~3.11]{moi} (credited to Guivarc'h and Raugi \cite{GR89} in the specific setting of products of invertible matrices), for all $k \notin \Theta(\nu)$, there exists a constant $B_k$ such that 
		$0 \le \kappa_{k}(g) - \kappa_{k+1}(g) \le B_k$ for all $g \in \Gamma_\nu$. 
		Therefore,  for all $g, h \in \Gamma_\nu$, we have:
		\begin{equation}\label{beka}
			|\Delta\kappa_k(g,h) - \Delta\kappa_{k + 1}(g,h)| \le 2B_k.
		\end{equation}
		For all $1 \le k \le d$, we write $\lceil k \rceil_{\Theta(\nu)}$ for the smallest element of $\theta(\nu)\cap \{d\}$ that is larger than $k$.
		By \eqref{beka}, for $B = \sum_{k \in\{1, \dots, d -1\}\setminus \Theta(\nu)} B_k$ and for all $0 \le j, k \le d$ such that $\lceil k \rceil_{\Theta(\nu)} = \lceil j \rceil_{\Theta(\nu)}$, we have formally on $\Gamma_\nu$:
		\begin{equation}
			|\Delta\kappa_k - \Delta\kappa_{j}| \le 2B.
		\end{equation}
		We use the convention $\overline{\kappa}_0 = \kappa_0(g) = 0$ for all $g$.
		Now let $j > \lceil 1 \rceil_{\Theta(\nu)}$, let $k = \lceil j \rceil_{\Theta(\nu)}$ and let $i = \lfloor j \rfloor_{\Theta(\nu)}$ be the maximal element of $\{0\} \cup \Theta(\nu)$ such that $i \le j$.
		We have $\Delta\overline\kappa_k - \Delta\overline\kappa_i = \sum_{j = i + 1}^k \Delta\kappa_{j}$
		so by a barycentric inequality, for all $i \le j \le k$, we have:
		\begin{equation*}
			\left| \Delta{\kappa}_j - \frac{1}{k - i}(\Delta\overline\kappa_k - \Delta\overline\kappa_i) \right|  \le 2B
		\end{equation*}
		and taking the partial sum:
		\begin{equation*}
			\left| \Delta\overline{\kappa}_j - \frac{j - i}{k}(\Delta\overline\kappa_k - \Delta\overline\kappa_i)\right|  \le 2B (k  - i).
		\end{equation*}
		and by triangular inequality:
		\begin{equation*}
			|\Delta\overline{\kappa}_j| \le |\Delta\overline{\kappa}_k| + |\Delta\overline{\kappa}_i| + 2B d.
		\end{equation*}
		Therefore, for all $0 \le l \le m \le n$, and for $0 < q \le 1$, by triangular inequality for $|\cdot|^q$, we have:
		\begin{equation*}
			\EE\left(|\Delta\overline{\kappa}_j(\gamma_{l,m}, \gamma_{m,n})|^q\right) \le C_k + C_i + 2Bd,
		\end{equation*}
		and for $q > 1$, by Minkowski's inequality, we have:
		\begin{equation*}
			\EE\left(|\Delta\overline{\kappa}_j(\gamma_{l,m}, \gamma_{m,n})|^q\right)^{1/q} \le C_k^{1/q} + C_i^{1/q} + 2Bd
		\end{equation*}
		with the convention $C_0 = C_d = 0$.
		Therefore, there exist a constant $\overline{C}$ such that for all $0 \le j \le d$, we have:
		\begin{equation}
			\EE\left(|\Delta\overline{\kappa}_j(\gamma_{l,m}, \gamma_{m,n})|^q\right) \le C.
		\end{equation} 
		Moreover, for all $1 \le j \le d$, we have $\Delta{\kappa}_j = \Delta\overline{\kappa}_j - \Delta\overline{\kappa}_{j - 1}$ so by Minkowski's inequality, we have:
		\begin{equation}
			\EE\left(|\Delta{\kappa}_j(\gamma_{l,m}, \gamma_{m,n})|^q\right) \le \max\{2, 2^q\} C.
		\end{equation} 
		Taking the sum coordinates by coordinates, and by Minkowski's inequality again, we have:
		\begin{equation}
			\EE\left(\|\Delta\dot{\kappa}(\gamma_{l,m}, \gamma_{m,n})\|^q\right) \le \max\{d, d^{q/2}\}\max\{2, 2^q\} C. \qedhere
		\end{equation} 
	\end{proof}

	\subsection{A dichotomy formula}
	
	In the present paragraph, we prove a non-probabilistic formula that allows us to view non-additive additive processes as an infinite sum of additive processes.
	
	\begin{Def}[$\Delta$-process]
		Let $(S_{m,n})_{m < n}$ be a process.
		By convention, we set $S_{n,n} = 0$ for all $n$.
		To help formulas fit in their own line, we write $\overline{S}_{n}$ as short for $S_{0,n}$ and $S_n$ as short for $S_{n, n+1}$ for all $n \ge 0$.
		To avoid having to distinguish cases, we use the convention $S_{n,m} := - S_{m,n}$ for all $m < n$ and $S_{n,n} := 0$ for all $n$.
		We use the following notations:
		\begin{itemize}
			\item For all $j \ge 1$ and for all family $(n_i)_{0 \le i \le j} \in \NN^{j + 1}$ of integers, we write $\Delta S(n_0, \dots, n_j) = S_{n_0, n_j} - \sum_{i = 0}^{j-1} S_{n_i, n_{i+1}}$.
			\item We denote by $\Delta S$ the error process associated to $S$, defined as $\Delta S_{m,n} = S_{m, n} - \sum_{j  = m}^{n-1}S_{j}$ for all $0 \le m < n$. 
			\item For all $k \ge 1$, we write $S^k$ for the process defined as $S^k_{m,n} = S_{km,kn}$ for all $0 \le m < n$.
		\end{itemize}
		We say that $S$ is additive if $\Delta S = 0$.
	\end{Def}

	Given a sequence of matrices $(\gamma_n)_{n \ge 0}$, and $S$ defined as $S_{m,n} = \kappa(\gamma_{m,n})$ for all $0 \le m < n$, we have $\Delta S_{m,n} = \Delta\kappa(\widetilde\gamma_{m,n})$.
	
	For all real number $x$, we write $\lfloor x \rfloor$ for the integer part of $x$ \ie the largest integer $n \in \ZZ$ such that $n \le x$.
	Given $r > 0$ and $x \in \RR$, we write $\lfloor x \rfloor_r = r \lfloor x / r \rfloor$.

	Before giving the proof of \eqref{sumSlogi}, let us break down the notations.
	In practice, we are only interested in the quantity $\Delta S(n_0, \dots, n_j)$ when $j \ge 2$ and $n_0 < n_1 < \dots < n_j$. 
	Note that by a telescopic sum argument, for all process $S$ and all family $(n_0, \dots, n_j)$, we have $\Delta S(n_0, \dots, n_j) = \Delta (\Delta S)(n_0, \dots, n_j)$.
	Hence, for all $0 \le m < n$, we have $\Delta S_{m,n} = \Delta S (m, m+1, m+2, \dots, n) = \Delta (\Delta S)_{m,n}$.
	However, we do not have $\Delta (S^k) = (\Delta S)^k$ for $k \ge 2$, since we are not interested in the process $(\Delta S)^k$, we write $\Delta S^k$ for $\Delta (S^k)$, that way for all $m < n$, we have $\Delta S^{k}_{m,n} = S^k_{m,n} - \sum_{j = m}^{n-1} S^k_j = S_{km, k_n} - \sum_{j = m}^{n-1}  S_{kj, kj+k} = \Delta S (km, km+k, \dots, kn)$.
	More generally, for all $(n_i)_{1 \le i \le j}$ and for all $k$, we have:
	\begin{equation}\label{delta-k}
		\Delta S^k(n_0, \dots, n_j) = \Delta S(kn_0, kn_1, \dots, kn_j).
	\end{equation}
	
	In fact it is enough to know the data of $(\Delta S (l,m,n))_{0 \le l < m < n}$ to compute the values of $\Delta S(n_0, \dots, n_j)$ for all $(n_i)_{1 \le i \le j}$, using the formula:
	\begin{equation}\label{triangulation}
		\forall j \ge 2, \forall (n_i) \in \NN^{j+1}, \Delta S (n_0, \dots, n_j) = \Delta S (n_0, \dots, n_{j-1}) + \Delta S (n_0, n_{j-1}, n_j).
	\end{equation} 
	This is a particular case of the fact that for all $1 \le k < j$ and all $(n_i) \in \NN^{j+1}$, we have:
	\begin{equation}\label{mid-trig}
		\Delta S (n_0, \dots, n_j) = \Delta S (n_0, \dots, n_{k}) + \Delta S (n_0, n_{k}, n_j) + \Delta S (n_{k},\dots,  n_j).
	\end{equation} 
	Let us now prove the following formula.
	By convention, we set $\sum_{k = 0}^{-1} x(k) = 0$ for all expression $x(k)$.
	
	\begin{Lem}\label{lem:dichotomy}
		Let $(S_{m,n})_{0 \le m \le n}$ be a family of elements of an Abelian group $V$.
		For all $i, n \ge 0$, we have:
		\begin{equation}
			\overline{S}_n = \sum_{k = 0}^{n-1} S_{k} + \sum_{j = 1}^{+\infty} \Delta S(0,\lfloor n \rfloor_{2^j}, \lfloor n\rfloor_{2^{j-1}}) + \sum_{k = 0}^{\lfloor n / 2^j \rfloor - 1}\Delta S (2^j k, 2^j(k+1/2), 2^j(k+1)). \label{sumSlogi}
		\end{equation}
		Moreover, for all $j > \log_2(n)$, we have $\lfloor n / 2^j \rfloor = \lfloor n \rfloor_{2^j} = 0$ so the sum may be taken over $j \in\{1, \dots, \lfloor \log_2(n)\rfloor\}$.
	\end{Lem}
	
	\begin{proof}
		First note that for all $n$, we have $\overline{S}_n = S_{0,n} = \sum_{k  = 0}^{n-1} S_n + \Delta S_{0,n}$.
		By \eqref{mid-trig} applied to $j = k+2$ and $(n_i)_i = (i)_i$, for all $k \ge 0$, we have $\Delta S_{0, 2k+2}- \Delta S_{0, 2k } = \Delta S(0, 2k, 2k+2) + \Delta S (2k, 2k +1, 2k +2)$. By induction, we have for all $n \ge 0$:
	\begin{equation*}
		\Delta{S}_{0,\lfloor n\rfloor_2} = \sum_{k = 0}^{\lfloor n / 2 \rfloor - 1}(\Delta S(0, 2k, 2k+2) + \Delta S (2k, 2k +1, 2k +2)).
	\end{equation*}
	Moreover, by \eqref{triangulation} and by induction, we have:
	\begin{equation*}
		\sum_{k = 0}^{\lfloor n / 2 \rfloor - 1}\Delta S(0, 2k, 2k+2) = \Delta S(0,2, 4, \dots, \lfloor n\rfloor_2) = \Delta S^2_{0,\lfloor n / 2\rfloor}.
	\end{equation*}
	Therefore, for all $n \ge 0$, we have:
	\begin{equation}\label{sum2}
		\Delta{S}_{0,n} = \Delta S_{0, \lfloor n\rfloor_2, n} + \sum_{k = 0}^{\lfloor n / 2 \rfloor - 1}\Delta S (2k, 2k +1, 2k +2) + \Delta S^2_{0,\lfloor n / 2\rfloor}.
	\end{equation}
	Now we have show that \eqref{sum2} holds for all process $S$.
	Therefore, it holds for $S^2$, and in fact for $S^{2^j}$ for all $j$.
	So for all $n, j \ge 0$, we have:
	\begin{multline*}
		\Delta S^{2^j}_{0,\lfloor n / 2^{j}\rfloor} = \Delta S^{2^j}(0,\lfloor\lfloor n / 2^{j}\rfloor\rfloor_2, \lfloor n / 2^{j}\rfloor) \\ + \sum_{k = 0}^{\lfloor n / 2^{j+1} \rfloor - 1}\Delta S^{2^j} (2k, 2k +1, 2k +2) + \Delta S^{2^{j+1}}_{0,\lfloor n / 2^{j+1}\rfloor}.
	\end{multline*}
	Now we multiply everything by $2^k$, following \eqref{delta-k} and we get:
	\begin{multline}\label{sumj}
		\Delta S^{2^j}_{0,\lfloor n / 2^{j}\rfloor} = \Delta S(0,\lfloor n \rfloor_{2^{j+1}}, \lfloor n \rfloor_{2^j})\\ +\sum_{k = 0}^{\lfloor n / 2^{j+1} \rfloor - 1}\Delta S(2^{j}2k, 2^{j}(2k +1), 2^j(2k +2)) + \Delta S^{2^{j+1}}_{0,\lfloor n / 2^{j+1}\rfloor}.
	\end{multline}
	Note that for all $0\le n < 2^{j}$, we have $\Delta S^{2^{j}}_{0,\lfloor n / 2^{j}\rfloor} = \Delta S^{2^{j}}_{0,0} = 0$ and for $2^{j} \le n < 2^{j+1}$, we have $\Delta S^{2^{j}}_{0,\lfloor n / 2^{j}\rfloor} = \Delta S^{2^{j}}_{0,1} = 0$.
	By a telescopic argument, for all $n \ge 0$, we have:
	\begin{equation}
		\Delta S_{0,n} = \sum_{j = 0}^{\lfloor \log_2(n)\rfloor - 1} (\Delta S^{2^j}_{0,\lfloor n / 2^{j}\rfloor} - \Delta S^{2^{j+1}}_{0,\lfloor n / 2^{j+1}\rfloor}) = \sum_{j = 1}^{\lfloor \log_2(n)\rfloor} (\Delta S^{2^{j-1}}_{0,\lfloor n / 2^{j-1}\rfloor} - \Delta S^{2^{j}}_{0,\lfloor n / 2^{j}\rfloor}).
	\end{equation}
	By replacing $j$ by $j-1$ in \eqref{sumj}, we have:
	\begin{equation*}
		\Delta S^{2^{j-1}}_{0,\lfloor n / 2^{j-1}\rfloor} - \Delta S^{2^{j}}_{0,\lfloor n / 2^{j}\rfloor} = \Delta S(0,\lfloor n \rfloor_{2^{j}}, \lfloor n \rfloor_{2^{j-1}}) + \sum_{k = 0}^{\lfloor n / 2^{j} \rfloor - 1}\Delta S(2^{j}k, 2^{j}(k +1/2), 2^j(k +1)).
	\end{equation*}
	Therefore, for all $n$, we have:
	\begin{equation}
		\Delta S_{0,n} = \sum_{j = 1}^{\lfloor \log_2(n)\rfloor} \Delta S(0,\lfloor n \rfloor_{2^{j}}, \lfloor n \rfloor_{2^{j-1}}) + \sum_{k = 0}^{\lfloor n / 2^{j} \rfloor - 1}\Delta S(2^{j}k, 2^{j}(k +1/2), 2^j(k +1)).
	\end{equation}
	We conclude using the fact that $\overline{S}_n = \sum_{k  = 0}^{n-1} S_n + \Delta S_{0,n}$ and that the terms of index $j > \log_2(n)$ are all zero.
\end{proof}


\subsection{Weak law of large numbers}

In the present paragraph, we prove the weak law of large numbers for almost additive processes in $\mathrm{L}^q$ for $0 < q < 2$.	
Let us fix $V$ to be a real second countable Hilbert space, endowed with a scalar product $\langle \cdot, \cdot\rangle$ and the associated norm $\|\cdot\|$.
In the present article, we only use the results for $V = \RR^d$ but the proofs works the same in infinite dimension.
We remind that second countable Hilbert spaces are all isometric to $\ell^2(\NN) = \{(x_i)_{i \ge 0}, \sum x_i^2 < + \infty\}$.
We do not need to consider larger Hilbert spaces because all measurable events can be expressed in term of a second countable quotient.

\begin{Def}[Fractional variance]
	Let $1 \le q \le 2$ and let $x \in V$ be a random variable.
	We write:
	\begin{equation}
		\mathrm{Var}_q(x) := \EE(\|x - \EE_0(x)\|^q).
	\end{equation}  
	That way $\mathrm{Var}_q(x) = + \infty$ when $x$ has infinite moment of order $q$.
\end{Def}

With our notations, $\mathrm{Var}_2$ is the classical variance. 
We remind that the variance of a sum of independent random variables is equal to the sum of their variances.
For $1 \le q < 2$, this equality becomes an inequality.

\begin{Lem}\label{lem:var-q}
	Let $x,y \in V$ be two random variables and let $1 \le q \le 2$.
	By Minkowski's inequality, we have $\mathrm{Var}_q(x+y)^{1/q} \le \mathrm{Var}_q(x)^{1/q} + \mathrm{Var}_q(y)^{1/q}$.
	If moreover $x$ and $y$ are independent, then $\mathrm{Var}_q(x+y) \le \mathrm{Var}_q(x) + \mathrm{Var}_q(y)$.
\end{Lem}

We have two extremal case: for $q = 1$, the equality stays true without the independence assumption and for $q = 2$ the inequality is in fact an equality.
The proof is based on the following functional inequality:
\begin{equation}\label{qge1}
	\forall 1 < q \le 2, \forall x, y \in V,\; \|x + y\|^q \le \|x\|^q + \|y\|^q + q \|y\|^{q-2} \langle y, x \rangle,
\end{equation}
where $\langle\cdot, \cdot\rangle$ denotes the scalar product associated to the norm $\|\cdot\|$ on $V$.
To prove \eqref{qge1}, we consider a decomposition $x = x_1 y + x_2$, with $x_1 \in \RR$ and $\langle y, x_2 \rangle = 0$.
Note that \eqref{qge1} is trivial when $y = 0$ and homogeneous, so we may assume that $\|y\| = 1$.
That way, we have $\|x + y\|^2 = (x_1 + 1)^2 + \|x_2\|^2$ and \eqref{qge1} reduces to:
\begin{equation*}
	(\|x\|^2 + 1 + 2 x_1)^{q/2} \le (\|x\|^2)^{q/2} + q x_1 + 1,
\end{equation*}
which follows from the concavity of the map $t \mapsto t^{q/2}$ on $[0, + \infty)$.

From \eqref{qge1}, we deduce that $\EE(\|x + y\|^q) \le \EE(\|x\|^q) + \EE(\|y\|^q)$ when $x$ is centred and independent of $y$, which proves Lemma \ref{lem:var-q}. 
As a consequence, we know that given an i.i.d. sequence $(x_n)_{n \ge 0}$ in a Hilbert space $V$, for all $1 \le q \le 2$, we have:
\begin{equation}
	\mathrm{Var}_q(\overline{x}_n) \le n \mathrm{Var}_q(x_0).
\end{equation}
When $q = 2$ this is an equality and when $q < 2$, we moreover have $\mathrm{Var}_q(\overline{x}_n) /n \to 0$ as we will see in the next Lemma.

We remind that for all $x,y \in V$, by concavity of the map $t \mapsto t^q$ on $[0, + \infty)$, we have:
\begin{equation}\label{qle1}
	\forall 0 < q \le 1,\, \|x + y\|^q \le \|x\|^q + \|y\|^q.
\end{equation}
From this, we deduce that $\EE(\|x + y\|^q) \le \EE(\|x\|^q) + \EE(\|y\|^q)$ for any pair of random variables $x$ and $y$.

Given a random variable $x$ that is not centred, we do not necessarily have $\mathrm{Var}_q(x) \le \EE(\|x\|^q)$ but by a naive application of Minkowski's inequality, we always have $\mathrm{Var}_q(x) \le \EE(\|x\|^q) +  \|\EE(x)\|^q \le 2\EE(\|x\|^q)$.
In fact the infimum of $\EE(\|x - b\|^q)$, which is sometimes used as $q$-variance proxy, is reached for a value of $b$ called $q$-barycentre of the law of $x$. 
The $2$-barycentre is the usual mean, the $1$-barycentre is the median, which is not necessarily 

\begin{Lem}[Weak law of large numbers]\label{lem:ldev-poly}
	Let $(x_n)_{n \ge 0}$ be i.i.d. random variables in a Hilbert space $V$. 
	Let $q \in (0,2)$ and assume that $\EE(\|x_0\|^q) < + \infty$.
	Let $b = \EE(x_0)$ if $q \ge 1$ and $0$ otherwise.
	Then $\EE(\|\overline{x}_n - n b\|^q) / n \to 0$ or in other words, $(\overline{x}_n - n b)/n^{1/q}$ converges to $0$ in probability and in $\mathrm{L}^q$.
\end{Lem}

\begin{proof}
	Up to a translation, we may always assume that $b = 0$.
	We want to show, using the dominated convergence Theorem, that:
	\begin{equation}\label{shcb}
		\EE(\|\overline{x}_n\|^q) / n = \int_{t = 0}^{+ \infty} \PP(\|\overline{x}_n\| > (nt)^{1/q}) dt \underset{n \to \infty}{\longrightarrow} 0.
	\end{equation}
	For that, we use the following decomposition.
	For all $k \in\NN$ and for all $s \ge 0$, we write $y^s_k = x_k \mathds{1}_{\|x_k\|^q \le s}$.
	Note that for all $n$, and for all $s$, if $\|x_k\| \le s^{1/q}$ for all $k \le n$, then $\overline{x}_n =  \overline{y}^{s}_n$.
	Therefore, for all $n \ge 1$ and all $t > 0$, we have:
	\begin{equation}\label{dvnjin}
		\PP(\|\overline{x}_n\| > (nt)^{1/q}) \le \PP\left(\max_{0 \le k < n} \|x_k\| > (nt)^{1/q}\right) + \PP(\|\overline{y}^{nt}_n\| > (nt)^{1/q}).
	\end{equation}
	To control the first term, note that
	\begin{equation*}
		\PP\left(\max_{0 \le k < n}\|x_k\| > (nt)^{1/q}\right) \le \sum_{k = 0}^{n-1}\PP\left(\|x_k\| > (nt)^{1/q}\right) \le n \PP(\|x_0\|^q > nt)
	\end{equation*}
	Therefore, we have:
	\begin{align*}
		\int_{0}^{+\infty}\PP\left(\max_{0 \le k < n}\|x_k\| > (nt)^{1/q}\right) dt & \le \int_{0}^{+\infty}\min\{1, n \PP(\|x_0\|^q > nt)\} dt \\
		& \le \int_{0}^{+\infty}\min\{1 / n, \PP(\|x_0\|^q > t)\} dt
	\end{align*}
	Moreover $\int_{0}^{+\infty}\PP(\|x_0\|^q > t)dt$ is finite and $1/n \to 0$ so by dominated convergence, we have:
	\begin{equation}\label{svgykbignkc}
		\int_{0}^{+ \infty}\PP\left(\exists k < n,\; \|x_k\| > (nt)^{1/q}\right) dt \underset{n \to + \infty}{\longrightarrow} 0.
	\end{equation}
	
	Let us now bound the second term.
	The random variables $(y^{nt}_k)_{0 \le k \le n}$ are i.i.d. and $\overline{y}_n^{nt} = \sum_{k = 0}^{n-1} y^{nt}_k$ so we have:
	\begin{gather*}
		\EE\left(\overline{y}_n^{nt}\right) = n \EE\left(y_0^{nt}\right) \quad \text{and} \quad
		\mathrm{Var}\left(\overline{y}^{nt}_n\right) = n \mathrm{Var}\left(y^{nt}_0\right)
	\end{gather*}
	For all $s > 0$, write $b_{s} := \|\EE({y}^{s}_0)\|$ and $v_{s} := \mathrm{Var}(\overline{y}^{s}_0)$.
	Then, by triangular  inequality for $\|\cdot\|$ and by Chebyshev's inequality, for all $n,t$, we have:
	\begin{equation}\label{chebyjhqsvk}
		\PP(\|\overline{y}^{nt}_n\| > (nt)^{1/q}) \le ((nt)^{1/q} -  n b_{nt})_+^{-2} n v_{nt}.
	\end{equation}
	Here we use the notation $(s)_+ := \max\{s,0\}$ and $0^{-1} = + \infty$.
	For $q < 1$, we use the estimate:
	\begin{equation}\notag
		b_{nt} \le \EE(\|y^{nt}_0\|) \le \int_{0}^{nt^{1/q}} \PP(\|x_0\| > s) ds
	\end{equation}
	For $q \ge 1$, we rely on the fact that $\EE(x_0) = 0$ to use the estimate:
	\begin{equation*}
		b_{nt} \le \EE(\|x_0 - y^{nt}_0\|) = \int_{nt^{1/q}}^{+ \infty} \PP(\|x_0\| > s) ds 
	\end{equation*}
	By integration by parts, for $q > 1$, we have:
	\begin{align*}
		\int_0^{+ \infty}  \left(\int_{u^{1 / q}}^{+ \infty} \PP(\|x_0\| > s) ds\right) u^{- 1/ q} du & = \int_0^{+ \infty} \int_0^{s^q} u^{- 1/ q} du  \PP(\|x_0\| > s) ds\\
		& = \int_0^{+ \infty} \frac{s^{q - 1}}{1 - 1 / q}  \PP(\|x_0\| > s) ds \\
		& = \frac{1}{q - 1} \EE(\|x_0\|^q).
	\end{align*}
	For $q < 1$, we have
	\begin{align*}
		\int_0^{+ \infty}  \left(\int_{0}^{u^{1 / q}} \PP(\|x_0\| > s) ds\right) u^{- 1/ q} du & = \int_0^{+ \infty} \int_{s^q}^{+\infty} u^{- 1/ q} du  \PP(\|x_0\| > s) ds\\
		& = \int_0^{+ \infty} \frac{s^{q - 1}}{1 / q - 1}  \PP(\|x_0\| > s) ds \\
		& = \frac{1}{1 - q} \EE(\|x_0\|^q).
	\end{align*}
	In both cases, we have:
	\begin{equation}
		\int_0^{+ \infty} b_u u^{- 1/ q} du \le \frac{1}{|q - 1|} \EE(\|x_0\|^q)
	\end{equation}
	and therefore we have $\lim_{u \to + \infty} b_u u^{1- 1 /q} = 0$.
	Finally, for $q = 1$, we have $\lim_{s \to  + \infty} b_s = \|\EE(x_0)\| = 0$.
	Therefore, there exists a constant $C \ge 1$ such that for all $u \ge C$, we have $u b_{u} \le u^{1/q} / 2$.
	So for all $t \ge C$ and for all $n \ge 1$, we have $n b_{nt} \le (nt)^{1/ q} / 2$.
	For all values of $q < 2$, we use the estimate:
	\begin{equation*}
		v_{u}  \le \EE(\|y_0^u\|^2) =  \int_{0}^{u} 2s \PP(\|x_0\| > s) ds.
	\end{equation*}
	Therefore, for all $t \ge C$ and for all $n \ge 1$, we have:
	\begin{equation}
		((nt)^{1/q} -  n b_{nt})^{-2} n v_{nt} \le 8(nt)^{-2/q}\int_{0}^{(nt)^{1/q}} s \PP(\|x_0\| > s) ds.
	\end{equation}
	We integrate in $t$ and for all $n \ge 1$ we have:
	\begin{align*}
		\int_{C}^{+ \infty}((nt)^{1/q} -  n b_{nt})^{-2} n v_{nt} dt & \le 8 \int_{0}^{+ \infty}(nt)^{-2/q}\int_{0}^{(nt)^{1/q}} s \PP(\|x_0\| > s) ds dt \\
		& \le 8\int_{0}^{\infty} s  \PP(\|x_0\| > s) \int_{s^q / n}^{\infty} (nt)^{-2/q} dt ds \\
		& \le 8\int_{0}^{\infty} \frac{s^{q-1}  \PP(\|x_0\| > s)}{(1-2/q)n} ds\\
		& \le \frac{8}{(q - 2)n}\EE(\|x_0\|^q).
	\end{align*}
	Moreover for all $t > 0$, we eventually have $nt \ge C$ and therefore, $((nt)^{1/q} -  n b_{nt})_+^{-2} n v_{nt} \to 0$ for all $t$ so we have:
	\begin{equation*}
		\int_{0}^{C}\min\{1,((nt)^{1/q} -  n b_{nt})^{-2} n v_{nt}\} dt \to 0,
	\end{equation*}
	by dominated convergence.
	Then by \eqref{chebyjhqsvk}, we have:
	\begin{equation}\label{skdbvj}
		\int_{0}^{+ \infty}\PP(\|\overline{y}^{n,t}_n\| > (nt)^{1/q}) dt \le \int_{0}^{+ \infty}\min\{1,((nt)^{1/q} -  n b_{nt})_+^{-2} n v_{nt}\} dt \to 0.
	\end{equation}
	And by \eqref{dvnjin}, we have
	\begin{equation*}
		\EE(\|\overline{x}_n\|^q) / n \le \int_{0}^{+ \infty}\PP(\|\overline{y}^{n,t}_n\| > (nt)^{1/q}) dt  + \int_{0}^{+ \infty}\PP(\max_{0 \le k < n}\|x_k\| > (nt)^{1/q}) dt 
	\end{equation*}
	And by \eqref{svgykbignkc} an \eqref{skdbvj}, both terms have limit $0$.
\end{proof}

\begin{Th}[Weak law of large numbers without drift for almost additive processes]\label{almost-wlln-no-drift}
	Let $V$ be a real Hilbert space.
	Let $0 < q < 1$ and $C_q \ge 0$.
	Let $(\mathrm{S}_{m,n})_{0 \le m \le n}$ be a $V$-valued mixing process that is $C_q$-almost-additive in $\mathrm{L}^q$.
	Assume also that $\EE(\|S_0\|^q) < + \infty$.
	Then we have:
	\begin{equation*}
		\lim_n \EE(\|\overline{S}_{n}\|^q)/ n = 0.
	\end{equation*}
	Moreover, for all $n \ge 1$, we have:
	\begin{equation}\label{dom-q}
		\EE(\|\overline{S}_{n}\|^q) \le n \left(\EE(\|S_0\|^q) + C_q \sum_{j \ge 1} 2^{-j}\right) + \lfloor\log_2(n)\rfloor C_q.
	\end{equation}
\end{Th}

\begin{proof}
	By Lemma \ref{lem:dichotomy}, we have for all $n \ge 1$:
	\begin{equation*}
		\overline{S}_n  = \sum_{k = 0}^{n-1}S_k + \sum_{j = 1}^{\lfloor\log_2(n)\rfloor} \Delta S(0, \lfloor n \rfloor_{2^j}, \lfloor n \rfloor_{2^{j-1}}) + \sum_{k = 0}^{\lfloor n/2^j\rfloor -1} \Delta S(2^j k, 2^j(k+1/2), 2^j(k+1)).
	\end{equation*}
	By \eqref{qle1}, we have:
	\begin{multline*}
		\EE(\|\overline{S}_{n}\|^q) \le \sum_{k = 0}^{n-1}\EE(\|S_k\|^q) + \sum_{j = 1}^{\lfloor\log_2(n)\rfloor} \EE(\|\Delta S(0, \lfloor n \rfloor_{2^j}, \lfloor n \rfloor_{2^{j-1}})\|)^q \\ + \sum_{k = 0}^{\lfloor n/2^j\rfloor -1} \EE(\|\Delta S(2^j k, 2^j(k+1/2), 2^j(k+1))\|^q).
	\end{multline*}
	Moreover, we have $\EE(\|\Delta(l,m,n)\|^q) \le C_q$ for all $l \le m \le n$ so we have:
	\begin{equation*}
		\EE(\|\overline{S}_{n}\|^q) \le \sum_{k = 0}^{n-1}\EE(\|S_k\|^q) + \sum_{j = 1}^{\lfloor\log_2(n)\rfloor} C_q + \lfloor n/2^j\rfloor C_q,
	\end{equation*}
	which proves \eqref{dom-q}.
	By Lemma \ref{lem:dichotomy} and \eqref{qle1} again, we have:
	\begin{multline*}
		\EE(\|\overline{S}_{n}\|^q) \le \EE\left(\left\|\sum_{k = 0}^{n-1}S_k\right\|^q\right) + \sum_{j = 1}^{\lfloor\log_2(n)\rfloor} \EE\left(\left\|\sum_{k = 0}^{\lfloor n/2^j\rfloor -1} \Delta S(2^j k, 2^j(k+1/2), 2^j(k+1))\right\|^q\right) \\+ \lfloor \log_2(n)\rfloor C_q.
	\end{multline*}
	For all $j \ge 0$, we have:
	\begin{equation*}
		K^q_j : =\EE\left(\left\|\sum_{k = 0}^{\lfloor n/2^j\rfloor -1} \Delta S(2^j k, 2^j(k+1/2), 2^j(k+1))\right\|^q\right) \le \lfloor n/2^j\rfloor C_q \le n 2^{-j}C_q
	\end{equation*}
	and $K^q_j / n \to 0$ by Lemma \ref{lem:ldev-poly}.
	Therefore, by dominated convergence, we have:
	\begin{equation*}
		\forall n \ge 0,\; \sum_{j  = 1}^\infty K^q_j / n \le \sum_{j  = 1}^\infty 2^{-j}C_q = C_q \quad \text{and} \quad \lim_{n\to \infty}\sum_{j  = 1}^\infty K^q_j / n = 0.
	\end{equation*}
	By Lemma \ref{lem:ldev-poly} again, we have $\EE\left(\left\|\sum_{k = 0}^{n-1}S_k\right\|^q\right)/ n \to 0$. Therefore:
	\begin{equation*}
		\EE(\|\overline{S}_{n}\|^q)/n \le \EE\left(\left\|\sum_{k = 0}^{n-1}S_k\right\|^q\right)/n + \sum_{j = 1}^{\lfloor\log_2(n)\rfloor} K^q_j / n + \lfloor \log_2(n)\rfloor C_q/n \to 0. \qedhere
	\end{equation*}
\end{proof}

\begin{Th}[Weak law of large numbers with drift for almost additive processes]\label{th:almost-wlln}
	Let $V$ be a real Hilbert space.
	Let $1 \le q < 2$ and $C_q \ge 0$.
	Let $(\mathrm{S}_{m,n})_{0 \le m \le n}$ be a $V$-valued mixing process that is $C_q$-almost-additive in $\mathrm{L}^q$.
	Assume also that $\EE(\|S_0\|^q) < + \infty$.
	Let
	\begin{equation}\label{def-b}
		b := \lim_n \frac{\EE(S_{0,n})}{n}.
	\end{equation}
	Then, $\|b\| \le \EE(S_0)+ C_q^{1/q}$ and we have:
	\begin{equation*}
		\lim_n \mathrm{Var}_q (\overline{S}_{n}) / n = 0.
	\end{equation*}
	Moreover, for $q \ge 1$, and for all $n \ge 1$, we have:
	\begin{equation}\label{dom-var-q}
		\mathrm{Var}_q (\overline{S}_{n})^{1/q} \le n^{1/q} \left(\mathrm{Var}_q(S_0)^{1/q} + V_q^{1/q} \sum_{j \ge 1} 2^{-j/q}\right) + \lfloor\log_2(n)\rfloor V_q^{1/q}.
	\end{equation}
	with $V_q = \max_{k \le l \le m} \mathrm{Var}_q(\Delta S(k,l,m)) \le 2C_q$.
\end{Th}

\begin{proof}
	By Lemma \ref{lem:dichotomy}, we have for all $n \ge 1$:
	\begin{equation*}
		\overline{S}_n  = \sum_{k = 0}^{n-1}S_k + \sum_{j = 1}^{\lfloor\log_2(n)\rfloor} \Delta S(0, \lfloor n \rfloor_{2^j}, \lfloor n \rfloor_{2^{j-1}}) + \sum_{k = 0}^{\lfloor n/2^j\rfloor -1} \Delta S(2^j k, 2^j(k+1/2), 2^j(k+1)).
	\end{equation*}
	If we take the expectation, using the fact that $S$ is time-invariant, we get:
	\begin{equation*}
		\EE(\overline{S}_n) = n \EE(S_0) + \sum_{j = 1}^{\lfloor\log_2(n)\rfloor} \EE(\Delta S(0, \lfloor n \rfloor_{2^j}, \lfloor n \rfloor_{2^{j-1}})) + \lfloor n/2^j\rfloor \EE(\Delta S(0, 2^{j-1}, 2^j)).
	\end{equation*}
	Let $b := \EE(S_0) + \sum_{j \ge 1}2^{-j}\EE(\Delta S(0, 2^{j-1}, 2^j))$. 
	We remind that we have $\|\EE(\Delta S(l,m,n))\|\le C_q^{1/q}$ for all $l \le m \le n$ and that $n/2^j - \lfloor n/2^j\rfloor \le 1$ for all $j \le \log_2(n)$ and $\sum_{j > \log_2(n)} 2^{-j} n \le 2$ for all $n$. 
	Then, by triangular inequality, we have $\|b\| \le \EE(S_0)+ C_q^{1/q}$ and for all $n \ge 1$, we have:
	\begin{equation*}
		\|\EE(\overline{S}_n) - nb \| \le 2C_q^{1/q} \lfloor\log_2(n) + 1\rfloor,
	\end{equation*}
	which proves \eqref{def-b}.
	
	By the mixing assumption, the sequence $(S_k)_k$ is i.i.d. and therefore, we have:
	\begin{equation*}
		\forall n\ge 0, \; \mathrm{Var}_q\left(\sum_{k = 0}^{n-1}S_k\right) \le n \mathrm{Var}_q(S_0) \quad \text{and} \quad \lim_n \mathrm{Var}_q\left(\sum_{k = 0}^{n-1}S_k\right)/ n = 0.
	\end{equation*}
	Moreover, for all $j$, and for all $k$ the quantity $\Delta S(2^j k, 2^j(k+1/2), 2^j(k+1))$ is given by a function of $(S_{l,m})_{2^j k \le l \le m \le 2^{j}(k+1)}$ so the sequence $(\Delta S(2^j k, 2^j(k+1/2), 2^j(k+1)))_{k \ge 0}$ is i.i.d. for all $j \ge 0$.
	Therefore, we have:
	\begin{equation*}
		\mathrm{Var}_q\left(\sum_{k = 0}^{\lfloor n/2^j\rfloor -1} \Delta S(2^j k, 2^j(k+1/2), 2^j(k+1))\right) \le \lfloor n/2^j\rfloor \mathrm{Var}_q(\Delta S(0, 2^{j-1}, 2^j)),
	\end{equation*}
	for all $n \ge 0$ and:
	\begin{equation*}
		\lim_n \mathrm{Var}_q\left(\sum_{k = 0}^{\lfloor n/2^j\rfloor -1} \Delta S(2^j k, 2^j(k+1/2), 2^j(k+1))\right) / n = 0.
	\end{equation*}
	Moreover, by Minkowski's inequality, we have:
	\begin{multline}\label{minko}
		\mathrm{Var}_q (\overline{S}_{n})^{1/q} \le \mathrm{Var}_q\left(\sum_{k = 0}^{n-1}S_k\right)^{1/q} + \sum_{j = 1}^{\lfloor \log_2(n)\rfloor} \mathrm{Var}_q(\Delta S(0, \lfloor n \rfloor_{2^j}, \lfloor n \rfloor_{2^{j-1}}))^{1/q} \\ 
		+ \sum_{j = 1}^{+ \infty}\mathrm{Var}_q\left(\sum_{k = 0}^{\lfloor n/2^j\rfloor -1} \Delta S(2^j k, 2^j(k+1/2), 2^j(k+1))\right)^{1/q}.
	\end{multline}
	Therefore, we have
	\begin{multline*}
		\mathrm{Var}_q (\overline{S}_{n})^{1/q} \le n^{1/q} \mathrm{Var}_q(S_0)^{1/q} + \sum_{j = 1}^{\lfloor \log_2(n)\rfloor} \mathrm{Var}_q(\Delta S(0, \lfloor n \rfloor_{2^j}, \lfloor n \rfloor_{2^{j-1}}))^{1/q} \\ 
		+ \sum_{j = 1}^{+ \infty} \lfloor n/2^j\rfloor^{1/q} \mathrm{Var}_q(\Delta S(0, 2^{j-1}, 2^j))^{1/q},
	\end{multline*}
	which proves \eqref{dom-var-q}.
	Moreover, by \eqref{minko} and by dominated convergence, we have:
	\begin{equation*}
		\mathrm{Var}_q (\overline{S}_{n})^{1/q}/n^{1/q}  \to 0. \qedhere
	\end{equation*}
\end{proof}

\begin{proof}[Proof of Theorem \ref{th:wlln}]
	Let $\nu$ be strongly irreducible and proximal and let $(\gamma_n)\sim \nu^{\otimes\NN}$.
	Let $0 < q <2$ and assume that $\EE(N(\gamma_0)^{q / 2}) < \infty$. 
	By Lemma \ref{lem:ellq-almost}, the process $(\Delta\kappa(\widetilde{\gamma}_{m,n}))_{0 \le m \le n}$ is $\mathrm{L}^q$-almost-additive.
	Therefore,, by Theorem \ref{th:almost-wlln}, for $b = - \delta(\nu)$ when $q \ge 1$ and $b = 0$ when $q < 1$, we have:
	\begin{equation*}
		\lim_{n \to  \infty}\EE\left(\left|\Delta\kappa(\widetilde\gamma_{0,n}) +  nb\right|^q\right) / n = 0. \qedhere
	\end{equation*}
\end{proof}

\begin{proof}[Proof of Theorem \ref{th:wlln-high}]
	Let $\nu$ be strongly irreducible and proximal and let $(\gamma_n)\sim \nu^{\otimes\NN}$.
	Let $0 < q <2$ and assume that $\EE(N(\gamma_0)^{q / 2}) < \infty$. 
	By Lemma \ref{lem:ellq-higher}, the process $(\Delta\dot\kappa(\widetilde{\gamma}_{m,n}))_{0 \le m \le n}$ is $\mathrm{L}^q$-almost-additive.
	Therefore, by Theorem \ref{th:almost-wlln}, when $q > 1$, the limit:
	\begin{equation*}
		b = \lim_n \frac{- \EE(\Delta\dot\kappa(\widetilde{\gamma}_{0,n}))}{n}
	\end{equation*}
	is well defined and taking $b = 0$ when $q < 1$, we have:
	\begin{equation*}
		\lim_{n \to  \infty}\EE\left(\left\|\Delta\dot\kappa(\widetilde\gamma_{0,n}) +  nb\right\|^q\right) / n = 0. \qedhere
	\end{equation*}		
\end{proof}


	\subsection{Central limit Theorem}
	
	In this section, $V$ denotes a Euclidean vector space (\ie $V \simeq \RR^d$ endowed with the usual scalar product) on which we will consider additive mixing processes.
	We do not denote this Euclidean space by $E$ or $\RR^d$ to avoid confusion with the space on which random matrices act.
	In practice, we want to see $V$ as the target space of the map $(\kappa_i)_{i \in \Theta}$ defined on $\mathrm{End}(E)$ for a given Euclidean, Hermitian or ultra-metric space $E$ and a subset $\Theta \subset \{1, \dots,\dim(E)\}$. 
	Contrary to $E$ that may be Hermitian or ultra-metric, $V$ is always a real Hilbert space.
	
	Let $x$ a random variable that takes values in $V$.
	Assume that $x$ is $\mathrm{L}^2$, \ie $\EE(\|x\|^2) < + \infty$.
	We write $\mathrm{Cov}(x)$ for the covariance matrix of $x$.
	Given an orthonormal basis $(e_i)_i$ of $V$, and $x = \sum x_i e_i$, the covariance matrix of $x$ is given by the formula $\mathrm{Cov}(x)_{i,j} = \mathrm{Cov}(x_i, x_j) = \EE((x_i - \EE(x_i))(x_j -\EE(x_j)))$.
	 
	That way $\mathrm{Cov}(x)$ is a symmetric matrix in the sense that $\mathrm{Cov}(x)_{i,j} = \mathrm{Cov}(x)_{j,i}$ for all $i,j$ and it is positive in the sense that $\mathrm{Cov}(x)_{i,i} \ge 0$ for all $i$.
	This property in invariant under a change of orthonormal basis. 
	We denote by $\mathrm{Sym}^2_+(V)$ the space of symmetric positive matrices.
	We remind that any element $g \in \mathrm{Sym}^2_+(V)$ can be written as a diagonal matrix $\mathrm{diag}(a_i)_i$ with non-negative coefficients in an orthonormal basis $(e_i)_i$, of $V$, its square root, denoted by $\sqrt{g}$, is associated to the matrix $\mathrm{diag}(\sqrt{a_i})_i$ in the same basis and is also in $\mathrm{Sym}^2_+(V)$.
	One cane easily check that in any basis, we have $g_{i,j} = \sum_k \sqrt{g}_{i,k}\sqrt{g}_{k,j}$.
	
	In the present section, we endow the space of matrices in $V$ with the norm $\|g\|^2 = \sum_{i,j} |g_{i,j}|^2$.
	This norm is sub-multiplicative and does not depend on the choice of pair of orthonormal bases, however it is not equal to the operator norm.
	
	It is clear that given two independent random variables $x,y$, we have $\mathrm{Cov}(x+y) = \mathrm{Cov}(x) + \mathrm{Cov}(y)$.
	Without assuming the independence of $x$ and $y$, we have:
	\begin{equation}\label{cov-var}
		\left\|\sqrt{\mathrm{Cov}(x+y)} - \sqrt{\mathrm{Cov}(x)}\right\|^2 \le \mathrm{Var}(y).
	\end{equation}
	In other words, the map $x \mapsto \sqrt{\mathrm{Cov}(x)}$ is contracting in $\mathrm{L}^2$.
	
	We remind that every $\mathrm{L}^2$ random variable $x$ shares its covariance matrix with a unique centred Gaussian distribution denoted by $\mathcal{N}_{\mathrm{Cov}(x)}$.
	Moreover, we have $N_{\mathrm{Cov}(x)} = \sqrt{\mathrm{Cov}(x)}_* \mathcal{N}^{\otimes I}$, where $\mathcal{N}$ denotes the centred Gaussian distribution ov variance $1$ on $\RR$ and $\sqrt{\mathrm{Cov}(x)}* \mathcal{N}^{\otimes I}$ the law of $\sum_{i,j} e_i\sqrt{\mathrm{Cov}(x)}_{i,j}x_j$ with $(x_i)_{i \in I} \sim \mathcal{N}^{\otimes I}$.
	
	Given two probability distributions $\mu$ and $\eta$, we write $\mathcal{W}_2(\mu, \eta)$ for the minimum of $\sqrt{\EE(\|x-y\|^2)}$ for all couplings $x \sim \nu$ and $y \sim \eta$.
	That way, for all $a, a' \in \mathrm{Sym}_+^2(V)$, we have $\mathcal{W}_2(\mathcal{N}_a, \mathcal{N}_{a'}) = \|\sqrt{a} - \sqrt{a'}\|$.
	Moreover for all random $x \in V$, the law $\mathcal{N}_{\mathrm{Cov}(x)}$ is the orthogonal projection of the law of $x$ onto the space of Gaussian measures, which yields \eqref{cov-var}.
	 
	The Central Limit Theorem tell us that given an i.i.d. sequence $(x_n)_{n \ge 0}$ such that $\mathrm{Var}(x_0) < + \infty$, the distribution of $\frac{\overline{x}_n - n\EE(x_0)}{\sqrt{n}}$ converges in distribution and in the quadratic Wasserstein topology, to the centred Gaussian distribution of covariance matrix $\mathrm{Cov}(x_0)$.
	In other words, there exists, on the same probability space: a random i.i.d. sequence $(x_n) \sim \mu^{\otimes\NN}$ and a sequence $(y_n)$ such that $y_n \sim \mathcal{N}_a$ for all $n$ and:
	\begin{equation}\label{eq:clt}
		\lim_n \EE\left(\left\| \frac{\overline{x}_n - nb}{\sqrt{n}} - y_n\right\|^2\right) \to 0.
	\end{equation}

	\begin{Th}[Central Limit Theorem for almost additive processes]\label{th:almost-clt}
		Let $C_2 \ge 0$ and let $(S_{m,n})_{0 \le m \le n}$ be a mixing process in a Hilbert space $V$ that is $C_2$-almost additive in $\mathrm{L}^2$. 
		Assume that $\EE(\|S_{0}\|^2) < + \infty$.
		Then, for all $n \ge 1$, we have:
		\begin{equation}\label{dom-var}
			\sqrt{\mathrm{Var}(\overline{S}_n)} \le \sqrt{n \mathrm{Var}(S_{0})} + \sqrt{C_2} \left(\log_2(n) + \frac{\sqrt{n}}{\sqrt{2} - 1}\right)
		\end{equation} 
		and the quantities:
		\begin{gather*}
			b =\lim_n \frac{\EE(\overline{S}_{n})}{n}, \\ 
			a = \lim_n \frac{\mathrm{Cov}(\overline{S}_{n})}{n} 
		\end{gather*}
		are both well defined in $V$ and $\mathrm{Sym}^2_+(V)$ respectively.
		Moreover there exists a coupling of $(S_{m,n})_{0 \le m _le n}$ with sequence $(y_n)$ such that $y_n \sim \mathcal{N}_a$ for all $n$ and:
		\begin{equation*}
			\lim_n \EE\left(\left\| \frac{\overline{x}_n - nb}{\sqrt{n}} - y_n\right\|^q\right) \to 0.
		\end{equation*}
	\end{Th}
	
	\begin{proof}
		We know that $b$ is well defined by Theorem \ref{th:almost-wlln}.
		First we prove that $a$ is well defined.
		We write $C_2 = C_q^{2/q}$, that way $(S_{m,n})_{0 \le m \le n}$ is $C_2$-almost additive in $\mathrm{L}^2$
		By \eqref{cov-var}, we have for all $0 \le m \le n$:
		\begin{equation}
			\left\|\sqrt{\mathrm{Cov}({S}_{0,n})} - \sqrt{\mathrm{Cov}({S}_{0,m} + S_{m,n})}\right\| \le  \sqrt{C_2}
		\end{equation}
		Moreover, by independence, we have:
		\begin{equation}
			\mathrm{Cov}({S}_{0,m} + S_{m,n}) = \mathrm{Cov}({S}_{0,m}) + \mathrm{Cov}(S_{m,n}).
		\end{equation}
		Therefore, for all $i \ge 1$, we have:
		\begin{equation*}
			\left\|\sqrt{2 \mathrm{Cov}({S}_{0,2^{i}})} - \sqrt{\mathrm{Cov}({S}_{0,2^{i+1}})}\right\| \le \sqrt{C_2}
		\end{equation*}
		Therefore, the limit $a := \lim 2^{-i}\mathrm{Cov}({S}_{0,2^{i}})$ is well defined and for all $n \in \NN$, we have:
		\begin{equation}\label{qsdkg}
			\sqrt{\mathrm{Var}({S}_{0,2^{i}})} \le 2^{i/2} \sqrt{\mathrm{Var}(S_0)} + \frac{2^{i / 2} - 1}{\sqrt{2}-1} \sqrt{C_2}
		\end{equation}
		and:
		\begin{equation}\label{sqdbkgh}
			\left\|\sqrt{\mathrm{Cov}({S}_{0,2^{i}})/2^i} - \sqrt{a}\right\| \le \frac{2^{- i / 2}}{\sqrt{2}-1} \sqrt{C_2}.
		\end{equation}
		
		Let $n$ be an arbitrary integer \ie not necessarily a power of $2$. Then we have :
		\begin{equation*}
			\overline{S}_{n} = \sum_{i = 0}^{\lfloor\log_{2}(n)\rfloor} {S}_{\lfloor n\rfloor_{2^{i + 1}}, \lfloor n\rfloor_{2^{i}}} + \Delta S(0, \lfloor n\rfloor_{2^{i + 1}}, \lfloor n\rfloor_{2^{i}}).
		\end{equation*}
		By independence and time-invariance, we have:
		\begin{equation*}
			\mathrm{Cov}\left(\sum_{i = 0}^{\lfloor\log_{2}(n)\rfloor} {S}_{\lfloor n\rfloor_{2^{i + 1}}, \lfloor n\rfloor_{2^{i}}}\right) = \sum_{i = 0}^{\lfloor\log_{2}(n)\rfloor} 2^{-i}\mathrm{Cov}({S}_{0,2^{i}})(\lfloor n\rfloor_{2^{i}} - \lfloor n\rfloor_{2^{i+1}}).
		\end{equation*}
		Therefore, for all $n$, we have:
		\begin{equation*}
			\left\|\mathrm{Cov}\left(\sum_{i = 0}^{\lfloor\log_{2}(n)\rfloor} {S}_{\lfloor n\rfloor_{2^{i + 1}}, \lfloor n\rfloor_{2^{i}}}\right) - n a\right\| \le \sum_{i = 0}^{\lfloor\log_{2}(n)\rfloor} (\lfloor n\rfloor_{2^{i}} - \lfloor n\rfloor_{2^{i+1}})\left\|{\mathrm{Cov}({S}_{0,2^{i}})/2^i} - {a}\right\|
		\end{equation*}
		By \eqref{sqdbkgh} and by sub-multiplicativity of the norm, there exists a constant $K$ such that $\|{\mathrm{Cov}({S}_{0,2^{i}})/2^i} - {a}\| \le 2^{-i/2} K$ for all $i$ and $(\lfloor n\rfloor_{2^{i}} - \lfloor n\rfloor_{2^{i+1}}) \le 2^i$ for all $n, i$. Therefore, we have:
		\begin{equation*}
			\left\|\mathrm{Cov}\left(\sum_{i = 0}^{\lfloor\log_{2}(n)\rfloor} {S}_{\lfloor n\rfloor_{2^{i + 1}}, \lfloor n\rfloor_{2^{i}}}\right) - n a\right\| \le \sum_{i = 0}^{\lfloor\log_{2}(n)\rfloor} K 2^{i/2} \le K \frac{\sqrt{2n} - 1}{\sqrt{2} - 1}.
		\end{equation*}
		Moreover, by \eqref{cov-var}, we have:
		\begin{equation*}
			\left\|\sqrt{\mathrm{Cov}\left(\sum_{i = 0}^{\lfloor\log_{2}(n)\rfloor} {S}_{\lfloor n\rfloor_{2^{i + 1}}, \lfloor n\rfloor_{2^{i}}}\right)} - \sqrt{\mathrm{Cov}(\overline{S}_{n})}\right\| \le \sqrt{C_2} \lfloor\log_{2}(n)\rfloor.
		\end{equation*}
		Then by triangular inequality, we have $\mathrm{Cov}(\overline{S}_{n}) / n \to a$.4
		
		To bound the variance, we use Lemma \ref{lem:dichotomy} and the triangular inequality for the $\mathrm{L}^2$ norm to get:
		\begin{multline*}
			\sqrt{\mathrm{Var}(\overline{S}_n)} \le \sqrt{\mathrm{Var}\left(\sum_{k = 0}^{n-1} S_k\right)} + \sum_{j = 1}^{\lfloor\log_2(n)\rfloor} \sqrt{\mathrm{Var}(\Delta S(0,\lfloor n \rfloor_{2^j}, \lfloor n\rfloor_{2^{j-1}}))} \\
			+ \sum_{j = 1}^{\lfloor\log_2(n)\rfloor} \sqrt{\mathrm{Var}\left(\sum_{k = 0}^{\lfloor n / 2^j \rfloor - 1}\Delta S (2^j k, 2^j(k+1/2), 2^j(k+1))\right)}
		\end{multline*}		
		By $C_2$-almost additivity, we have $\mathrm{Var}(\Delta S(k,l,m)) \le C_2$ for all $k \le l \le m$ and by independence, we have:
		\begin{gather*}
			\mathrm{Var}\left(\sum_{k = 0}^{n-1} S_k\right) = n \mathrm{Var}(S_0)\quad\text{and}\\
			\mathrm{Var}\left(\sum_{k = 0}^{\lfloor n / 2^j \rfloor - 1}\Delta S (2^j k, 2^j(k+1/2), 2^j(k+1))\right) = \lfloor n / 2^j \rfloor \mathrm{Var}(\Delta S (0, 2^{j-1}, 2^j))\le \frac{nC_2}{2^j}.
		\end{gather*}
		To get \eqref{dom-var}, we use the fact that $\sum_{j = 1}^{+\infty} \sqrt{2^{-j}} = 1/(\sqrt{2} - 1)$.
		
		Let un now prove the central limit Theorem. Let $n$ and $i$ be arbitrary integers, by Lemma \ref{lem:dichotomy} applied to $S^i$, we have:
		\begin{equation*}
			\overline{S}_{\lfloor n\rfloor_{2^i}} = \sum_{k = 0}^{\lfloor n/2^i\rfloor-1} S^i_{k} + \sum_{j = i}^{\lfloor \log_2(n)\rfloor} \Delta S(0,\lfloor n \rfloor_{2^j}, \lfloor n\rfloor_{2^{j-1}}) + \sum_{k = 0}^{\lfloor n / 2^j \rfloor - 1}\Delta S (2^j k, 2^j(k+1/2), 2^j(k+1)).
		\end{equation*}
		Let $(y^i_n)_{n \ge 0}$ be a random sequence such that $y^i_n \sim \mathcal{N}_{\mathrm{Cov}(S^i_0)}$ for all $n$ and:
		\begin{equation*}
			\EE\left(\left\| \sum_{k = 0}^{n} (S^i_{k} - \EE(S^i_0)) /\sqrt{n} - y_n \right\|^2\right) \to 0.
		\end{equation*}
		For all $j\ge 1$, we write $b_j  = \EE(\Delta S (0, 2^{j-1}, 2^j))$. 
		By independence, we have for all $j,n$:
		\begin{equation*}
			\mathrm{Var}\left({\sum_{k = 0}^{\lfloor n / 2^j \rfloor - 1}\Delta S (2^j k, 2^j(k+1/2), 2^j(k+1))}\right) = \lfloor n / 2^j \rfloor \mathrm{Var}(\Delta S (0, 2^{j-1}, 2^j)) \le 2^{-j}n C_2.
		\end{equation*}
		Then by triangular inequality for the $\mathrm{L}^2$ norm, we have:
		\begin{equation*}
			\sqrt{\EE\left(\left\|\overline{S}_{\lfloor n\rfloor_{2^i}} - \sum_{k = 0}^{\lfloor n/2^i\rfloor-1} S^i_{k} - \sum_{j = i}^\infty \lfloor n / 2^j \rfloor b_j\right\|^2\right)} \le \sum_{j = i}^\infty \sqrt{2^{-j}n C_2} + \log_2(n) \sqrt{C_2}.
		\end{equation*}
		Moreover, we have $b = \sum_{j = i}^\infty 2^{-j} b_j + \EE(S^i_0) / 2^i$.
		Therefore, we have:
		\begin{equation*}
			\limsup_{n \to  \infty}\sqrt{\EE\left(\left\| (\overline{S}_n - nb) /\sqrt{n} - 2^{-i / 2}y_{\lfloor n / 2^i\rfloor} \right\|^2\right)} \le \frac{2^{-i/2}\sqrt{C_2}}{1 - 1/\sqrt{2}}.
		\end{equation*}
		For all $n$, write $\eta_n$ for the law of $(\overline{S}_n - nb) /\sqrt{n}$ we have:
		\begin{equation*}
			\limsup_{n \to  \infty}\mathcal{W}_2(\eta_n, \mathcal{N}_{\mathrm{Cov}(S^i_0)/ 2^i}) \le \frac{2^{-i/2}\sqrt{C_2}}{1 - 1/\sqrt{2}}.
		\end{equation*}
		and by \eqref{sqdbkgh}, we have $\mathcal{W}_2(\mathcal{N}_a, \mathcal{N}_{\mathrm{Cov}(S^i_0)/ 2^i}) \le  \frac{2^{- i / 2}}{\sqrt{2}-1} \sqrt{C_2}$ so by triangular inequality for the Wasserstein distance, we have:
		\begin{equation*}
			\limsup_{n \to  \infty}\mathcal{W}_2(\eta_n, \mathcal{N}_{a}) \le \frac{2^{-i/2}\sqrt{C_2}}{1 - 1/\sqrt{2}} + \frac{2^{- i / 2}}{\sqrt{2}-1} \sqrt{C_2}.
		\end{equation*}
		Since this holds for all $i$, we have $\lim_{n \to  \infty}\mathcal{W}_2(\eta_n, \mathcal{N}_{a}) = 0$, which concludes the proof.
	\end{proof}

	Let us now apply the Central limit Theorem for almost additive processes to products of random matrices.
	
	\begin{proof}[Proof of Theorem \ref{th:clt}]
		Let $\nu$ be strongly irreducible and proximal and let $(\gamma_n)\sim \nu^{\otimes\NN}$.
		Assume that $\EE(N(\gamma_0)) < \infty$. 
		For all $0 \le m \le n$, write $S_{m,n} = \kappa(\gamma_{m,n})$.
		The process $(S_{m,n})$ is $\mathrm{L}^2$-almost additive by Lemma \ref{lem:ellq-almost}.
		Moreover $S_0 = \kappa(\gamma_0)$ so $\EE(\|S_0\|^2) < + \infty$ by assumption.
		Then, by Theorem \ref{th:almost-clt}, up to taking a coupling, there exist a sequence $(y_n)$ of identically distributed centred Gaussian random variables, such that:
		\begin{equation*}
			\lim_{n \to  \infty}\EE\left(\left|\frac{\Delta\kappa(\widetilde\gamma_{0,n}) - n \lambda_1(\nu)}{\sqrt{n}} - y_n\right|^2\right)= 0. \qedhere
		\end{equation*}
	\end{proof}
	
	\begin{proof}[Proof of Theorem \ref{th:clt-Delta}]
		Let $\nu$ be strongly irreducible and proximal and let $(\gamma_n)\sim \nu^{\otimes\NN}$.
		Assume that $\EE(N(\gamma_0)) < \infty$.
		For all $0 \le m \le n$, write $S_{m,n} = \Delta\kappa(\widetilde\gamma_{m,n})$.
		The process $(S_{m,n})$ is $\mathrm{L}^2$-almost additive by Lemma \ref{lem:ellq-almost}.
		Moreover, we have $S_0 = 0$ so $\EE(\|S_0\|^2) = 0 < + \infty$.
		Then, by Theorem \ref{th:almost-clt}, up to taking a coupling, there exist a sequence $(y_n)$ of identically distributed centred Gaussian random variables, such that:
		\begin{equation*}
			\lim_{n \to  \infty}\EE\left(\left|\frac{\Delta\kappa(\widetilde\gamma_{0,n}) - n \delta(\nu)}{\sqrt{n}} - y_n\right|^2\right)= 0. \qedhere
		\end{equation*}
	\end{proof}
	
	\begin{proof}[Proof of Theorem \ref{th:clt-high}]
		Let $\nu$ be totally irreducible and let $(\gamma_n)\sim \nu^{\otimes\NN}$.
		Assume that $\EE(N(\gamma_0)) < \infty$. 
		For all $0 \le m \le n$, write $S_{m,n} = \Delta\dot\kappa(\widetilde\gamma_{m,n})$.
		The process $(S_{m,n})$ is $\mathrm{L}^2$-almost additive by Lemma \ref{lem:ellq-almost}.
		Moreover, we have $S_0 = 0$ so $\EE(\|S_0\|^2) = 0 < + \infty$.
		Then, by Theorem \ref{th:almost-clt}, up to taking a coupling, there exist a sequence $(y_n)$ of identically distributed centred Gaussian random variables in $\RR^d$, such that:
		\begin{equation*}
			\lim_{n \to  \infty}\EE\left(\left|\frac{\Delta\kappa(\widetilde\gamma_{0,n}) - n \delta(\nu)}{\sqrt{n}} - y_n\right|^2\right)= 0. \qedhere
		\end{equation*}
	\end{proof}

	
	\section{On the Generalized Central limit Theorem}
	
	In the present section, we detail how Theorems \ref{th:wlln} and \ref{th:clt-Delta} imply Theorem \ref{th:GCLT-1}, from which Theorem \ref{th:loi-stable} follows. 
	With the same argument, Theorems \ref{th:wlln-high} and \ref{th:clt-high} imply Theorem \ref{th:GCLT-Higher}.
	
	Let us state all the facts about the domain of attraction of stable laws that we will need to use in the proof of Theorem \ref{th:loi-stable}, Theorem \ref{th:GCLT-1} and Theorem \ref{th:GCLT-Higher}.
	
	\begin{Lem}[Black box, Lévi, Feller~\cite{Feller_68}]\label{lem:stable-law}
		Let $d \ge 1$ and let $V = \RR^d$.
		Let $0 < \alpha \le 2$ and let $\mathcal{L}$ be a non-degenerate $\alpha$-stable law on $V$.
		Let $\mu$ be a probability distribution on $V$ and let $(x_n)_{n \ge 0} \sim \mu^{\otimes\NN}$.
		Let $(a_n)_{n \ge 0} \in \RR_{\ge 0}$ and $(b_n)_{n \ge 0} \in V^\NN$ be non random sequences such that the distribution of $(\overline{x}_n - b_n)/a_n$ converges in the weak $*$ topology to $\mathcal{L}$.
		Then the following assertions hold:
		\begin{enumerate}
			\item For all $q < \alpha$, $\mu$ has a finite moment of order $q$, \ie $\EE\|x_0\|^q < + \infty$.\label{lq}
			\item Up to a coupling, there exist a sequence of random variables $(y_n)$ such that for all $n$, we have $y_n \sim \mathcal{L}$ and for all $q < \alpha$, we have: \label{wass}
			\begin{equation*}
				\lim_{n \to \infty} \EE\left(\left\|\frac{\overline{x}_n - b_n}{a_n} - y_n\right\|^q\right) = 0.
			\end{equation*}
			\item For all $q > \alpha$, we have $n^{1/q} / a_n \to 0$ and for $\alpha = 2$, if we moreover assume that $\EE(\|x_0\|^2) = + \infty$, then $\sqrt{n} / a_n \to 0$.\label{lim-an}
		\end{enumerate}
	\end{Lem}
	
	Point \eqref{wass} is not explicitly mentioned in \cite{Feller_68}.
	Saying that there exists an identically distributed sequence $(y_n)$ such that $\frac{\overline{x}_n - b_n}{a_n} - y_n$ converges in probability to $0$ is equivalent to saying that the law of $\frac{\overline{x}_n - b_n}{a_n}$ converges in the weak-$*$ topology to the law of the $y_n$'s.
	The fact that this convergence also holds in $\mathrm{L}^q$ for all $q < \alpha$ is equivalent to saying that moreover, the moment of order $q$ of $\frac{\overline{x}_n - b_n}{a_n}$ converges to the moment of order $q$ of the $y_n$'s.
	This is direct when one proves the convergence via order statistics as detailed in \cite{series}.
	
	Via Fourier transform, we need to use the fact that for all $q < \alpha$, the function $\widehat{\mu}$ is uniformly $\mathcal{C}^q$ is the sense that there exists a constant $C$ such that for all $\Theta, \eps \in V^*$, we have:
	\begin{equation*}
		\left|\widehat\mu(\theta + \eps) - \widehat\mu(\theta) - \sum_{k = 1}^{\lfloor q \rfloor} \eps^{\otimes k}\partial^k\widehat\mu(\theta)/k! \right| \le C\|\eps\|^q.
	\end{equation*}
	Therefore, the Fourier transform of the law of $\frac{\overline{x}_n - b_n}{a_n}$ converges in the Sobolev space $W^{q, \infty}$ to the Fourier transform of $\mathcal{L}$.
	This implies that the law of $\frac{\overline{x}_n - b_n}{a_n}$ converges to $\mathcal{L}$ in the Wasserstein $\mathrm{L}^{q'}$ topology for all $q' < q$ and therefore for all $q' < \alpha$.
	
	Let un now prove Theorem \ref{th:GCLT-1} and Theorem \ref{th:GCLT-Higher} via the following Lemma:

	\begin{proof}[Proof of Theorem \ref{th:GCLT-1} and Theorem \ref{th:GCLT-Higher}]
		To prove Theorem \ref{th:GCLT-1}, write $\check\kappa = \kappa$ and to prove Theorem \ref{th:GCLT-Higher}, write $\check{\kappa} = \dot\kappa$.
		For all $n$, write $x_n = \check\kappa(\gamma_n)$.
		For all $n$, we have $\check\kappa(\overline{\gamma}_n) = \overline{x}_n + \Delta\check\kappa(\widetilde{\gamma}_{0,n})$.
		Let $(a_n)$ and $(b_n)$ be non-random sequences such that $(\overline{x}_n - b_n)/ a_n$ converges to the $\alpha$-stable limit $\mathcal{L}$. 
		Let $(y_n)$ be as in point \eqref{wass} in Lemma \ref{lem:stable-law}.
		For $q \ge 1$, let $b = \delta(\nu)$ and for $q < 1$ let $b = 0$.
		For all $n$, we have formally:
		\begin{equation}\notag
			\frac{\check\kappa(\overline{\gamma}_n) - b_n + nb}{a_n} = \frac{\overline{x}_n - b_n}{a_n} - y_n + \frac{\Delta\check\kappa(\widetilde{\gamma}_{0,n}) + n b}{a_n}
		\end{equation}
		Hence, for $q \le 1$, by triangular inequality and monotonicity of the expectation we have:
		\begin{equation}\notag
			\EE\left(\left\|\frac{\check\kappa(\overline{\gamma}_n) - b_n + nb}{a_n}\right\|^q\right) \le \EE\left(\left\|\frac{\overline{x}_n - b_n}{a_n} - y_n \right\|^q\right)+ \EE\left(\left\|\frac{\Delta\check\kappa(\widetilde{\gamma}_{0,n}) + n b}{a_n}\right\|^q\right)
		\end{equation}
		And for $q > 1$, by Minkowski's inequality, we have:
		\begin{equation}\notag
			\EE\left(\left\|\frac{\check\kappa(\overline{\gamma}_n) - b_n + nb}{a_n}\right\|^q\right)^{1/q} \le \EE\left(\left\|\frac{\overline{x}_n - b_n}{a_n} - y_n \right\|^q\right)^{1/q} + \EE\left(\left\|\frac{\Delta\check\kappa(\widetilde{\gamma}_{0,n}) + n b}{a_n}\right\|^q\right)^{1/q}
		\end{equation}
		For $q < 2$, by Theorem \ref{th:wlln} or Theorem \ref{th:wlln-high} applied to $(\gamma_n)$, we have:
		\begin{equation*}
			\lim_{n \to \infty} \EE\left(\left\|\frac{\Delta\check\kappa(\widetilde{\gamma}_{0,n}) + nb}{n^{1/q}}\right\|^q\right) = 0.
		\end{equation*}
		For $q = 2$, by Theorem \ref{th:clt-Delta} or Theorem \ref{th:clt-high} applied to $(\gamma_n)$, up to a coupling, there exist a sequence $y'_n$ of identically distributed Gaussian variables such that we have:
		\begin{equation*}
			\lim_{n \to \infty} \EE\left(\left\|\frac{\Delta\check\kappa(\widetilde{\gamma}_{0,n}) + nb}{\sqrt{n}} - y_n'\right\|^2\right) = 0.
		\end{equation*}
		Therefore, by triangular inequality for the $\mathrm{L}^2$ norm, we have:
		\begin{equation*}
			\lim_{n \to \infty} \EE\left(\left\|\frac{\Delta\check\kappa(\widetilde{\gamma}_{0,n}) + nb}{\sqrt{n}}\right\|^2\right) = \EE(\|y'_0\|^2) < + \infty.
		\end{equation*}
		By \eqref{lim-an} in Lemma \ref{lem:stable-law}, we have $n^{1/q} / a_n \to 0$.
		So, in both cases, we have:
		\begin{equation*}
			\lim_{n \to \infty} \EE\left(\left\|\frac{\Delta\check\kappa(\widetilde{\gamma}_{0,n}) + nb}{a_n}\right\|^q\right) = 0.
		\end{equation*}
		Moreover, by \eqref{wass} in Lemma \ref{lem:stable-law}, we have:
		\begin{equation}\notag
			\lim_{n \to \infty} \EE\left(\left\|\frac{\overline{x}_n - b_n}{a_n} - y_n\right\|^q\right) = 0.
		\end{equation}
		Therefore:
		\begin{equation*}
			\EE\left(\left\|\frac{\check\kappa(\overline{\gamma}_n) - b_n + nb}{a_n}\right\|^q\right)^{\min\{1/q, 1\}} = 0.
		\end{equation*}
		We use the fact that $\min\{1/q, 1\} > 0$ to conclude.
	\end{proof}
	
	Here, we have in fact proven the following intermediate Lemma and applied it to $\Delta\check\kappa\widetilde\gamma$.
	\begin{Lem}
		Let $S$ be a mixing process that takes values in a finite dimensional real vector space $V$.
		Let $0 < \alpha \le 2$ and assume that there exists $C \ge 0$ and $q \in (\alpha, + \infty) \cup \{2\}$ such that $S$ is $C$-almost additive in $\mathrm{L}^q$.
		Assume that the law of $S_0$ is in the domain of attraction of an $\alpha$-stable law $\mathcal{L}$ and that $S_0$ has an infinite moment of order $2$, let $(a_n)$ and $(b_n)$ be non random sequences and let $(y_n)$ be a random sequence such that for all $n$, we have $y_n \sim \mathcal{L}$ and for all $p < \alpha$, we have:
		\begin{equation*}
			\lim_{n \to \infty} \EE\left(\left\|\frac{\sum_{k = 0}^{n-1} S_k - b_n}{a_n} - y_n\right\|^p\right) = 0.
		\end{equation*}
		Then there exists a constant $b \in V$ that can be taken arbitrarily when $n / a_n \to 0$ and is equal to $\lim_n -\EE(\Delta S_{0,n})/n$ otherwise and such that:
		\begin{equation}\label{cv-Wp}
			\lim_{n \to \infty} \EE\left(\left\|\frac{S_{0, n} - b_n + n b}{a_n} - y_n\right\|^p\right) = 0.
		\end{equation}
	\end{Lem}
	
	Let us now prove Theorem \ref{th:loi-stable}. 
	We simply need to show that the hypotheses of Theorem \ref{th:GCLT-1} hold.
	
	\begin{proof}[Proof of Theorem \ref{th:loi-stable}]
		Let $(\gamma_n) \sim \nu^{\otimes\NN}$.
		We want to apply Theorem \ref{th:GCLT-1} to $(\gamma_n)$. 
		For that we need to check that there exist $q$ such that $\alpha < q < 2$ or $\alpha = q = 2$ and $\EE(N(\gamma_0)^{q / 2}) < + \infty$.
		Remember that we assume $\nu$ to be supported on $\SL(E)$.
		Therefore, we have $\sum_{i = 1}^d\kappa_i(\gamma_0) = \log|\det(\gamma_0)| = 0$ almost surely.
		We also know that $\kappa_i(\gamma_0) \le \kappa(\gamma_0)$ for all $1 \le i \le d$.
		Therefore, we have $\kappa_d(\gamma_0) = -\sum_{i = 1}^{d-1}\kappa_i(\gamma_0)\ge -(d-1) \kappa(\gamma_0)$ for all $i$.
		Moreover, we have formally $N(\gamma_0) = \kappa_1(\gamma_0) - \kappa_d(\gamma_0) \le d \kappa(\gamma_0)$.
		By \eqref{lq} in Lemma \ref{lem:stable-law}, we have $\EE(\kappa(\gamma_0)^{q / 2}) < + \infty$ for all $q$ such that $q/2 < \alpha$.
		Let $q = \sqrt{2 \alpha}$, then we have $q/2 < \alpha$ (and therefore $\EE(N(\gamma_0)^{q / 2}) < + \infty$) and $\alpha < q < 2$ or $\alpha = q = 2$ so we may apply Theorem \ref{th:GCLT-1} to $(\gamma_n)$. 
	\end{proof}

	To truly prove the General Central Limit Theorem, we would need to show that the converse of Theorems \ref{th:loi-stable} and \ref{th:GCLT-Higher} hold, namely, given $(\gamma_n)_n$ i.i.d. that satisfies the correct algebraic assumptions and such that $\kappa(\gamma_0)$ has an infinite moment of order $2$, if we have sequences $(a_n)$ and $(b'_n)$, such that $\frac{\kappa(\overline{\gamma}_n) - b'_n}{a_n}$ converges in law to a non-degenerate limit $\mathcal{L}$, then there exists $b \ge 0$ such that $\frac{\sum_{k = 0}^{n-1}\kappa(\gamma_k) - b'_n -  nb}{a_n}$ converges to the same limit $\mathcal{L}$ that is therefore a stable distribution.
	
	To prove that we need a fine understanding of the proof of the Generalized Central Limit Theorem via ordered statistics quite outside the scope of the present article that is not really about stable laws.

	\bibliographystyle{alpha}
	\bibliography{biblio.bib}
	
\end{document}